\documentclass{alggeom}

\usepackage{enumerate}
\usepackage{amsmath, amssymb, mathrsfs}
\usepackage[all]{xy}

\newtheorem{thm}{Theorem}[section]
\newtheorem{cor}[thm]{Corollary}
\newtheorem{lem}[thm]{Lemma}
\newtheorem{prop}[thm]{Proposition}
\theoremstyle{definition}
\newtheorem{defn}[thm]{Definition}
\theoremstyle{remark}
\newtheorem{rem}[thm]{Remark}
\newtheorem{eg}[thm]{Example}
\numberwithin{equation}{section}

\newcommand{\C}{\mathbb{C}}      
\newcommand{\Z}{\mathbb{Z}}      
\newcommand{\Sp}{\operatorname{Spec}}
\newcommand{\A}{\mathbb{A}}
\newcommand{\R}{\mathbb{R}}
\newcommand{\Q}{\mathbb{Q}}
\newcommand{\Ou}{\mathcal{O}}
\newcommand{\Pj}{\mathbb{P}}
\newcommand{\E}{\mathcal{E}}
\newcommand{\F}{\mathcal{F}}
\newcommand{\G}{\mathcal{G}}
\newcommand{\K}{\mathcal{K}}
\newcommand{\fA}{\mathfrak{A}}
\newcommand{\fY}{\mathfrak{Y}}
\newcommand{\fC}{\mathfrak{C}}
\newcommand{\fX}{\mathfrak{X}}

\newcommand{\I}{\mathcal{I}}
\newcommand{\J}{\mathcal{J}}

\newcommand{\cX}{\mathcal{X}}
\newcommand{\cY}{\mathcal{Y}}
\newcommand{\cD}{\mathcal{D}}
\newcommand{\cV}{\mathcal{V}}
\newcommand{\cM}{\mathcal{M}}
\newcommand{\cN}{\mathcal{N}}
\newcommand{\cH}{\mathcal{H}}
\newcommand{\cZ}{\mathcal{Z}}
\newcommand{\cA}{\mathcal{A}}
\newcommand{\uW}{\underline{W}}
\newcommand{\uD}{\underline{D}}
\newcommand{\uY}{\underline{Y}}
\newcommand{\uX}{\underline{X}}
\newcommand{\uDelta}{\underline{\Delta}}
\newcommand{\uH}{\underline{H}}

\newcommand{\bI}{\mathbb{I}}

\newcommand{\pr}{\operatorname{pr}}
\newcommand{\tr}{\operatorname{tr}}
\newcommand{\Bl}{\operatorname{Bl}}
\newcommand{\Id}{\operatorname{Id}}
\newcommand{\coker}{\operatorname{coker}}
\newcommand{\Aut}{\operatorname{Aut}}

\newcommand{\Quot}{\operatorname{Quot}}
\newcommand{\Hilb}{\operatorname{Hilb}}

\newcommand{\Supp}{\operatorname{Supp}}
\newcommand{\Err}{\operatorname{Err}}
\newcommand{\Ext}{\operatorname{Ext}}
\newcommand{\Hom}{\operatorname{Hom}}
\newcommand{\fQuot}{\mathfrak{Quot}}
\newcommand{\fHilb}{\mathfrak{Hilb}}
\newcommand{\rar}{\rightarrow}
\newcommand{\xrar}{\xrightarrow}
\newcommand{\tf}{\mathrm{tf}}
\newcommand{\vir}{\operatorname{vir}}

\newcommand{\ch}{\widetilde{ch}}

\newcommand{\mr}{\operatorname{mr}}
\newcommand{\orb}{\mathrm{orb}}
\newcommand{\age}{\mathrm{age}}
\newcommand{\ev}{\operatorname{ev}}
\newcommand{\pt}{\operatorname{pt}}

\allowdisplaybreaks

\title{Relative Orbifold Donaldson--Thomas Theory and the Degeneration Formula}
\author{Zijun Zhou}
\email{zzhou@math.columbia.edu}
\address{Zijun Zhou, Department of Mathematics, Columbia University, New York, NY 10027, USA}

\begin{document}

\begin{abstract}

We generalize the notion of expanded degenerations and pairs for a simple degeneration or smooth pair to the case of smooth Deligne-Mumford stacks. We then define stable quotients on the classifying stacks of expanded degenerations and pairs and prove the properness of their moduli's. On 3-dimensional smooth projective DM stacks this leads to a definition of relative Donaldson-Thomas invariants and the associated degeneration formula.

\end{abstract}

\maketitle

\vspace*{6pt}\tableofcontents  



\section{Introduction}

\subsection{Background and motivation}

Let $X$ be a smooth projective 3-fold. Motivated by many as a higher-dimensional gauge theory and introduced by R. Thomas \cite{Th} as a holomorphic analogue of the Casson invariant, Donaldson--Thomas theory counts ideal sheaves of curves on $X$ in certain fixed topological classes.

A coherent sheaf $\I$ can be realized as an ideal sheaf of a 1-dimensional subscheme $Z\subset X$ if and only if it is torsion-free of rank 1 with trivial determinant, which means that DT theory can be viewed either as a sheaf counting theory or a curve counting theory. The essential connection of DT theory to other curve counting theories was first established by the work of MNOP \cite{MNOP,MNOP2}, known as the Gromov--Witten/Donaldson--Thomas correspondence. They conjectured and proved in special cases that the generating functions of GW and DT theory can be equated to each other after a change of variable. They adopted the localization technique in the toric setting and developed the theory of DT topological vertex.

GW/DT correspondence has also been proved for local curves \cite{BP,OP}. This is the first non-toric case one can actually do computations. The crucial tool is the DT degeneration formula, motivated from the degeneration formula in GW theory and developed by J. Li and B. Wu \cite{LW}. With the degeneration formula in the simple degeneration case, one can ``split" $X$ into two simpler spaces $Y_-$ and $Y_+$, and express the DT invariant of $X$ in terms of the relative theory of $Y_\pm$ with respect to the divisor $D$.

The goal of this paper is to generalize the relative Donaldson--Thomas theory and the degeneration formula to 3-dimensional smooth projective \emph{orbifolds}, which in algebraic settings, refers to Deligne--Mumford stacks. One important reason why people care about orbifolds is Ruan's crepant resolution conjecture \cite{Ruan,BG}: the GW or DT invariants of a 3-orbifold satisfying the hard Lefschetz condition should be equivalent to those of its crepant resolution (if exists), up to change of variables and analytic continuations. Orbifold GW theory and its degeneration formula have already appeared in \cite{AGV,AF}.

On the DT side, the orbifold topological vertex technique was developed in \cite{BCY}, for toric Calabi--Yau 3-orbifold.  Orbifold DT theory is also treated in \cite{GT}, for projective CY 3-orbifolds. In both cases the CY condition gives a symmetric obstruction theory and defines a Behrend's function $\nu$, and the DT invariants are defined as weighted Euler characteristic with respect to $\nu$. The GW/DT correspondence for the orbifold topological vertex with transversal $A_n$-singularities was proved in \cite{Zong,RZ,RZ2,Ross}, and the crepant resolution conjecture for DT vertex with transversal $A_n$-singularities was proved in \cite{Ross2}.

We are particularly interested in the following picture, which indicates the relationship between various theories involved, and will be pursued in the future work.
$$
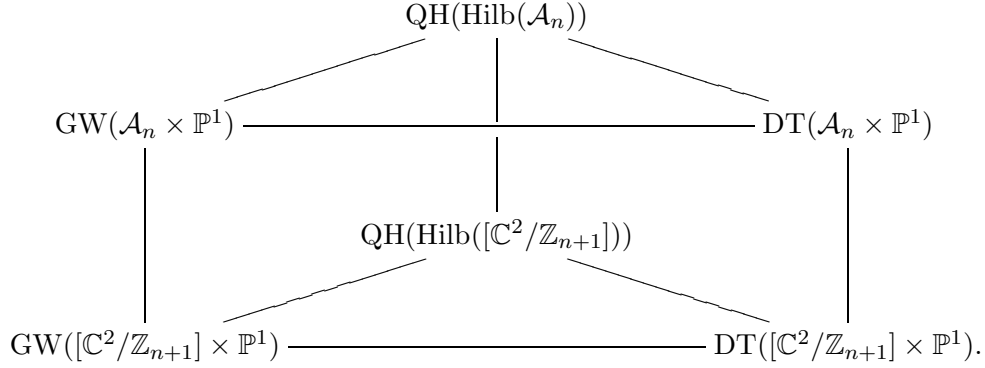
\xymatrix{
& \text{QH}(\Hilb(\cA_n)) \ar@{-}[dl] \ar@{-}[dr] \ar@{-}'[d][dd] & \\
\text{GW}(\cA_n\times \Pj^1) \ar@{-}[rr] \ar@{-}[dd] & & \text{DT}(\cA_n\times \Pj^1) \ar@{-}[dd] \\
& \text{QH}(\Hilb([\C^2/\Z_{n+1}])) \ar@{-}[dl] \ar@{-}[dr] & \\
\text{GW}([\C^2/\Z_{n+1}]\times \Pj^1) \ar@{-}[rr] & & \text{DT}([\C^2/\Z_{n+1}]\times \Pj^1).}$$
The upper triangle in the diagram is established in \cite{Mau,MO,MO2}, and the vertical lines stand for crepant transformation correspondences. We hope to establish the lower triangle to complete the whole diagram.

\subsection{Outline}

For simplicity we work over the field of complex numbers $\C$. We always use $A_*$, $A^*$ to denote appropriate cohomology and homology theory over $\Q$, which could be Chow groups, Borel-Moore, etc. K-theory will be either topological or algebraic K-theory, over $\Q$.

In this paper we consider a smooth projective Deligne--Mumford stack $W$, and define the absolute DT invariants with descendants and insertions, generalizing the DT invariants in the scheme case. For a smooth divisor $D\subset W$, we also define the relative DT invariants of $W$ with respect to $D$. We follow the approach of introducing a perfect obstruction theory in each case, \emph{without} the CY assumption.

For simplicity we assume that the divisor $D$ is \emph{connected}. The theory can be generalized easily for disconnected $D$.

We treat the case of the simple degeneration in this paper. Let $\pi: X\rar C$ be a projective family of smooth Deligne--Mumford stacks of relative dimension 3. By a simple degeneration, we roughly mean that $\pi$ is in the form of a simple normal crossing near some point $0\in C$, with singular fiber $X_0= Y_-\cup_D Y_+$ splitting into two pieces. Let $X_c$ be a smooth fiber.

Let $F_1 K(X)$ be the subgroup of $K(X)_\Q$ generated by the structure sheaves of 1-dimensional closed substacks. Given $P\in K(X)$, let $\cM^P$ be the stack parameterizing 1-dimensional closed substacks on expanded degenerations
$$X_0[k]:=Y_-\cup_D \Delta_1 \cup_D \cdots \cup_D \Delta_k \cup_D Y_+,$$
with topological datum $P$. Similarly, given $\theta_\pm\in K(Y_\pm)$, and $\theta_0\in K(\Hilb(D))$, let $\cN_\pm^{\theta_\pm, \theta_0}$ be the stack parameterizing 1-dimensional closed substacks on expanded pairs
$$Y[l]:=Y\cup_D \Delta_1 \cup_D \cdots \cup_D \Delta_l,$$
with topological datum $(\theta_\pm, \theta_0)$. Let $\cM^\theta\subset \cM^P$ be the open and closed substack associated with the relative data $\theta$. We will construct perfect obstruction theories on these stacks.

Let $i_c: \{c\} \hookrightarrow C$ be the inclusion of a point. One has the following two Cartesian diagrams,
$$\xymatrix{
\Hilb_{X_c}^P \ar[r] \ar[d] & \cM^P \ar[d] & \cM^\theta \ar@{^(->}[r]^-{\iota_\theta} & \fQuot^{\Ou_X,P}_{\fX_0/\fC_0} \ar@{^(->}[r] \ar[d] & \cM^P \ar[d] \\
\{c\} \ar@{^(->}[r] & C, & &  \{0\} \ar@{^(->}[r] & C,
}$$
where the map $\cM^P\rar C$ is the composition $\cM^P\rar \fC^P\rar C$. Our main theorem is the following.

\begin{thm}[(Degeneration formula -- cycle version)]
  \begin{equation*} 
  i_c^! [\cM^P]^{\vir} = [\Hilb_{X_c}^P]^{\vir},
  \end{equation*}
  \begin{equation*} 
  i_0^! [\cM^P]^{\vir} = \sum_{\theta\in \Lambda_P^{spl}} \iota_{\theta*} \Delta^! \left( [\cN_-^{\theta_-,\theta_0}]^{\vir} \times [\cN_+^{\theta_+,\theta_0}]^{\vir} \right),
  \end{equation*}
  where the classes in the second row are viewed in $0\times_{\A^1} \cM^P$.
\end{thm}

Here $\Delta$ is the diagonal map in the following diagram.
$$\xymatrix{
\cM^\theta & \cN_-^{\theta_-,\theta_0} \times_{\Hilb_D^{\theta_0}} \cN_+^{\theta_+,\theta_0} \ar[l]^-{\Phi_\theta}_-\cong \ar[r] \ar[d] & \cN_-^{\theta_-,\theta_0} \times \cN_+^{\theta_+,\theta_0} \ar[d] \\
 & \Hilb_D^{\theta_0} \ar[r]^-\Delta & \Hilb_D^{\theta_0} \times \Hilb_D^{\theta_0}.
 }$$
 
Using virtual classes one can define the corresponding numerical DT invariants $\left\langle \prod \tau_{k_i}(\gamma_i) \right\rangle_{X_c}^P$, $\left\langle \prod \tau_{k_i}(\gamma_{i, \pm}) \middle| C \right\rangle_{Y_\pm,D}^{\theta_\pm}$, where $\gamma_i$'s are descendent insertions and $C$ is the relative insertion. Let $\{ C_k \}$ be a basis for the cohomology of $\Hilb(D)$, and $g^{kl}$ be the inverse matrix for the Poincar\'e paring under this basis. We have the following numerical version of the degeneration formula.

\begin{thm}[(Degeneration formula -- numerical version)]
  Given $P\in F_1 K(X_c)$, assume that $\gamma_{i, \pm}$ are disjoint with $D$. We have
  $$\left\langle \prod_{i=1}^r \tau_{k_i}(\gamma_i) \right\rangle_{X_c}^P =
  \sum_{\substack{\theta_-+\theta_+ - P_0 = P, \\
                S\subset \{1,\cdots,r\}, k,l}}
  \left\langle \prod_{i\in S} \tau_{k_i}(\gamma_{i, -}) \middle| C_k \right\rangle_{Y_-,D}^{\theta_-} g^{kl} \left\langle \prod_{i\not\in S} \tau_{k_i}(\gamma_{i, +}) \middle| C_l \right\rangle_{Y_+,D}^{\theta_+}, $$
  where $\theta_\pm \in F_1 K(Y_\pm)$ range over all configurations that satisfy $\theta_-+\theta_+ - P_0 = P$.
\end{thm}

$F_1^{\mr} K(X)$ be the \emph{multi-regular} subgroup, generated by multi-regular stacky curves, i.e. those whose associated representation of the stablizer group at the generical point is a multiple of the regular representation. In the multi-regular case, consider the image $(\beta, \varepsilon)$ of $P$ in $F_1^{\mr} K(X_c)/F_0(X_c) \oplus K_0(X_c)$ as the topological datum.

\begin{thm}[(Degeneration formula -- numerical version for multi-regular case)]
  Given $\beta\in F_1^{\mr}K(X_c)/F_0 K(X_c)$, assume that $\gamma_{i,\pm}$ are disjoint with $D$. We have
  $$\left\langle \prod_{i=1}^r \tau_{k_i}(\gamma_i) \right\rangle_{X_c}^{\beta,\varepsilon} =
  \sum_{\substack{\beta_-+\beta_+=\beta, \\
                \varepsilon_-+\varepsilon_+= \varepsilon+m, \\
                S\subset \{1,\cdots,r\}, k,l}}
  \left\langle \prod_{i\in S} \tau_{k_i}(\gamma_{i,-}) \middle| C_k \right\rangle_{Y_-,D}^{\beta_-,\varepsilon_-} g^{kl} \left\langle \prod_{i\not\in S} \tau_{k_i}(\gamma_{i,+}) \middle| C_l \right\rangle_{Y_+,D}^{\beta_+,\varepsilon_+}, $$
  where $\beta_-\in F_1^{\mr}K(Y_-)/F_0 K(Y_-)$, $\beta_-\in F_1^{\mr}K(Y_+)/F_0 K(Y_+)$ range over all curve classes that coincide on $D$ and satisfy $\beta_-+\beta_+=\beta$.
\end{thm}

The theorem can be restated in the form of generating functions, obtained by summing over $\varepsilon$.

\begin{cor}
Given $\beta\in F_1^{\mr}K(X_c)/F_0 K(X_c)$, assume that $\gamma_{i,\pm}$ are disjoint with $D$. Then,
  \begin{eqnarray*}
  Z_\beta \left( X_c;q \ \middle| \ \prod_{i=1}^r \tau_{k_i}(\gamma_i) \right) &=& \sum_{\substack{\beta_- + \beta_+ =\beta \\
                            S\subset \{1,\cdots, r\},k,l }}
  \frac{g^{kl}}{q^m} Z_{\beta_-,C_k} \left( Y_-,D;q \ \middle| \ \prod_{i\in S} \tau_{k_i}(\gamma_{i,-}) \right) \\
  && \cdot Z_{\beta_+,C_l} \left( Y_+,D;q \ \middle| \ \prod_{i\not\in S} \tau_{k_i}(\gamma_{i,+}) \right).
  \end{eqnarray*}
\end{cor}

The paper is organized as follows. In Section 2 we generalize J. Li's construction of expanded degenerations and pairs to the orbifold case, and in Section 3 we discuss the concept of stable quotients. Section 3,4,5 aim to define and prove the properness of the moduli stacks of stable quotiens on the classifying stacks of expanded degenerations and pairs. Finally in Section 6,7 and 8 we specialize to the case of 3-orbifolds, construct perfect obstruction theories on such stacks and prove the degeneration formula, for both cycle version and numerical version.

\subsection{Acknowledgements}

I would like to express sincere gratitude to my advisor Chiu-Chu Melissa Liu, who led me into this exciting area and provided plenty of help and guidance. Her encouragement and patience are really important to the author. Special thanks go to Jun Li and Baosen Wu, for their preceding work and helpful discussions. The author would also like to thank Davesh Maulik, Jason Starr and Zhengyu Zong for useful discussions.

\section{Stacks of expanded degenerations and pairs}

In this section we construct the stacks of expanded degenerations and expanded pairs, which serve as the target spaces for relative and degeneration of curve-counting theories. There are several approaches to this construction. The first algebraic-geometric approach is due to J. Li \cite{Li01}, where the author explicitly constructs standard families of expanded degenerations and pairs, and then forms the stacks as limits of them. Various equivalent approaches are well summarized in the paper \cite{ACFW}, which works with algebraic stacks and proves the independence of the resulting stacks on the original target.

We will mainly adopt J. Li's explicit process. In fact, the following are just direct generalizations of his method to the case of Deligne--Mumford stacks. We will define the notions of expanded degenerations and pairs, build standard families of expanded degenerations and pairs, and then construct the general stacks using the standard models as coverings. Our basic setting is as follows.

\begin{defn}
  Let $Y$ be a separated Deligne--Mumford stack of finite type over $\C$, and $D\subset Y$ be a locally smooth connected effective divisor. By locally smooth we mean that \'etale locally near a point on $D$, the pair $(Y,D)$ is a pair of a smooth stack and a smooth divisor. We call such $(Y,D)$ a \emph{locally smooth pair}. We call it a \emph{smooth pair} if both $Y$ and $D$ are smooth.
\end{defn}

\begin{defn}
  Let $X$ be a separated Deligne--Mumford stack of finite type over $\C$, and $(C,0)$ be a smooth pointed curve over $\C$. A \emph{locally simple degeneration} is a flat morphism $\pi:X\rightarrow C$, with central fiber $X_0=Y_-\cup_D Y_+$, where $Y_\pm \subset X_0$ are closed substacks, $(Y_\pm,D)$ are locally smooth pairs, and they intersect transversally as $Y_- \cap Y_+ = D$.

  By $Y_-\cup_D Y_+$ we mean the pushout in the 2-category of algebraic stacks in the sense of \cite{AGV}. We denote the divisor $D$ in $Y_\pm$ by $D_\pm$.

  By \emph{transversal} intersection, we mean the following. For any point $p$ in the divisor $D$, there are \'etale neighborhoods $V$ of $0\in C$ and $U$ of $p\in X$, such that the following diagram is commutative, with all horizontal maps \'etale.
$$\xymatrix{
 \Sp \frac{\C[x,y,t]}{(xy-t)} \times\A^n \ar[d]^\pi & U \ar[d]^{\pi|_U} \ar[l]_-{\text{\'et}} \ar[r]^-{\text{\'et}} & X \ar[d]^\pi \\
 \Sp \C[t] & V \ar[l]_-{\text{\'et}} \ar[r]^-{\text{\'et}} & C.
 }$$
In other words, there is a common \'etale neighborhood $V$ of $\Sp \C[t]$ and $C$, and a common \'etale neighborhood $U$ of $\Sp \frac{\C[x,y,t]}{(xy-t)} \times\A^n$ and $X$.

We call such data a \emph{simple degeneration} if furthermore $(Y_\pm,D_\pm)$ are smooth pairs, and for any $c\in C$, $c\neq 0$, the fiber $X_c:=\pi^{-1}(c)$ is smooth.

We call such $U$, together the restricted family $\pi_U$, a \emph{standard local model}.
\end{defn}

Let $N_-:=N_{D/Y_-}$ and $N_+:=N_{D/Y_+}$. Check on local coordinates and one can easily find that $N_-\otimes N_+ \cong \Ou_D$.

\begin{rem}
By Artin's algebraic approximation theorem, our definition of transversality is equivalent to that in the language of formal neighborhoods, as in \cite{Li01}.
\end{rem}

\begin{rem}
In the study of degenerations, we are only interested in local behaviors around the singular divisor. For convenience, we can always replace the curve $C$ by an \'etale neighborhood around $0\in C$, and just assume $C=\A^1$.
\end{rem}

\subsection{Expanded degenerations and pairs}

For the central fiber $X_0=Y_-\cup_D Y_+$, consider the $\Pj^1$-bundle over $D$,
$$\Delta:=\Pj_D(\Ou_D\oplus N_+) \cong \Pj_D(N_- \oplus \Ou_D),$$
with two sections $D_-$ and $D_+$ corresponding to the 0-sections of the two expressions, where $N_{D_-/\Delta}\cong N_+$ and $N_{D_+/\Delta}\cong N_-$. $\Delta$ will often be called a ``bubble component".

For an integer $k\geq 0$, take $k$ copies of $\Delta$, indexed by $\Delta_1,\cdots,\Delta_k$, and insert them between $\Delta_0:=Y_-$ and $\Delta_{k+1}:=Y_+$. We have
$$X_0[k]:=Y_-\cup_D \Delta_1 \cup_D \cdots \cup_D \Delta_k \cup_D Y_+,$$
where $D_-\subset \Delta_i$ is glued to $D_+\subset \Delta_{i+1}$, $i=0,\cdots,k$.

$X_0[k]$ is a Deligne--Mumford stack of nodal singularities, with $k+2$ irreducible components. Following \cite{AF}, we call $X_0[k]$ an \emph{expanded degeneration of length $k$} with respect to $X_0$. We index the divisors in $X_0[k]$ in order by $D_0,\cdots,D_k$.

For the expanded pair, consider a locally smooth pair $(Y,D)$. In practice, the pair we mostly care about would be $(Y_\pm,D_\pm)$. Denote by $N$ the normal bundle $N_{D/Y}$. In the same manner, consider the $\Pj^1$-bundle
$$\Delta:=\Pj_D(\Ou_D\oplus N),$$
and for an integer $l\geq 0$, let
$$Y[l]:=Y\cup_D \Delta_1 \cup_D \cdots \cup_D \Delta_l,$$
where the gluing is proceeded in the same manner.

Denote by $D[l]$ the last divisor $D_-\subset \Delta_l$ (in cases where $l$ is implicit we might also use $D[\infty]$). $(Y[l],D[l])$ forms a new locally smooth pair. We call $(Y[l],D[l])$ the \emph{expanded pair of length $l$} with respect to $(Y,D)$. We index the divisors in $(Y[l],D[l])$ by $D_0,\cdots,D_{l-1},D_l=D[l]$.

\begin{rem}\label{action_fiber}
For $X_0[k]$ (resp. $Y[l]$) above, we can introduce a natural $(\C^\ast)^k$ (resp. $(\C^\ast)^l$) -action: the $i$-th factor of $(\C^\ast)^k$ (resp. $(\C^\ast)^l$) acts on $\Delta_i$ fiberwise, and trivially on $Y_\pm$ (resp. $Y$).
\end{rem}

\subsection{Expanded degeneration of standard local model}

Now we describe the process of building standard family of expanded degenerations. First look at the baby case,
$$\pi: X= \Sp \frac{\C[x,y,t]}{(xy-t)}  \rightarrow \A^1 =\Sp\C[t],$$
with central fiber $X_0= \Sp \frac{\C[x,y]}{(xy)} $ and singular divisor $D=\pt$.
where $Y_-$ is given by the equation $(y=0)$. Consider the base change
$$m:\A^2 \rightarrow \A^1,\qquad (t_0,t_1)\mapsto t_0 t_1.$$
The family becomes
$$\pi: X\times_{\A^1} \A^2 =  \Sp \frac{\C[x,y,t_0,t_1]}{(xy-t_0t_1)}  \rightarrow \A^2= \Sp \C [t_0,t_1],$$
with smooth generic fibers, and fibers over the axes $t_0t_1=0$ are isomorphic to $X_0$.

Blow up $X\times_{\A^1} \A^2$ along the singular divisor $D\times_{\A^1} 0$, and denote it by $\widetilde{X(1)}$. It is the proper transform of the degree 2 hypersurface $(xy-t_0t_1=0)\subset \A^4$ in $\Bl_0 \A^4=\Ou_{\Pj^3}(-1) \subset \A^4\times \Pj^3$. We have
$$\widetilde{X(1)} \cong  \Ou_{\Pj^1\times\Pj^1}(-1,-1), $$
where $\Pj^1\times \Pj^1 \subset \Pj^3$ is given by the Segre embedding.

The next step is to contract one factor of the $\Pj^1\times\Pj^1$. For details one can look up in \cite{Li01}, where calculations are given in explicit coordinates. The contraction map and resulting space are
$$p: \widetilde{X(1)} \rar X(1), \qquad X(1):=\Ou_{\Pj^1}(-1)^{\oplus2},$$
with the associated projection
$$\pi: X(1)\rar \A^2.$$
$X(1)$ obtained this way is a resolution of singularities of $X\times_{\A^1} \A^2$. This is the first standard expanded degeneration family for the local model $X\rightarrow \A^1$. The properties of this family will be summarized in the next subsection.

\begin{rem}
In the contraction $\Ou_{\Pj^1\times\Pj^1}(-1,-1) \rightarrow \Ou_{\Pj^1}(-1)^{\oplus2}$, the choice of the $\Pj^1$ factor to contract is not canonical. Here we adopt the convention that $x$ is the coordinate on $Y_-$ and $y$ is the coordinate on $Y_+$ respectively, and we choose to contract the first $\Pj^1$-factor. The global criterion for this choice is: the restriction of resulting family $\pi: X(1)\rightarrow \A^2$ to the coordinate line $L_0=(t_1=0)$ is a smoothing of the divisor $D_0\subset X(1)\times_{\A^2} 0 \cong X[1]$.

Different choices of contracted factors give different resolutions. Our baby case is actually the simplest example of Atiyah's flop.
\end{rem}

\begin{rem}
For those familiar with toric geometry, the relationship between $X\times_{\A^1} \A^2$, $\widetilde{X(1)}$ and $X(1)$ can be understood as follows. Let $e_i$, $1\leq i\leq 3$ be the standard basis in $\R^3$. $X\times_{\A^1} \A^2$ is the toric variety associated to the cone $\Sigma$ spanned by $\{e_1, e_2, e_3, e_1+e_3-e_2\}$. $\widetilde{X(1)}$ is obtained by adding the vector $e_1+e_3$ and refining $\Sigma$; $X(1)$ is obtained by simply adding the 2-dimensional face spanned by $e_1$ and $e_3$ to $\Sigma$. One can easily see that $X(1)$ is a resolution of $X\times_{\A^1} \A^2$, but not the unique choice.
\end{rem}

\subsection{Standard families of expanded degenerations and pairs -- gluing}

For general locally simple degenerations $\pi:X\rightarrow \A^1$, we glue the local models together.

\'Etale locally around a point $p\in D$, the family $\pi:X\rightarrow \A^1$ is of the standard form, i.e. there is an \'etale neighborhood $V_p$ of $0\in \A^1$, and a common \'etale neighborhood $p\in U_p$ of $X\times_{\A^1} V_p$ and $\left( \Sp \frac{\C[x,y,t,\vec{z}]}{(xy-t)} \right) \times_{\Sp\C[t]} V_p$, where $\vec{z}=(z_1,\cdots,z_n)$, $n=\dim D$. Denote this local family by $\pi_p: U_p \rightarrow V_p$. Let $U_p(1)$ be the restriction on $U_p$ of the standard local model defined in the previous subsection.

\'Etale locally around a point $p\not\in D$, we can find an \'etale neighborhood $U_p$ of $p$, with $U_p \cap D= \emptyset$. In this case we just define $U_p(1):=U_p\times_{\A^1} \A^2$, which is already smooth.

Now $\{U_p\mid p\in X\}$ form an \'etale covering $U=\coprod_p U_p \rightarrow X$, consisting of standard local models. We have the underlying relation
$$\xymatrix{
R:=U\times_X U \ar@<.4ex>[r] \ar@<-.4ex>[r] & U \ar[r] & X.
}$$
where we denote the two projection maps by $q_1$ and $q_2$, and the inverse map switching the two factors by $i:R\rightarrow R$.

Now we define $p:U(1)\rightarrow U$ and $\pi: U(1)\rightarrow \A^2$, just as the disjoint union of all $U_p(1)$. Let $R(1):=R\times_{q_1,U}U(1)$, then we have the relation
$$\xymatrix{
R(1) \ar@<.4ex>[r] \ar@<-.4ex>[r] & U(1),
}$$
where the upper arrow is $\pr_2:R(1)=R\times_{q_1,U}U(1) \rightarrow U(1)$, and the lower arrow is $\pr_2\circ i$.

It is easy to check that $R(1)\rightrightarrows U(1)$ satisfies the axioms of a groupoid scheme in the sense of \cite{LMB}, and the map $R(1)\rightarrow U(1)\times U(1)$ is separated and quasi-compact. Thus by Proposition 4.3.1 of \cite{LMB}, the stack-theoretic quotient $X(1):=\left[ R(1)\rightrightarrows U(1) \right]$ defines a Deligne--Mumford stack of finite type. The maps $p:U(1)\rightarrow U$, $\pi:U(1)\rightarrow \A^2$ also glue globally on the stack.

\begin{defn}
Define $X(1):=\left[ R(1)\rightrightarrows U(1) \right]$, the associated stack of the groupoid scheme. We have the projection $p:X(1)\rightarrow X$, and the family map $\pi:X(1)\rightarrow \A^2$, giving the commutative diagram
\begin{equation} \label{sq1}
\xymatrix{
X(1) \ar[r]^p \ar[d]_\pi & X \ar[d]^\pi \\
\A^2 \ar[r]^m & \A^1.
}\end{equation}
We call $\pi:X(1)\rightarrow \A^2$ the \emph{standard family of length-$1$ expanded degenerations}, with respect to $X\rightarrow \A^1$.
\end{defn}

Note that if we compose $\pi:X(1)\rightarrow \A^2$ with the second projection $\pr_1: \A^2\rar \A^1$, $(t_0,t_1)\mapsto t_1$, the resulting $\pi_1: X(1)\rar \A^1$ is still a locally simple degeneration. This can be checked on the local model in explicit coordinates.

Now from $X(k)$ we proceed by induction to construct $X(k+1)$. Suppose that we already have $\pi: X(k)\rightarrow \A^{k+1}$, and if projected to the last factor, the composite $\pi_k: X(k)\rightarrow \A^{k+1} \xrightarrow{\pr_k} \A^1$ is a locally simple degeneration. Applying the $k=1$ procedure, we obtain the 2-dimensional family
$$\pi_k(1): X(k)(1)\rightarrow \A^2,$$
and we take $X(k+1)=X(k)(1)$ to be the $(k+1)$-th space.

It remains to define the family map. Consider the projection to the first $k$ factors $\pi_k^c: X(k)\xrar{\pi} \A^{k+1} \xrightarrow{\pr_{1,\cdots,k}} \A^k$, and take the composite
$$p(k)\circ \pi_k^c: X(k+1)= X(k)(1)\rightarrow X(k)\rightarrow \A^k.$$
Combine the two maps
$$\pi=(\pi_k(1), p(k)\circ \pi_k^c): X(k+1)\rightarrow \A^k\times\A^2 \cong \A^{k+2}.$$
This is the standard length-($k+1$) expanded degeneration family.

\begin{defn}
We have constructed $X(k)$ with projection $p:X(k)\rightarrow X$ and the family map $\pi:X(k)\rightarrow \A^{k+1}$, giving the commutative diagram
\begin{equation} \label{sqk}
\xymatrix{
X(k) \ar[r]^p \ar[d]_\pi & X \ar[d]^\pi \\
\A^{k+1} \ar[r]^m & \A^1,
}\end{equation}
where the map $m:\A^{k+1} \rightarrow \A^1$ is the multiplication $(t_0,\cdots,t_k) \mapsto t_0t_1\cdots t_k$.

We call $\pi:X(k)\rightarrow \A^{k+1}$ the \emph{standard family of length-$k$ expanded degenerations}, with respect to $X\rightarrow \A^1$.
\end{defn}

Now let's construct the standard families of expanded pairs. There are two equivalent ways of doing this. Let $(Y,D)$ be a locally smooth pair.

\textbf{Approach 1.} Consider $Y(1):=\Bl_{D\times 0}(Y\times \A^1)$, the blow-up of $Y\times \A^1$ at the closed substack $D\times 0$. We have the projection $\pi: Y(1)\rar \A^1$, which is a locally simple degeneration, with central fiber $Y\cup_D \Delta$. Let $D(1)$ be the proper transform of $D\times \A^1$, which is still isomorphic to $D\times \A^1$.

Apply the $X(k)$-construction as above with respect to this locally simple degeneration, and let $Y(k):=Y(1)(k-1)_{\deg}$, where we put the subscript ``$\deg$" to indicate it's the same construction as above, rather than the ``$Y(k)$" construction we are defining for locally smooth pairs. Here we choose $Y$ to be the ``--" piece and $\Delta$ to be the ``+" piece. We have the projection $\pi: Y(k)\rar \A^k$, and contraction map $p: Y(k)\rar Y(1)\rar Y$, which is the composition of the contraction in the degeneration case and the blow-up. Let $D(k)\subset Y(k)$ be the base change of $D(1)$ in the diagram (\ref{sqk}), which is isomorphic to $D\times \A^k$ since $D(1)$ does not intersect the singular divisor in the central fiber. Then $(Y(k),D(k))$ forms a locally smooth pair.

\begin{defn}
We have constructed the family $\pi:Y(k)\rightarrow \A^k$, with contraction map $p:Y(k)\rightarrow Y$, and the locally smooth divisor $D(k)\subset Y(k)$. We call $\pi: Y(k)\rightarrow \A^k$ the \emph{standard family of length-$k$ expanded pairs}, with respect to $(Y,D)$. $D(k)\subset Y(k)$ is called the \emph{distinguished divisor}.
\end{defn}

\textbf{Approach 2} (successive blow-up construction). Again we let $Y(1):=\Bl_{D\times 0}(Y\times \A^1)$. Assume that we already have $(Y(k),D(k))$, which is a locally smooth pair. Define $Y(k+1):= \Bl_{D(k)\times 0} (Y(k)\times \A^1)$. The family map $\pi: Y(k+1)\rar \A^{k+1}$ is defined obviously. The following proposition says that these two approaches are actually equivalent.

\begin{prop} \label{succ_bl}
  Let $(Y(k),D(k))$ and $\pi: Y(k)\rar \A^k$ be defined as in Approach 1. There is an isomorphism $Y(k)\cong \Bl_{D(k-1)\times 0}(Y(k-1)\times \A^1)$, making the following diagram commutative.
  $$\xymatrix{
  Y(k) \ar[rr]^-\sim \ar[d]_\pi && \Bl_{D(k-1)\times 0}(Y(k-1)\times \A^1) \ar[d]^{(\pi, \Id_{\A^1})} \\
  \A^k \ar[rr]^-\Id && \A^{k-1} \times \A^1.
  }$$
\end{prop}

The difference of the two definitions lies in the contraction $p: Y(k)\rar Y(k-1)$. If we look at this contraction on the base, the first one is given by $(t_1,\cdots,t_k) \mapsto (t_1,\cdots,t_{k-2},t_{k-1}t_k)$ and the other is $(t_1,\cdots,t_k) \mapsto (t_1,\cdots, t_{k-1})$. In other words, the contraction $Y(k)\cong \Bl_{D(k-1)\times 0}(Y(k-1)\times \A^1) \rar Y(k-1)$ is different from the contraction $p$ in Approach 1.

The successive blow-up definition, although easier to describe and appearing more natural, appears less compatible with the $X(k)$ construction. One can see that later via the different $(\C^*)^k$-actions.

Both approaches appear in J. Li's work on degeneration theories, and both will be used during the proof of the properness of Quot-stacks.

\begin{rem}
  Before stating the properties of the standard families, we introduce the following notation. Denote by $(Y(k)^\circ, \pi^\circ)$ the ``reverse" standard family of expanded pairs. Let $Y(k)^\circ$ be the same as $Y(k)$, but $\pi^\circ:= r\circ \pi: Y(k)^\circ \rightarrow \A^k$, where $r:\A^k\rightarrow \A^k$ is the ``order-reversing" map $(t_1,\cdots,t_k)\mapsto (t_k,\cdots,t_1)$. For example, $Y(1)^\circ = \Bl_{0\times D}(\A^1\times Y)$, and for larger $k$ the construction is conducted via base change from the left.
\end{rem}

\subsection{Properties of standard families}

Let's fix some notations on $\A^{k+1}$ for later use.

\begin{enumerate}[1)]
\setlength{\parskip}{1ex}

\item There is a $(\C^\ast)^k$-action on the base $\A^{k+1}$. For $\lambda=(\lambda_1,\cdots,\lambda_k)\in (\C^\ast)^k$, $t=(t_0,\cdots,t_k) \in \A^{k+1}$, let
    \begin{equation} \label{action_Ak_X}
    \lambda \cdot t := \left( \lambda_1 t_0, \frac{\lambda_2}{\lambda_1} t_1, \cdots, \frac{\lambda_k}{\lambda_{k-1}} t_{k-1}, \lambda_k^{-1} t_k \right).
    \end{equation}
    Under this action (and trivial action on $\A^1$), and the multiplicative group homomorphism $(\C^\ast)^k \rightarrow \C^\ast$, $(\lambda_1,\cdots,\lambda_k) \mapsto \lambda_1\cdots\lambda_k$, the multiplication map $m:\A^{k+1} \rightarrow \A^1$ is equivariant.

\item For a subset $I=\{i_0,\cdots,i_l\} \subset \{0,\cdots,k\}$ with $|I|=l+1\leq k+1$ and $i_0<\cdots<i_l$. Let $\tau_I:\A^{l+1}\rightarrow \A^{k+1}$ be the embedding given by
    $$t_i=\left\{
    \begin{aligned}
    &t_{i_p},  && \text{ if $i=i_p$ for some $p$}, \\
    &1,  && \text{ if $i\neq i_p$ for any $p$ }.
    \end{aligned}\right.$$

\item Again for a subset $I\subset \{0,\cdots,k\}$. Let
    $$U_I:=\{(t_0,\cdots,t_k)\mid t_i\neq 0,\ \forall i\not\in I\},$$
    which is a product of $\A^1$'s and $\mathbb{G}_m$'s. We have the natural open immersion $\tilde{\tau}_I: U_I \rightarrow \A^{k+1}$.

    For two subsets $I,I'\subset \{0,\cdots,k\}$ as above with $|I|=|I'|$, we have the natural isomorphism given by reordering the coordinates
    $$\tilde{\tau}_{I,I'}: U_I \xrightarrow{\sim} U_{I'}.$$

\item For $0\leq i\leq k$, denote by $\pr_i:\A^{k+1}\rightarrow \A^1$ the $i$-th projection $(t_0,\cdots,t_k)\mapsto t_i$.
\end{enumerate}

Now we state the properties of standard families of expanded degenerations $\pi:X(k)\rightarrow \A^{k+1}$. They can be proved by induction.

\begin{prop}
\begin{enumerate}[1)]

\setlength{\parskip}{1ex}

\item \label{dim_Xk} $X(k)$ is a separated Deligne--Mumford stack of finite type over $\C$, of dimension $n+k+1$, where $n=\dim X_0$. Moreover, if $\pi:X\rar \A^1$ is proper, the family map $\pi:X(k)\rightarrow \A^{k+1}$ is proper; if $X$ is smooth, $X(k)$ is smooth.

\item \label{fiber_Xk} For $t=(t_0,\cdots,t_k)\in \A^{k+1}$, if $t_0 \cdots t_k\neq 0$, the fiber of the family over $t$ is isomorphic to the generic fibers $X_c,\ c\neq 0$ in the original family $\pi:X\rightarrow \A^1$.

    Let $I\subset \{0,\cdots,k\}$ be a subset with $1\leq|I|=l+1\leq k+1$. If $t_i=0$, $\forall i\in I$ and $t_j\neq 0$, $\forall j\not\in I$, then the fiber over $t$ is isomorphic to $X_0[l]$.

    In particular, the fiber over $0\in\A^{k+1}$ is isomorphic to the length-$k$ expanded degeneration $X_0[k]$.

\item \label{action_Xk} There is a $(\C^\ast)^k$-action on $X(k)$, making the diagram (\ref{sqk}) $(\C^\ast)^k$-equivariant. This action gives isomorphisms of fibers in the same strata described as above. The induced action of the stabilizer on a fiber isomorphic to $X_0[l]$  is the same as described in Remark \ref{action_fiber}.

\item \label{dis_action_Xk} There are also discrete symmetries in $X(k)$, away from the central fiber. For two subsets $I,I'\subset \{0,\cdots,k\}$ with $|I|=|I'|$, the natural isomorphisms $\tilde{\tau}_{I,I'}: U_I \xrightarrow{\sim} U_{I'}$ induce isomorphisms on the families
    $$\tilde{\tau}_{I,I',X}: X(k)\big|_{U_I} \xrightarrow{\sim} X(k)\big|_{U_{I'}},$$
    extending those given by $(\C^\ast)^k$-actions on smooth fibers.

Moreover, restricted to the embedding $\tau_I:\A^{l+1}\rightarrow \A^{k+1}$, we have $X(k)\times_{\A^{k+1},\tau_I} \A^{l+1} \cong X(l)$, and under the $(\C^\ast)^{k-l}$-``translations" $U_I\cong \A^{l+1} \times (\C^\ast)^{k-l}$, we have the identification
    $$X(k)\big|_{U_I}\cong X(l)\times (\C^\ast)^{k-l}.$$
\end{enumerate}
\end{prop}

\begin{rem}
Here $(t_0, \cdots, t_k)\in \A^{k+1}$ means a closed point $\A^{k+1}$. But all definitions and descriptions can be easily generalized to an arbitrary point $\Sp K \to \A^{k+1}$. In this case the statement $t_i\neq 0$ or $t_i=1$ means that the image of $t_i \in \C[t_0, \cdots, t_k]$ in $K$ is $\neq0$ or $=1$.
\end{rem}

\begin{rem}
  We say something about the $(\C^*)^k$-action. It's easy to define the action on local models. To glue the actions on local models together, say for $X(1)$, the key observation is that the action is actually trivial on the $X$ factor and only nontrivial on the extra $\A^1$ factor. Hence starting from an \'etale covering of $X$, one can get a natural $\C^*$-equivariant covering of $X(1)$, which easily descends to a $\C^*$-action on $X(1)$.
\end{rem}

For $\pi: Y(k)\rar \A^k$, again we need some notations on $\A^k$. Here we use $(t_1,\cdots,t_k)$ for coordinates of the base $\A^k$. We adopt almost the same notations $\tau_I$, $\tilde{\tau}_I$, and $\tilde\tau_{I,I'}$, except that the indices range from $1$ to $k$ instead of $0$ to $k$.

The only difference is the group action. The group $(\C^\ast)^k$ (instead of $(\C^*)^{k-1}$) acts on the base $\A^k$, but differently for the two different constructions. Consider $\lambda=(\lambda_1,\cdots,\lambda_k)\in (\C^\ast)^k$, $t=(t_1,\cdots,t_k) \in \A^k$.

For Approach 1, the action is
\begin{equation} \label{action_Ak_Y}
\lambda \cdot t := \left( \lambda_1 t_1, \frac{\lambda_2}{\lambda_1} t_2, \cdots, \frac{\lambda_k}{\lambda_{k-1}} t_k \right),
\end{equation}
whereas for the successive blow-up construction (Approach 2), the action is
$$\lambda \cdot t:= (\lambda_1 t_1, \lambda_2 t_2, \cdots, \lambda_k t_k).$$

The standard families of expanded pairs have the following properties, some of which follows from the properties of $X(k)$ and the other can be easily proved by induction.

\begin{prop}
\begin{enumerate}[1)]
\setlength{\parskip}{1ex}

\item $Y(k)$ is a separated Deligne--Mumford stack of finite type over $\C$, of dimension $n+k$, where $n=\dim Y$. If $Y$ is proper, then the map $\pi$ is proper. Moreover, if $(Y,D)$ is a smooth pair, so is $(Y(k),D(k))$.

    $D(k)\cong D\times\A^k$, and the restriction of $p$ and $\pi$ on $D(k)$ can be identified with the two projections to $D$ and $\A^k$ respectively.

\item \label{fiber_Yk} For $t=(t_1,\cdots,t_k)\in \A^{k}$, if $t_1 \cdots t_k\neq 0$, then the fiber of the pair over $t$ is isomorphic to the original pair $(Y,D)$.

    Let $I\subset \{1,\cdots,k\}$ be a subset with $1\leq|I|=l\leq k$. If $t_i=0$, $\forall i\in I$ and $t_j\neq 0$, $\forall j\not\in I$, then the fiber over $t$ is isomorphic to $(Y[l],D[l])$.

    In particular, the fiber over $0\in\A^k$ is isomorphic to the length-$k$ expanded pair $(Y[k],D[k])$.

\item \label{action_Yk} There is a $(\C^\ast)^k$-action on $Y(k)$, in both constructive approaches, compatible to the corresponding actions on $\A^k$. This action gives isomorphisms of fibers in the same strata as described above. In particular, the induced action of the stabilizer on the a fiber isomorphic to $Y[l]$ is the same as described in Remark \ref{action_fiber}.

\item \label{dis_action_Yk} There are also discrete symmetries in $Y(k)$, away from the central fiber. For two subsets $I,I'\subset \{1,\cdots,k\}$ with $|I|=|I'|=l$, the natural isomorphisms $\tilde{\tau}_{I,I'}: U_I \xrightarrow{\sim} U_{I'}$ induce isomorphisms on the families
    $$\tilde{\tau}_{I,I',Y}: Y(k)\big|_{U_I} \xrightarrow{\sim} Y(k)\big|_{U_{I'}},$$
    extending those given by $(\C^\ast)^k$-actions on smooth fibers.

    The discrete actions restricted on $D(k)\cong D\times \A^k$ are just a reordering of the coordinates of the $\A^k$ factors.

    Moreover, if restricted to the embedding $\tau_I:\A^{l}\rightarrow \A^{k}$ we have $Y(k)\times_{\A^{k},\tau_I} \A^{l} \cong Y(l)$, and under the $(\C^\ast)^{k-l}$-``translations" $U_I\cong \A^{l} \times (\C^\ast)^{k-l}$, we have the identification
    $$Y(k)\big|_{U_I}\cong Y(l)\times (\C^\ast)^{k-l}.$$
\end{enumerate}
\end{prop}

\begin{rem}
  For the $(\C^*)^k$-action we have one more factor than the $X(k)$ case. By construction the action on $X(k)$ gives a $(\C^*)^{k-1}$-action on $Y(k)$. The last $\C^*$ comes form the original action in $Y(1)=\Bl_{D\times 0}(Y\times \A^1)$. One can see again that an arbitrary \'etale covering of $Y$ would give us a $\C^*$-equivariant covering of $Y(1)$, and hence $(\C^*)^k$-equivariant covering of $Y(k)$, which makes the action available. In the successive blow-up construction the action is more obvious -- the blow-up process brings one $\C^*$-factor each time.
\end{rem}

Standard families of expanded degenerations and pairs are closely related to each other. Consider a locally simple degeneration $\pi: X\rar \A^1$, with central fiber $X_0\cong Y_-\cup_D Y_+$.

1) (Restriction to hyperplanes) Let $H_i\subset \A^{k+1}$ be the coordinate hyperplane defined by $t_i=0$, $0\leq i\leq k$. The restriction of the family $\pi: X(k)\rightarrow \A^{k+1}$ to $H_i$ is a ``smoothing" of all divisors except $D_i$, in the following sense.

    \begin{prop} \label{Hi_Xk}
    Composing $\pi$ with the projection $\pr_i: \A^{k+1} \rightarrow \A^1$ to the factor $t_i$, we obtain $\pi_i: X(k)\rightarrow \A^1$. Then $\pi_i$ is a locally simple degeneration.

    Denote the singular divisor by $D_i(k)$, and the central fiber decomposition by $X(k)\times_{\pi_i,\A^1} 0=: X(k)_-^i \cup_{D_i(k)} X(k)_+^i$. Then one has
    $$D_i(k)\cong \A^i \times D\times \A^{k-i},\quad X(k)_-^i\cong Y_-(i) \times \A^{k-i}, \quad X(k)_+^i \cong \A^i \times Y_+(k-i)^\circ.$$
    Moreover, the restriction of the map $\pi:X(k)\rightarrow \A^{k+1}$ on these two components are given in the following diagrams
    $$\xymatrix{
     & X(k)_-^i \ar[dl]_{\pi_{0,\cdots,i-1}} \ar[d]_{\pi_i} \ar[dr]^{\Id_{\A^{k-i}}} & & & X(k)_+^i \ar[dl]_{\Id_{\A^i}} \ar[d]_{\pi_i} \ar[dr]^{\pi_{i+1,\cdots,k}^\circ} & \\
     \A^i & 0 & \A^{k-i} & \A^i & 0 & \A^{k-i}
     }$$
    where the map from $Y_-(i)$, $Y_+(k-i)^\circ$ to the corresponding bases are just the maps of standard families of expanded pairs.

    The gluing is along $D_i(k) \cong D_-(i)\times \A^{k-i} \cong \A^i\times D \times \A^{k-i} \cong \A^i \times D_+(k-i)^\circ$, where the order of copies of $\A^1$ here reflects the map $\pi|_D$; and we have $D_i(k)\cap D_j(k) =\emptyset$ for $i\neq j$.

    All statements above are compatible with the corresponding $(\C^*)^k$-actions.
    \end{prop}

2) (Restriction to lines) Let $L_i\subset \A^{k+1}$ be the coordinate line corresponding to $t_i$, $0\leq i\leq k$. The restriction of the family $\pi: X(k)\rightarrow \A^{k+1}$ to $L_i$ is a ``smoothing" of the divisor $D_i$, in the following sense.

    \begin{prop} \label{Li_Xk}
    $\pi|_{L_i}: X(k) \big|_{L_i}\rar L_i\cong \A^1$ is a locally simple degeneration, with central fiber decomposition $X_0[k]= Y_-[i]\cup_{D_i} Y_+[k-i]$, and generic fiber isomorphic to $X_0[k-1]$.

    Moreover, we have a description of the total space:

    $X(k) \big|_{L_0}$ is $Y_-(1)\cup_{D_-(1)} (\A^1\times Y_+[k-1])$;

    $X(k) \big|_{L_i}$ for $1\leq i\leq k-1$ is
    $$Y_-[i](1)\cup_{D_-[i](1)} (\A^1\times Y_+[k-i-1]) \ \cong \ (Y_-[i-1]\times \A^1) \cup_{D_-[i-1]\times \A^1} Y_+[k-i](1)^\circ;$$

    $X(k)\big|_{L_k}$ is $(Y_-[k-1]\times \A^1) \cup_{D_-[k-1]\times \A^1} Y_+(1)^\circ$.
    \end{prop}

\subsection{Stacks of expanded degenerations and pairs}

Now we can define the stacks parameterizing expanded degenerations and pairs. We start with the relative case.

Let $(Y,D)$ be a locally smooth pair. We have constructed the standard family $\pi:Y(k)\rightarrow \A^k$, with $(\C^*)^k$-action on both $Y(k)$ and the base, making the family maps equivariant. We also have isomorphisms on open substacks $\tilde{\tau}_{I,I',Y}: Y(k) \big|_{U_I} \xrightarrow{\sim} Y(k) \big|_{U_{I'}}$, compatible with the isomorphisms $\tilde{\tau}_{I,I'}: U_I\xrightarrow{\sim} U_{I'}$, where $I,I'\subset \{1,\cdots,k\}$ have the same number of elements $|I|=|I'|=l$.

These discrete symmetries do not form a group action on $Y(k)$ or $\A^k$, but they give \'etale equivalence relations. We rephrase the discrete symmetries as the relations
\begin{equation}\label{R_II'}
\xymatrix{
R_{I,I',\A^k}:=\A^l\times (\C^*)^{k-l} \ar@<.4ex>[r] \ar@<-.4ex>[r] & \A^k,}
\end{equation}
$$\xymatrix{
R_{I,I',Y(k)}:=Y(l)\times (\C^*)^{k-l} \ar@<.4ex>[r] \ar@<-.4ex>[r] & Y(k),
}$$
where the maps are given by open immersions $\A^l\times (\C^*)^{k-l} \xrar{\sim} U_I \hookrightarrow \A^k$ and $\A^l\times (\C^*)^{k-l} \xrar{\sim} U_{I'} \hookrightarrow \A^k$ respectively, and similar for $Y(k)$.

The \'etale equivalence relation on the whole space is just the union
$$R_{d,\A^k}:=\coprod_{1\leq |I|=|I'|\leq k} R_{I,I',\A^k}, \quad R_{d,Y(k)}:=\coprod_{1\leq |I|=|I'|\leq k} R_{I,I',Y(k)},$$
where ``$d$" refers to ``discrete".

In other words, we have $R_{d,Y(k)}=\pi^* R_{d,\A^k}$ for $\pi: Y(k)\rar \A^k$, or the following \emph{Cartesian} diagram,
$$\xymatrix{
R_{d,Y(k)} \ar[r] \ar[d] & Y(k)\times Y(k) \ar[d] \\
R_{d,\A^k} \ar[r] & \A^k\times \A^k.
}$$
Let $S_k$ be the symmetric group acting on $\A^k$ by permuting the coordinates. The discrete relation on $\A^k$ is a sub-relation of the $S_k$-action
$$\xymatrix{
R_{d,\A^k} \ar@{^{(}->}[r] & S_k\times \A^k \ar@<.4ex>[r] \ar@<-.4ex>[r] & \A^k.
}$$
Hence the two maps $R_{d,\A^k}\rightrightarrows \A^k$ are open immersions followed by a disjoint trivial covering, which is quasi-affine and \'etale.

Now let's describe the combined equivalence relation. $S_k\subset GL(k)$ acts on $(\C^*)^k$ by conjugation and we can form the semidirect product $(\C^*)^k \rtimes S_k$. We have the following smooth equivalence relation $\sim$ generated by $(\C^*)^k$-action and discrete symmetries, which is quasi-compact and separated,
$$\xymatrix{
R_{\sim, \A^k}:= (\C^*)^k\times R_{d,\A^k} \ar@{^{(}->}[r] & (\C^*)^k \rtimes S_k \times \A^k \ar@<.4ex>[r] \ar@<-.4ex>[r] & \A^k.
}$$
We also have the similar relation $R_{\sim,Y(k)}$ on $Y(k)$ and they form a \emph{non-Cartesian} diagram
$$\xymatrix{
R_{\sim,Y(k)}= (\C^*)^k\times R_{d,Y(k)} \ar@<.4ex>[r] \ar@<-.4ex>[r] \ar@<4.5ex>[d] & Y(k) \ar[d] \\
R_{\sim,\A^k}= (\C^*)^k\times R_{d,\A^k} \ar@<.4ex>[r] \ar@<-.4ex>[r] & \A^k.
}$$
$R_{\sim, \A^k}$ and $R_{\sim, Y(k)}$ are the equivalence relations generated by the $(\C^*)^k$-action and discrete symmetries.

Consider the quotients, Artin stacks $[\A^k/R_{\sim,\A^k}]$ of dimension 0 and $[Y(k)/R_{\sim,Y(k)}]$ of dimension $(\dim Y-k)$, with induced 1-morphism $\pi: [Y(k)/R_{\sim,Y(k)}] \rar [\A^k/R_{\sim,\A^k}]$.

Now consider the embeddings $\tau_I: \A^l\hookrightarrow \A^k$ and $\tau_{I,Y}: Y(l)\hookrightarrow Y(k)$, which are compatible with the equivalence relation. Thus we have embeddings of Artin stacks $[\A^l/R_{\sim,\A^l}] \rar [\A^k/R_{\sim,\A^k}]$ and $[Y(l)/R_{\sim,Y(l)}]\rar [Y(k)/R_{\sim,Y(k)}]$. In fact they are open immersions. To see this one can use the property $Y(k)|_{U_I} \cong Y(l)\times (\C^*)^{k-l}$ and take the smooth covers $\A^l\times (\C^*)^{k-l} \rar \A^l$ and $Y(l) \times (\C^*)^{k-l} \rar Y(l)$. Now it is clear that these immersions form inductive systems, leading to the following definition.

\begin{defn}
Define $\fA:=\varinjlim\ [\A^k/R_{\sim,\A^k}]$ to be the \emph{stack of expanded pairs}, with respect to $(Y,D)$, and let $\fY:=\varinjlim\ [Y(k)/R_{\sim,Y(k)}]$ be the \emph{universal family of expanded pairs}. There is a family map $\pi: \fY \rar \fA$, which is of Deligne--Mumford type and is proper.
\end{defn}

We interpret the definition in the categorical sense. For a fixed map $\xi_0: S \rar \A^k$, pulling back the standard family $Y(k)$, one obtains $\pi: \cY_S:= \xi_0^*Y(k) \rar S$, a $\emph{family of expanded pairs}$ over $(S,\xi_0)$. The equivalence relation acts on $\xi_0$ by acting on the target, and for maps related by this relation we get isomorphic families of expanded pairs. In this way one can think of a map $\xi_0$ as a family of expanded pairs over $S$, up to the action of the equivalence relation.

Given a morphism of schemes $f: T\rar S$, and a map $\xi: S\rar \A^k$, one can take the composite $\xi\circ f: T\rar \A^k$. The corresponding family of expanded pairs over $T$ is just given by $\cY_T = f^*\cY_S$. By the following lemma we see that it is unique up to unique isomorphism.

\begin{lem} \label{1-rigid}
  Suppose one has the following 2-commutative diagram
  $$\xymatrix{
  \cY_T \ar[r]^F \ar[d] & \cY_S \ar[d] \\
  T \ar[r]^f & S.
  }$$
  Then the 1-morphism of stacks $F$ is representable. As a consequence, $F$ has no nontrivial 2-isomorphisms.
\end{lem}

\begin{proof}
  By Lemma 4.4.3 of \cite{AV}, it suffices to prove that the homomorphisms between isotropy groups of geometric points are monomorphisms. Since maps between fibers are just maps between expanded pairs $Y[k]\rar Y[l]$, it suffices to prove for those maps. But any map $Y[k]\rar Y[l]$ is a successive composition of contractions of bubbles and embeddings, which is obviously representable.
\end{proof}

Let $S$ be a scheme. An object $\bar\xi\in \fA(S)$ is a compatible system of objects $\xi_k\in [\A^k/R_{\sim,\A^k}](S)$. By construction, an object $\xi\in [\A^k/R_{\sim,\A^k}](S)$ is given by a ``descent datum" over $S$, i.e. a map $\xi: S_\xi \rar \A^k$, where $S_\xi=\coprod S_i\rar S$ is a surjective \'etale covering of $S$ which satisfies the descent compatibility on overlaps; in other words, we have a map $r_\xi: R_\xi:=S_\xi\times_S S_\xi \rar R_{\sim,\A^k}$ compatible with the groupoid structure. By the interpretation above, one can view this as a family of expanded pairs over $S_\xi$.

The 2-isomorphisms are as follows. Given another $\xi': S_{\xi'}\rar \A^{k'}$ representing the same object, we can embed $\A^k$ and $\A^{k'}$ into a larger base $\A^{k''}$ for $k''\geq k,k'$, and pass to the refinement $S_{\xi\xi'}:=S_\xi\times_S S_{\xi'}$. Then for sufficiently large $k''$, the two resulting maps $S_{\xi\xi'}\rar \A^{k''}$ give the same object on the $k''$ level. In other words, they factor through the relation $R_{\sim,\A^{k''}}\rar \A^{k''}\times \A^{k''}$.

For a map of schemes $f: T\rar S$, passing to the \'etale cover, we have $f_\xi: T_\xi:= T\times_S S_\xi \rar S_\xi$. Then the 1-arrow is defined by the composition $\xi\circ f$.

In particular, for a fixed object $\bar\xi\in \fA(S)$, the stabilizer group of this object is as follows. Take a representative $\xi: S_\xi\rar \A^k$. The stabilizer group of this representative is the group scheme $\Aut_\sim(\xi,k)$ over $S_\xi$ in the following Cartesian diagram,
$$\xymatrix{
\Aut_\sim(\xi,k) \ar[r]\ar[d] & S_\xi \ar[r]\ar[d]^{\Delta\circ \xi} & S \\
R_{\sim,\A^k} \ar[r] & \A^k\times \A^k &.
}$$
Note that this group scheme actually does not depend on $k$ for large $k$, and by descent theory of affine morphisms it glues to a group scheme over $S$. We denote it by $\Aut_\sim(\bar\xi)$, which is independent of $S_\xi$ and $k$ for large $k$. As a result, one can see that $\fA$ is actually a 0-dimensional Artin stack.

Now let's apply the same procedure to a locally simple degeneration $\pi:X\rar \A^1$. We have constructed the standard family $\pi: X(k)\rar \A^{k+1}$, again with a compatible $(\C^*)^k$-actions on $X(k)$ and the base. We introduce smooth equivalence relations on $X(k)$ and $\A^{k+1}$ in the same manner. For $k\geq 0$, $[\A^{k+1}/R_{\sim,\A^{k+1}}]$ and $[X(k)/R_{\sim,X(k)}]$ also form inductive systems.

\begin{defn}
  Define $\fC:= \varinjlim \ [\A^{k+1}/R_{\sim,\A^{k+1}}]$ to be the \emph{stack of expanded degenerations}, with respect to the locally simple degeneration $\pi:X\rar \A^1$, and let $\fX:= \varinjlim \ [X(k)/R_{\sim,X(k)}]$ be the \emph{universal family of expanded degenerations}. There is a family map $\pi: \fX\rar \fC$, of Deligne--Mumford type and proper, making the following diagram commute (but not Cartesian):
  $$\xymatrix{
  \fX \ar[r]^p \ar[d]_\pi & X \ar[d]^\pi \\
  \fC \ar[r]^p & \A^1.
  }$$
\end{defn}

Both $\fC$ and $\fX$ can be viewed as Artin stacks over $\A^1$. $\fC$ is 1-dimensional and $\fX$ is of dimension $\dim X$.

Given an $\A^1$-map $\xi: S\rar \A^{k+1}$, a \emph{family of expanded degenerations} $\pi: \cX_S\rar S$ over $(S,\xi)$ is obtained by pull back of the standard family. Given a morphism of $\A^1$-schemes $f: (T,\eta) \rar (S,\xi)$, the pull-back 1-morphism is just given by $f^*\cX_S = \cX_T$, unique up to 2-isomorphisms.

For an $\A^1$-scheme $S$, an object $\bar\xi \in \fC(S)$ is represented by an $\A^1$-map $\xi: S_\xi\rar \A^{k+1}$, where $S_\xi=\coprod S_i$ is a surjective \'etale covering of $S$. Different representatives of the same object are related by embedding into a common larger base $\A^{k'+1}$ and the equivalence relation over that base.

\begin{rem}
  Given a family of expanded pairs (resp. degenerations) $\pi: \cY\rar S$ (resp. $\cX\rar S$), each fiber of the family is an expanded pair (resp. degeneration) with respect to $(Y,D)$ (resp. $\pi:X\rar \A^1$). Given a base change $f: (T,\eta) \rar (S,\xi)$, the induced map between fibers are given by maps of the corresponding expanded pairs (resp. degenerations). In this way we see that the definition makes sense to parameterize all expanded pairs (resp. degenerations).
\end{rem}

\section{Admissible sheaves and stable quotients}

In this section we introduce the notion of admissible sheaves, the correct objects we need to consider for the relative Donaldson--Thomas Theory. Again, our definitions are direct generalizations of J. Li and B. Wu \cite{LW}.

Let $W$ be a separated Deligne--Mumford stack of finite type, and $Z\subset W$ be a closed substack. Let $\F$ be a coherent sheaf on $W$.

\begin{defn}
  $\F$ is said to be \emph{normal} to $Z$ if $\text{Tor}_1^{\Ou_W} (\F,\Ou_Z)=0$. Moreover, we say that $\F$ is \emph{normal} to $Z$ \emph{at a point} $p\in Z$, if there is an \'etale neighborhood $i: U_p \rar W$ of $p$, such that $i^*\F$ is normal to $Z\times_W U_p$.
\end{defn}

According to the following lemma, most properties of the normality of coherent sheaves can be directly generalized to Deligne--Mumford stacks.

\begin{lem}
  $\F$ is normal to $Z$ if and only if it is normal to $Z$ at every point $p\in Z$. In other words, normality is a local property in the \'etale topology.
\end{lem}

\begin{proof}
  Let $\I$ be the ideal sheaf of $Z$. $\F$ is normal to $Z$ if and only if the map $\I\otimes \F \rar \F$ is injective, which is a local property in the \'etale topology.
\end{proof}

\begin{rem}
  From the definition we note that normality can be checked on stalks or completion over local rings.
\end{rem}

In the following we will consider two cases we mostly care about, the relative case and degeneration case.

\subsection{Relative case}

This is the case where $Y$ is a separated Deligne--Mumford stack of finite type and $D\subset Y$ is an effective Cartier divisor. Take an affine \'etale neighborhood $U=\Sp A$ of $p\in D$, where $D|_U$ is defined by some nonzero-divisor $f\in A$. Then there is a map $U\rar \A^1=\Sp \C[y]$, given by $y\mapsto f$. A coherent sheaf $\F$ on $U$ is represented by an $A$-module $M$. The following lemma gives a local description of normality in this case, which is a restatement of the definition by the flatness criteria in commutative algebra.

\begin{lem} \label{lem-local-normal}
  The followings are equivalent.
  \begin{enumerate}[1)]
  \setlength{\parskip}{1ex}

  \item $\F|_U$ is normal to $D|_U$;

  \item The map $M\xrar{\times f} M$ is injective;

  \item $M$ is flat over $\A^1$ at the point $0\in \A^1$, i.e. the stalk of $M$ at $0\in \A^1$ is flat over the local ring $\Ou_{\A^1,0}=\C[y]_{(y)}$.
  \end{enumerate}
\end{lem}

Let's look closer into the coherent sheaf $\F$ and analyze its normality. Let $\I\subset \Ou_Y$ be the ideal sheaf of $i:D\hookrightarrow Y$. Define $\F_\I:=i^! \F$ to be the maximal subsheaf of $\F$ supported on $D$. More precisely, locally in an affine open $U$, for $\I(U)=I$,
$$\F_\I(U):=\{ m\in M \mid \text{ann}(m) \supset I^k, \text{ for some } k\in \Z_+\}.$$
One can easily check by descent theory that this definition defines a global coherent sheaf on $Y$. We call $\F_\I$ the torsion subsheaf of $\F$ along $D$.

Another way to define the torsion subsheaf is to take $\F_\I(U)=M_I=\ker(M\xrar{\times f} M)$, which is the torsion part of $M$ as a $\C[y]$-module. In this viewpoint the torsion free quotient is $$\F^\tf:=\F/\F_\I=\coker(M\xrar{\times y} M),$$
and we have the short exact sequence
$$\xymatrix{
0 \ar[r] & \F_\I \ar[r] & \F \ar[r] & \F^\tf \ar[r] & 0.
}$$
The following proposition is a direct result of Lemma \ref{lem-local-normal}.

\begin{prop}
  $\F$ is normal to $D$ at $p\in D$ if and only if $(\F_\I)_p=0$, where the subscript $p$ stands for the stalk in the \'etale topology. In particular, $\F$ is normal to $D$ if and only if $\F_\I=0$.
\end{prop}

The following propositions tell us that normality is an open condition.

\begin{prop} \label{open_Y}
  Let $S$ be a scheme, and $\F$ be a coherent sheaf on $Y\times S$ which is flat over $S$. Then the set $\{s\in S\mid \F|_s:=\F\otimes_{\Ou_S} k(s) \text{ is normal to } D\times s \}$ is open in $S$.
\end{prop}

In order to prove this proposition, we use the fiberwise criterion for flatness, which is a consequence of Theorem 11.3.10 of \cite{EGA3} (also see Theorem 36.16.2 of \cite[\href{http://stacks.math.columbia.edu/tag/039C}{Tag 039C}]{stacks-project}).

\begin{lem}[(Fiberwise criterion for flatness of coherent sheaves)]
  Let $S,X,Y$ be locally noetherian schemes, and $g:X\rar S$, $h: Y\rar S$ be two morphisms of schemes. Let $f: X\rar Y$ be an $S$-morphism, and $\F$ be a coherent sheaf on $X$. $x\in X$, $y=f(x)$, $s=h(y)=g(x)$. Assume that the stalk $\F_x\neq 0$. Then the followings are equivalent:
  \begin{enumerate}[1)]
  \setlength{\parskip}{1ex}
    \item $\F$ is flat over $S$ at $x$, and $\F|_s$ is flat over $Y_s$ at $x\in X_s$;
    \item $Y$ is flat over $S$ at $y$, and $\F$ is flat over $Y$ at $x$.
  \end{enumerate}
\end{lem}

\begin{proof}[Proof for Proposition \ref{open_Y}]
  Suppose for $s\in S$, $\F|_s$ is normal to $D$. Then $\forall p\in D$, $\F|_s$ is normal to $D$ in some affine \'etale neighborhood $U_p$ of $p$. Consider the local diagram
  $$\xymatrix{
  U_p\times S \ar[rr] \ar[dr] & & \A^1 \times S \ar[dl] \\
  & S, &
  }$$
where $U_p \to \A^1$ is given as above.

  We know that $\F|_s$, as a sheaf in $U_p\times \{s\}$, is flat over $\A^1\times \{s\}$ at $(p,s)$. By the fiberwise criterion, $\F$ is flat over $\A^1\times S$ at this point; in other words, there exist affine \'etale neighborhoods of $p\in Y$, which we still call $U_p$, and $V_p$ of $s\in S$ such that on $U_p\times V_p \subset Y\times S$, $\F$ is flat over $\Sp \Ou_{\A^1,0}\times V_p$. Since $D$ is quasi-compact, we can take a finite subcover $U_i\times V_i$, $1\leq i\leq N$ such that $\{U_i\}$ cover $D$, and take $V:=V_1\times_S \times \cdots \times_S V_N$. Then we have $V\subset \{s\in S\mid \F|_s \text{ is normal to } D \}$, which proves the proposition.
\end{proof}

\subsection{Degeneration case}

In this case we consider $X_0= Y_-\cup_D Y_+$, where the two separated finite-type Deligne--Mumford stacks $Y_\pm$ transversally intersect along an effective Cartier divisor, and $X_0$ is the glueing along $D$. By transversality we mean that \'etale locally around $p\in D$ we have a common affine neighborhood $U=\Sp A$ of the following,
$$\xymatrix{
T\times \Sp \C[x,y]/(xy) & \ar[l]_-{\text{\'et}} U \ar[r]^-{\text{\'et}} & X_0,
}$$
where $T$ is a scheme.

Hence we have a map $U\rar \Sp \C[x,y]/(xy)$. The ring can be rewritten as $\C[x,y]/(xy)\cong \C[x]\times_\C \C[y] = \{ (f,g)\in \C[x]\times \C[y] \mid f(0)=g(0) \}$.

A coherent sheaf $\F$ on $U$ is represented by an $A$-module $M$. We have the following lemma.

\begin{lem} \label{lem_normality_X}
  The followings are equivalent.
  \begin{enumerate}[1)]
    \setlength{\parskip}{1ex}

    \item $M$ is flat over the local ring $(\C[x,y]/(xy))_{(x,y)}$;

    \item $M/yM$ is flat over $\C[x]_{(x)}$, and $M/xM$ is flat over $\C[y]_{(y)}$;

    \item $\F|_{U\cap Y_\pm}$ are normal to $D|_U\subset U\cap Y_\pm$;

    \item $\F_U$ is normal to $D|_U$.

    In particular, 1)-4) implies that
    \item $M\cong (M/yM)\times_{M/(x,y)M} (M/xM)$.
  \end{enumerate}

\end{lem}

\begin{proof}
  2)$\Leftrightarrow$3) is the consequence of the relative case. 1)$\Leftrightarrow$4) is true since 4) is equivalent to the injectivity of the map $(x,y)\otimes M \rar M$, which is equivalent to 1) since $(x,y)$ is the maximal ideal, and the map is obviously injective when localized at other maximal ideals. 1)$\Rightarrow$2) is the base change property of flatness.

  Let's prove 2)$\Rightarrow$ 5). Suppose 2) holds. The obvious map $\varphi: M \rar (M/yM)\times_{M/(x,y)M} (M/xM)$ is easily seen to be surjective. In fact, given $(\bar m_1, \bar m_2)$ in the target, one can choose $m_1, m_2 \in M$ in the preimages of $\bar m_1$, $\bar m_2$ respectively. By construction, there exist $a, b\in M$ such that $m_1-m_2 = ax+by$. Then $m_1-by = m_2+ax \in M$ is in the preimage of $(\bar m_1, \bar m_2)$.

For the injectivity, $\varphi$ is identity at points $(x,y)\neq (0,0)$, thus it remains to check its injectivity at $(0,0)$, or equivalently, over the local ring at $(0,0)$. If for some $m\in M_{(x,y)}$, $\varphi(m)=0$, then there are some $m_1,m_2\in M_{(x,y)}$, such that $m=xm_1=ym_2$. Then by 2), $xm_1=ym_2$ implies that $m_1\in yM_{(x,y)}$ and $m_2\in xM_{(x,y)}$, which implies that $m=0$. Thus $\varphi$ is injective and 5) is proved.

  Finally let's prove 2)$\Rightarrow$1). Just need to check the injectivity of the map $(x,y)\otimes M_{(x,y)}\rar M_{(x,y)}$. We can view $(x,y)$ as $(x)\times_\C (y)$, the maximal ideal of $(\C[x,y]/(xy))_{(x,y)} \cong \C[x]_{(x)}\times_\C \C[y]_{(y)}$. By 2)$\Rightarrow$5), $M\cong (M/yM)\times_{M/(x,y)M} (M/xM)$; thus it suffices to check the injectivity on the two components, which reduces to the flatness stated in 2).
\end{proof}

\begin{cor} \label{normality_split}
  $\F$ is normal to $D$ if and only if $\F|_{Y_\pm}$ are normal to $D\subset Y_\pm$ respectively. Moreover, in this case $\F\cong \ker (\F|_{Y_-} \oplus \F|_{Y_+} \rar \F|_D)$.
\end{cor}

Again one can consider the maximal torsion subsheaf supported on $D$, denoted by $\F_\J$, where $\J\subset \Ou_{X_0}$ is the ideal sheaf of $D\subset X_0$. Let $\I_\pm$ be the ideal sheaves of $D\subset Y_\pm$, and $\F_{\I_\pm}$ be the corresponding torsion subsheaves defined in the relative case. We have the short exact sequence
$$\xymatrix{
0\ar[r] & \F_\J \ar[r] & \F \ar[r] & \F^\tf \ar[r] & 0.
}$$
Now the case is a little more complicated. $\F^\tf$ in general is not normal to $D$, but it satisfies the splitting property as in 5) of Lemma \ref{lem_normality_X}.

\begin{lem} \label{J_split}
  If $\F_\J=0$, then $\F\cong \ker (\F|_{Y_-} \oplus \F|_{Y_+} \rar \F|_D)$.
\end{lem}

\begin{proof}
  Locally the map $\varphi: M\rar (M/yM) \times_{M/(x,y)M} (M/xM)$ is surjective, where $M, x, y$ is the same as in Lemma \ref{lem_normality_X}. Take $v\in M$ in the kernel. By definition there is $m,n\in M$ such that $v=xm=yn$. Thus we have $(x,y)\cdot v=0$, which implies $v\in M_J=0$.
\end{proof}

To determine whether $\F$ is normal to $D$ we look at the quotient $\F^\tf$. It is normal to $D$ if and only if $\F^\tf|_{Y_\pm}$ is normal to $D$. Thus we have the following.

\begin{cor}
  $\F$ is normal to $D$ if and only if $\F_\J= (\F^\tf|_{Y_-})_{\I_-}= (\F^\tf|_{Y_+})_{\I_+}= 0$.
\end{cor}

Combining Corollary \ref{normality_split} with Proposition \ref{open_Y} we also have the following.

\begin{cor} \label{open_X}
  Let $S$ be a scheme, and $\F$ be a coherent sheaf on $X_0\times S$ which is flat over $S$. Then the set $\{s\in S\mid \F|_s:=\F\otimes_{\Ou_S} k(s) \text{ is normal to } D\times s \}$ is open in $S$.
\end{cor}

\subsection{Admissible sheaves}

Consider a locally smooth pair $(Y,D)$ for the relative case and a locally simple degeneration $\pi:X\rar \A^1$, $X_0=Y_-\cup_D Y_+$ for the degeneration case. We make the following definition.

\begin{defn}
  A coherent sheaf $\F$ on $Y[k]$ (resp. $X_0[k]$) is called \emph{admissible} if it is normal to each $D_i\subset Y[k]$ (resp. $D_i\subset X_0[k]$), $0\leq i\leq k$.

  Let $\cY_S\rar S$ (resp. $\cX_S\rar S$) be a family of expanded pairs (resp. degenerations) and $\F$ be a coherent sheaf on $\cY_S$ (resp. $\cX_S$), flat over $S$. We call $\F$ \emph{admissible} if for every point $s\in S$, the fiber $\F|_s$ is admissible on the fiber $\cY_{S,s}$ (resp. $\cX_{S,s}$).
\end{defn}

\begin{rem}
   The difference of the definition of admissibility on $X_0[k]$ and $Y[k]$ is that we include the distinguished divisor $D_k$ in the relative case.
\end{rem}

We would like to prove that admissibility is an open condition. Before that, we need a criterion for normality on a locally simple degeneration.

\begin{lem} \label{normality_simple}
  Consider the standard simple degeneration $\pi:X=\Sp \C[x,y,t]/(xy-t) \rar \A^1=\Sp\C[t]$, with singular divisor of the central fiber $D=\Sp \C$. Let $M$ be a coherent sheaf on $X$, flat over $\A^1$. Then $M$ is normal to $D\subset X$ if and only if $M|_0$ is normal to $D\subset X_0$.
\end{lem}

\begin{proof}
  This is an immediate consequence of the Slicing Criterion for flatness, which can be found in Corollary 6.9 of \cite{Ei}.
\end{proof}

\begin{prop} \label{open_XY_S}
  Let $\cY_S$ (resp. $\cX_S$) be a family of expanded pairs (resp. degenerations). Let $\F$ be a coherent sheaf on $\cY_S$, flat over $S$. Then the set $\{s\in S \mid \F|_s \text{ is admissible} \}$ is open in $S$.
\end{prop}

\begin{proof}
For standard families $Y(k)\to \A^k$, it suffices to prove the points $s\in \A^k$ over which $\F|_s$ is normal to all $D_i\in \cY_s$ form an open subset. By Lemma \ref{normality_simple}, those are exactly points $s\in \A^k$ where $\F$ is normal to all $D_i(k)$ at every $p\in \cY_s$. From the properties of $Y(k)$ we know that after projection to the $i$-th coordinate, $1\leq i\leq k$, the 1-dimensional family $\pi_i:= \pr_i\circ \pi: Y(k)\rar \A^1$ is also a locally simple degeneration. In other words, locally around $D_i(k)$, $Y(k)$ has a common \'etale neighborhood with the local model $D_i(k)\times \Sp \C[x,y,t_i]/(xy-t_i) \cong D_i\times \A^{k-1} \times \Sp \C[x,y,t_i]/(xy-t_i)$. At points where $t_i\neq 0$, the normality of $\F$ is trivial. Moreover, by Corollary \ref{open_X}, the points $s\in H_i=(t_i=0)$ where $\F$ is not normal to $D_i$ form a closed subset. Thus it is still closed in $\A^k$, whose complement is open. Similar proof for the distinguished divisor $D(k)$. For general families $\cY_S$ over $S$, the proof works by pulling everything back to $S$.
\end{proof}

\subsection{Stable quotients}

Now we come to the stability condition, and the notion of stable quotients, following the convention of J. Li and B. Wu \cite{LW}.

In this subsection, let $(Y,D)$ and $\pi:X\rar \A^1$ be as above, we have the contractions $p: Y[k]\rar Y$ and $p: X_0[k]\rar X_0$, contracting the extra bubbled components. Let $\cV$ be a vector bundle of finite rank on $Y$ (resp. $X$).

Consider quotient sheaves of the form $\phi: p^*\cV \rar \F$ on $Y[k]$ (resp. $X_0[k]$). For two quotients $\phi_1: p^*\cV \rar\F_1$ and $\phi_2: p^*\cV \rar\F_2$, an equivalence between them is defined as a pair $(\sigma, \psi)$, where $\sigma: Y[k]\rar Y[k]$ (resp. $X_0[k]\rar X_0[k]$) is an isomorphism \emph{induced from the $(\C^*)^k$-action}, and $\psi: \F_1\xrar{\sim} \sigma^*\F_2$ is an isomorphism of coherent sheaves, making the following diagram commute,
$$\xymatrix{
p^*\cV \ar[rr]^{\phi_1} \ar[d]_\Id && \F_1 \ar[d]^\psi \\
\sigma^* p^*\cV= p^*\cV \ar[rr]^(.6){\sigma^* \phi_2} && \sigma^* \F_2.
}$$

Let $\Aut(\phi)$ be the group of autoequivalences of a fixed quotient $\phi:p^*\cV \rar \F$. There is a map $\Aut(\phi)\rar (\C^*)^k$, $(\sigma,\psi) \mapsto \sigma$ forgetting the second component. This is an injection, since for a fixed quotient $\sigma=1$ will just identify $p^*\cV$ and induce $\Id$ on $\F$. As a result, $\Aut(\phi)\subset (\C^*)^k$ is a subgroup.

\begin{rem}
  An equivalent definition of the equivalence of quotients is to consider the kernels of the quotients, say $0\rar \mathcal{K}_i \rar p^*\cV \rar \F_i \rar 0$, and define an equivalence as an isomorphism of the two kernels as subsheaves of $p^*\cV$. It is clear from this definition that $\Aut(\phi)\subset (\C^*)^k$ is a subgroup.
\end{rem}

\begin{defn}
  Let $\phi:p^*\cV\rar \F$ be a quotient sheaf on $Y[k]$ (resp. $X_0[k]$). $\phi$ is called \emph{stable} if $\F$ is admissible, and $\Aut(\phi)$ is finite.
\end{defn}

For standard families $Y(k)$ and $X(k)$, we have contraction maps $p:Y(k)\rar Y$ and $p:X(k)\rar X$, which pass to general families.

\begin{defn}
  Let $\cY_S\rar S$ (resp. $\cX_S\rar S$) be a family of expanded pairs (resp. degenerations), with contraction map $p: \cY_S\rar Y$ (resp. $p:\cX_S\rar X$) and $\phi: p^*\cV\rar \F$ be a quotient sheaf on $\cY_S$ (resp. $\cX_S$), with $\F$ flat over $S$. Then $\phi$ is called \emph{stable} if for every point $s\in S$, the fiber $\F|_s$ is stable on the fiber $\cY_{S,s}$ (resp. $\cX_{S,s}$).
\end{defn}

The $(\C^*)^k$-actions on $Y(k)$ and $\A^k$ induce a $(\C^*)^k$-action on the Quot-space $\Quot_{Y(k)/\A^k}^{p^*\cV}$, which is a separated algebraic space locally of finite type by \cite{OS}. For a fixed object $(\xi,\phi)$ represented by $\xi: S\rar \A^k$ and an $S$-flat quotient $\phi$ on $\cY_S$, the stabilizer is given by
\begin{equation}\label{Aut_C_S_phi}
  \xymatrix{
  \Aut_{*,S}(\xi,\phi,k) \ar[r] \ar[d] & S \ar[d]^-{\Delta\circ (\xi,\phi)} \\
  (\C^*)^k \times \Quot_{Y(k)/\A^k}^{p^*\cV} \ar[r] & \Quot_{Y(k)/\A^k}^{p^*\cV} \times \Quot_{Y(k)/\A^k}^{p^*\cV}.
  }
\end{equation}
We see that $\Aut_{*,S}(\xi,\phi,k)\hookrightarrow (\C^*)^k\times S$ is a subgroup scheme over $S$. Thus it is quasi-compact and separated over $S$, which means that the action is quasi-compact and separated.

We also have the open condition property.

\begin{prop} \label{open_stable}
  Let $\cY_S\rar S$ (resp. $\cX_S\rar S$) be a family of expanded pairs (resp. degenerations), with contraction map $p: \cY_S\rar Y$ (resp. $p:\cX_S\rar X$) and $\phi: p^*\cV\rar \F$ be a quotient sheaf on $\cY_S$ (resp. $\cX_S$), with $\F$ flat over $S$. The set $\{s\in S\mid \phi_s:=p^*\cV \rar \F|_s \text{ is stable } \}$ is open in $S$.
\end{prop}

\begin{proof}
  By Proposition \ref{open_XY_S}, it suffices to prove the finiteness of the autoequivalence group is an open condition. We take the relative case for example; the degeneration case is similar.

  Now $(\C^*)^k$ acts on $\Quot_{Y(k)/\A^k}^{p^*\cV}$ by diagram (\ref{Aut_C_S_phi}). Given an object $(\xi,\phi)\in \Quot_{Y(k)/\A^k}^{p^*\cV} (S)$, the stabilizer group scheme of the action is just $\Aut_{*,S}(\xi,\phi,k)$. Thus the set $\{s\in S\mid \phi_s:=p^*\cV \rar \F|_s \text{ is stable } \}$ would be the set where $\Aut_{*,S}(\xi,\phi,k)\rar S$ is quasi-finite, which is open.
\end{proof}

We still need some concepts for later discussions of sheaves. The following definitions can be found in \cite{Lieb} and \cite{Ni}.

\begin{defn}
  Let $\F$ be a coherent sheaf on a Deligne--Mumford stack $W$. The \emph{support} of $\F$ is the closed substack defined by the ideal sheaf $0\rar \I \rar \Ou_W \rar \mathscr{E}nd_{\Ou_W}(\F)$. The \emph{dimension} of $\F$ is the dimension of its support. $\F$ is called \emph{pure}, if any proper subsheaf of it has the same dimension as $\F$.
\end{defn}

\begin{defn}
  Let $T_i(\F)$ be the maximal subsheaf of $\F$ of dimension $\leq i$. Then we have the \emph{torsion filtration}
  $$0\subset T_0(\F)\subset \cdots \subset T_d(\F)=\F,$$
  where $d=\dim \F$ and each factor $T_i(\F)/T_{i-1}(\F)$ is 0 or pure of dimension $i$.
\end{defn}

\begin{prop} \label{adm_supp}
  Let $\F$ be a coherent sheaf on a locally smooth pair $(Y,D)$. Then $\F$ is normal to $D$ if and only if no irreducible components of any nonempty $\Supp T_i(\F)$ lie in $D$.
\end{prop}

\begin{proof}
  If $\F$ is normal to $D$, and some irreducible component of $T_i(\F)$ lies in $D$, then for some point in that component, $\F$ is not flat along the normal direction to $D$ since the multiplication in the normal direction vanishes on $T_i(\F)$. Contradiction.

  Conversely, suppose no irreducible components of any $T_i(\F)$ lie in $D$ and $\F$ is not normal to $D$. Then $\F_\I\neq 0$, where $\I$ is the ideal sheaf of $D\subset Y$. Suppose $\dim \F_\I = d'$, then $\Supp\F_\I \subset D$ contains some irreducible components of $\Supp T_{d'}(\F)$. Contradiction.
\end{proof}

The followings are some examples for stable quotients.

\begin{eg}
  Let $\cV=\Ou_Y$ (similar with $\Ou_{X_0}$) be the structure sheaf. A quotient sheaf is just a closed substack $Z$ of $Y[k]$.

  If $\dim Z=0$, in other words, $Z$ is just a collection of (possibly multiple or gerby) points. Then $\Ou_{Y[k]}\rar \Ou_Z$ is admissible if and only if none of the points of $Z$ supports on any divisor $D_i$, $0\leq i\leq k$. It is stable if moreover $Z\cap \Delta_i\neq \emptyset$, $1\leq i\leq k$.

  If $\dim Z=1$ and $Z$ is pure, then $Z$ is admissible if and only if none of its irreducible components lies entirely in any divisor $D_i$, $0\leq i\leq k$. $Z$ is stable if furthermore it is not ``entirely fiberwise"; in other words, restricted to any extra bubbled component of $Y[k]$, it is not a fiber (or a union of fibers) of the corresponding (possibly orbi) $\Pj^1$-bundle.
\end{eg}

\subsection{$\C^*$-equivariant flat limits}

To conclude this section we describe the completion of flat families of quotients. Let $\mathcal W$ be a separated Deligne-Mumford stack of finite type, with a vector bundle $\cV$, and $\pi: \mathcal W \rar \A^1$ be proper and flat. Let $\phi^\circ: \cV|_{\C^*} \rar \F^\circ$ be a quotient of coherent sheaves on $\mathcal W|_{\C^*}$, flat over $\C^*$.

By properness of the Quot-space $\Quot_{\mathcal W/\A^1}^\cV$ (Theorem 1.5 of \cite{OS}), there is a unique quotient $\phi: \cV \rar \F$ on $\mathcal W$, flat over $\A^1$, whose restriction to $\mathcal W|_{\C^*}$ is the given family. We call $\phi$ the \emph{flat completion} and $\phi|_0$ the \emph{flat limit} of the given family at $0\in \A^1$.

Consider the following special case. Let $W$ be a proper Deligne-Mumford stack, with a vector bundle $V$. Consider the family $\pi: W\times \A^1\rar \A^1$, with the other projection $p: W\times \A^1\rar W$. $\C^*$ acts obviously on the $\A^1$-factor of the family $W\times \A^1$. Given $\phi^\circ: p^*V|_{W\times \C^*}\rar \F^\circ$ be a $\C^*$-equivariant and $\C^*$-flat quotient of coherent sheaves on $W\times \C^*$, we have the flat completion $\phi$, which is still $\C^*$-equivariant.

\begin{lem} \label{lem_equivariant_limit}
  $\phi\cong p^*\phi|_0$.
\end{lem}

\begin{proof}
  We rephrase the situation in terms of the language of Quot-spaces. $\phi^\circ$ defines a $\C^*$-equivariant map from $F^\circ: \C^*\to \Quot_W$. Here the $\C^*$-action on the Quot-space is the trivial action. Hence $F^\circ$ is a constant map, defined by restricting $\phi^\circ$ to any fiber. By the properness of $\Quot_W$, the unique extension of $F^\circ$ to $\A^1$ is the constant map, which gives the trivial family $p^*\phi|_0$.
\end{proof}

Let $(Y,D)$ be a locally smooth pair as in previous sections, and $\cV$ be a vector bundle on $Y$. Let $\Delta=\Pj_D(\Ou\oplus N) \xrar{p} D$ be a bubble component.

\begin{prop} \label{equivariant_pullback}
  Let $\phi:p^*\cV \rar \F$ be an admissible $\C^*$-equivariant quotient on $\Delta$, with $\C^*$ acting along the fiber. Assume furthermore that $\C^*$ acts trivially on $V|_D$, i.e. $\forall p\in D$, $V|_p$ is trivial as a $\C^*$-representation. Then $\F\cong p^*\F|_D$.
\end{prop}

\begin{proof}
   Locally $\Delta - D_-$ has the form $W\times \A^1$ with $W$ an affine chart. Admissibility implies the flatness of $\F$ over $\A^1$. By Lemma \ref{lem_equivariant_limit} there is an isomorphism $\F\cong p^* \F|_D$ on this affine chart. Separatedness of $\Quot_W$ in Lemma \ref{lem_equivariant_limit} ensures that the isomorphisms on local affine charts are unique, and hence glue to a global isomorphism.
\end{proof}

We still need a technical lemma for later use. Consider $\Sp R:=\Sp \C[x,y,t]/(xy)$. Let $R_0= R/tR$, $R_-=R/yR$ and $R_{0,-}=R_0/yR_0$. Let $J=(x,y)\subset R$, $J_0=(x,y)\subset R_0$, $I_-= (x)\subset R_-$ and $I_{0,-}=(x)\subset R_0/yR_0$ be ideals. Let $\C^*$ acts on $t$ by weight $a<0$, on $x$ by weight $b>0$, and trivially on $y$. Let $\phi: R^{\oplus r} \rar M\rar 0$ be a $\C^*$-equivariant quotient, flat over $\C[t]$. Let $\phi_-: R_-^{\oplus r}\rar M_-\rar 0$ be its restriction on $R_-$.

\begin{lem}[(Lemma 3.13, Lemma 3.14 of \cite{LW})] \label{3.13}

\

\begin{enumerate}
\item $M_J\otimes_R R_0 \cong (M\otimes_R R_0)_{J_0};$

\item Let $M_-^*$ be the flat limit of the restriction of $M_-$ on $\Sp \C[x,x^{-1},t]$ along the $x$-direction to the locus $(x=0)$, and let $M_{0,-}^*$ be flat limit of the restriction of $M_{0,-}:= M_-\otimes_{R_-} R_{0,-}$ on $\Sp \C[x,x^{-1}]$ along the $x$-direction to the locus $(x=0)$. Then $M_-^*$ is the pull back of $M_{0,-}^*$ along the projection $\Sp \C[x, x^{-1}, t] \to \Sp \C[x, x^{-1}]$.
\end{enumerate}

\end{lem}

\begin{proof}
  The proof is exactly the same as in \cite{LW}.
\end{proof}

\section{Moduli of stable quotients}

\subsection{Quot-functor of stable quotients}

Now let's consider stable quotients on the stacks of expanded pairs and degenerations. Let $(Y,D)$ be a locally smooth pair and $\pi: X\rar \A^1$ be a locally simple degeneration. Let $\cV$ be a vector bundle of finite rank on $Y$ or $X$. We have the stacks $\fA$ and $\fC$, with universal families $\fY$ and $\fX$.

First let's describe the equivalence relation on the Quot-spaces. Recall we have the discrete symmetries given in the following \emph{Cartesian} diagram
$$\xymatrix{
R_{d,Y(k)} \ar@<.4ex>[r] \ar@<-.4ex>[r] \ar[d] & Y(k) \ar[d] \\
R_{d,\A^k} \ar@<.4ex>[r] \ar@<-.4ex>[r] & \A^k.
}$$
One has an \'etale equivalence relation
$$\xymatrix{
R_{d, \Quot_{Y(k)/\A^k}^{p^*\cV}} := \Quot_{R_{d,Y(k)}/ R_{d,\A^k}}^{p^*\cV} \ar@<.4ex>[r] \ar@<-.4ex>[r] & \Quot_{Y(k)/\A^k}^{p^*\cV},
}$$
induced by the two projections in the Cartesian diagram above. Here the maps are well-defined because $R_{d, \A^k} \rightrightarrows \A^k$ are actually disjoint unions of open immersions.

This is the induced discrete symmetries on the Quot-space. Two representatives $(\xi,\phi),(\xi',\phi')$, with $\xi, \xi': S\rar \A^k$, represent the same object if the associate map factors through $R_{d, \Quot_{Y(k)/\A^k}^{p^*\cV}}$.

Recall that we also have the $(\C^*)^k$-action
$$\xymatrix{
(\C^*)^k\times \Quot_{Y(k)/\A^k}^{p^*\cV} \ar@<.4ex>[r] \ar@<-.4ex>[r] & \Quot_{Y(k)/\A^k}^{p^*\cV}.
}$$
They combine to give a smooth equivalence relation on the Quot-space, still as a subrelation of the group action by the semidirect product $(\C^*)^k \rtimes S_k$,
$$\xymatrix{
R_{\sim, \Quot_{Y(k)/\A^k}^{p^*\cV}} := (\C^*)^k\times R_{d, \Quot_{Y(k)/\A^k}^{p^*\cV}} \ar@{^{(}->}[r] & (\C^*)^k\rtimes S_k \times \Quot_{Y(k)/\A^k}^{p^*\cV} \ar@<.4ex>[r] \ar@<-.4ex>[r] & \Quot_{Y(k)/\A^k}^{p^*\cV}
}.$$

In particular, for a fixed object represented by a map $\xi: S\rar \A^k$ and a quotient $\phi$ on $\cY_S$, its isotropy group is the pull back in the following Cartesian diagram,
\begin{equation}\label{Aut_sim_S_phi}
\xymatrix{
  \Aut_{\sim,S}(\xi,\phi,k) \ar[rr] \ar[d]  && S \ar[d]^{ \Delta\circ (\xi,\phi)} \\
  R_{\sim,\Quot_{Y(k)/\A^k}^{p^*\cV}}  \ar[rr] && \Quot_{Y(k)/\A^k}^{p^*\cV}\times \Quot_{Y(k)/\A^k}^{p^*\cV}.
  }
\end{equation}
$\Aut_{\sim,S}(\xi,\phi,k)$ is a subgroup scheme over $S$ of $(\C^*)^k\rtimes S_k \times S$, which is quasi-compact and separated. We know that the smooth relation $R_{\sim, \Quot_{Y(k)/\A^k}^{p^*\cV}}$ is quasi-compact and separated.

Now we have the following Cartesian diagram of closed or open embeddings, for $|I|=l\leq k$,
$$\xymatrix{
Y(l) \ar@{^(->}[rr]^-{\tau_{I,Y}} \ar[d]_p && Y(k)|_{U_I} \ar@{^(->}[rr]^-{\tilde\tau_{I,Y}} \ar[d] && Y(k) \ar[d]^p \\
\A^l \ar@{^(->}[rr]^-{\tau_I} && U_I \ar@{^(->}[rr]^-{\tilde\tau_I} && \A^k,
}$$
which induces the closed embedding $\Quot_{Y(l)/\A^l}^{p^*\cV} \hookrightarrow \Quot_{Y(k)/\A^k}^{p^*\cV}$, and the open embedding $(\C^*)^{k-l}\times \Quot_{Y(l)/\A^l}^{p^*\cV} \cong \Quot_{Y(k)|_{U_I}/U_I}^{p^*\cV} \hookrightarrow \Quot_{Y(k)/\A^k}^{p^*\cV}$. Hence we have the embeddings,
$$\xymatrix{
R_{\sim,\Quot_{Y(l)/\A^l}^{p^*\cV}} \ar@{^(->}[r] \ar@<-.4ex>[d] \ar@<.4ex>[d] & (\C^*)^{k-l} \times R_{\sim,\Quot_{Y(l)/\A^l}^{p^*\cV}} \ar@{^(->}[r] \ar@<-.4ex>[d] \ar@<.4ex>[d] & R_{\sim,\Quot_{Y(k)/\A^k}^{p^*\cV}} \ar@<-.4ex>[d] \ar@<.4ex>[d] \\
\Quot_{Y(l)/\A^l}^{p^*\cV} \ar@{^(->}[r] & (\C^*)^{k-l}\times \Quot_{Y(l)/\A^l}^{p^*\cV} \ar@{^(->}[r] & \Quot_{Y(k)/\A^k}^{p^*\cV}.
}$$

Let $\Quot_{Y(k)/\A^k}^{p^*\cV,st} \subset \Quot_{Y(k)/\A^k}^{p^*\cV}$ denote the subfunctor of the ordinary Quot-functor, consisting of stable quotients. By Proposition \ref{open_stable} this is an open subfunctor, thus represented by an algebraic subspace of the Quot-space. It is invariant under the equivalence relation.

Define the stable Quot-stack as
$$\fQuot_{\fY/\fA}^\cV:= \varinjlim \ \left[ \Quot_{Y(k)/\A^k}^{p^*\cV,st} \middle/ R_{\sim,\Quot_{Y(k)/\A^k}^{p^*\cV}} \right].$$
Similarly in the degeneration case one has the corresponding stable Quot-stack
$$\fQuot_{\fX/\fC}^\cV:= \varinjlim \ \left[ \Quot_{X(k)/\A^{k+1}}^{p^*\cV,st} \middle/ R_{\sim,\Quot_{X(k)/\A^{k+1}}^{p^*\cV}} \right],$$
where the transition maps in the directed system are again open immersions and the limits make sense.

We interpret the objects and morphisms of $\fQuot_{\fX/\fC}^\cV$ and $\fQuot_{\fY/\fA}^\cV$ in the categorical sense. For any scheme $S$, an object $(\bar\xi,\bar\phi)$ of $\fQuot_{\fY/\fA}^\cV (S)$ is represented by some object $(\xi,\phi)\in [ \Quot_{Y(k)/\A^k}^{p^*\cV,st} / R_{\sim,\Quot_{Y(k)/\A^k}^{p^*\cV}}] (S)$, for some $k$, possibly after some further \'etale base change; more precisely, passing to a surjective \'etale covering, we have the map $\xi: S_\xi=\coprod S_i \rar \A^k$ and a stable quotient $\phi: p^*\cV\rar \F$ on the associated family of expanded pairs $\cY_{S_\xi}$.

Different representations are compatible in the following way. Given another $(\xi',\phi')$ with $\xi': S_{\xi'}\rar \A^{k'}$, passing to the refinement $S_{\xi\xi'}:=S_\xi \times_S S_{\xi'}$, denote the two induced families by $\cY_{S_{\xi\xi'}}$ and $\cY_{S_{\xi\xi'}}'$. By construction they are isomorphic to each other via some $\sigma: \cY \xrar{\sim} \cY'$, by embedding into a larger ambient space $Y(k'')$ for some $k''\geq k, k'$ and pulling back the $\sim$ equivalence relation. Then the compatibility means that, passing to $S_{\xi\xi'}$, there is an isomorphism $\phi \xrar{\sim} \sigma^* \phi'$.

Given $f:T\rar S$, the map $\fQuot_{\fY/\fA}^\cV (S) \rar \fQuot_{\fY/\fA}^\cV (T)$ is defined by pull back. Suppose $(\bar\xi, \bar\phi)\in \fQuot_{\fY/\fA}^\cV (S)$. For a representative $(\xi,\phi_{S_\xi})$, we have the corresponding map $f_\xi: T_\xi:= T\times_S S_\xi \rar S_\xi$. The 1-arrow between the two objects is defined by the composite $\eta=\xi\circ f_\xi$ and the pull back $\cY_{T,\xi}=f_\xi^* \cY_{S,\xi}$, $\phi_{T,\xi}:= f^*\phi_{S,\xi}$.

The functor in the degeneration case can be defined in the same manner. For any $\A^1$-scheme $S$, $\fQuot_{\fX/\fC}^\cV (S)$ is defined as the set of all pairs $(\bar\xi, \bar\phi)$, where $\bar\xi \in \fC(S)$ and $\bar\phi$ is a stable quotient of sheaves on the family over $\bar\xi$.

\begin{thm}
  \begin{enumerate}[1)]
    \setlength{\parskip}{1ex}

    \item  $\fQuot_{\fY/\fA}^\cV$ is a Deligne--Mumford stack, locally of finite type;

    \item $\fQuot_{\fX/\fC}^\cV$ is a Deligne--Mumford stack, locally of finite type over $\A^1$.
  \end{enumerate}
\end{thm}

\begin{proof}
  We only prove for the relative case and the degeneration case is similar. The limit stack is defined by the stack associated to the limit groupoid, with objects and morphisms described as above. It remains to prove it is Deligne--Mumford.

  Look at the stabilizer of a fixed object $(\bar\xi,\bar\phi)\in \fQuot_{\fY/\fA}^\cV(S)$, which is represented by $\xi: S_\xi\rar \A^k$ and a stable quotient $\phi$ on $\cY_{S_\xi}$. The stabilizer of this representative is given in the following Cartesian diagram,
  $$\xymatrix{
  \Aut_\sim(\xi,\phi,k) \ar[r]\ar[d] & S_\xi \ar[r]\ar[d]^{\Delta\circ (\xi,\phi)} & S \\
  R_{\sim,\Quot_{Y(k)/\A^k}^{p^*\cV,st}} \ar[r] & \Quot_{Y(k)/\A^k}^{p^*\cV,st} \times \Quot_{Y(k)/\A^k}^{p^*\cV,st}.
  }$$
  Again one can see that for $k$ sufficiently large, this stabilizer does not depend on $k$, since the limit is taken over open immersions of stacks, which are representable. It also does not depend on the choice of $S_\xi$, since the lower horizontal map is affine, and by descent theory of affine morphisms we glue to obtain a group scheme over $S$. Denote this group scheme by $\Aut_\sim(\bar\xi,\bar\phi)$. It is the stabilizer of the object, independent of $k$ and choice of representatives for $k$ large. Moreover, $\Aut_\sim(\bar\xi,\bar\phi)$ is a quasi-compact and separated (actually affine) group scheme over $S$. Thus $\fQuot_{\fY/\fA}^\cV$ is an Artin stack.

  It remains to check that each $[ \Quot_{Y(k)/\A^k}^{p^*\cV,st}/ R_{\sim,\Quot_{Y(k)/\A^k}^{p^*\cV}}]$ is a Deligne--Mumford stack, locally of finite type. Recall that $R_{\sim,\Quot_{Y(k)/\A^k}^{p^*\cV}} \cong (\C^*)^k\times \Quot_{R_{d,Y(k)}/R_{d,\A^k}}^{p^*\cV}$. The $(\C^*)^k$-action has finite stabilizer due to the stability condition. On the other hand, $$\xymatrix{
  R_{d, \Quot_{Y(k)/\A^k}^{p^*\cV}} \ar@{^{(}->}[r] & S_k \times \Quot_{Y(k)/\A^k}^{p^*\cV} \ar@<.4ex>[r] \ar@<-.4ex>[r] & \Quot_{Y(k)/\A^k}^{p^*\cV}
  }$$
  is obviously \'etale. We conclude that the stabilizer is finite and the stack is Deligne--Mumford. Since each Quot-space $\Quot_{Y(k)/\A^k}^{p^*\cV}$ is of finite type, the stack we are considering is locally of finite type.
\end{proof}

\subsection{Projective Deligne--Mumford stacks}

From now on we need the notion of projective Deligne--Mumford stacks. The references for this are \cite{OS}, \cite{Kr} and \cite{Ni}. First we list some properties of coarse moduli spaces, which can be found in \cite{AOV}, \cite{AV} and \cite{Ni}. We collect these results here just for the later use. They are not stated in the full generality.

\begin{prop} \label{coarse_base_change}
  Let $W$ be a Deligne--Mumford stack, $\uW$ be a noetherian scheme and $c: W\rar \uW$ be a proper quasi-finite map.
  \begin{enumerate}[1)]
    \setlength{\parskip}{1ex}

    \item If $\uW$ is the coarse moduli space of $W$, and $\uW'\rar \uW$ is any morphism of schemes, then $\uW'$ is the coarse moduli space of $\uW'\times_{\uW} W$;

    \item If $\uW'\rar \uW$ is a flat surjective morphism of schemes, and $\uW'$ is the coarse moduli space of $\uW'\times_{\uW} W$, then $\uW$ is the coarse moduli space of $W$;
  \end{enumerate}
\end{prop}

\begin{prop} \label{tame_prop}
  Let $W$ be a separated Deligne--Mumford stack over $\C$ and $c:W\rar \uW$ be its coarse moduli space.
  \begin{enumerate}[1)]
    \setlength{\parskip}{1ex}

    \item There exists an \'etale covering $\coprod \uW_\alpha \rar \uW$, such that $W\times_{\uW} \uW_\alpha \cong [U_\alpha/\Gamma_\alpha]$, where each $U_\alpha$ is a scheme with finite group $\Gamma_\alpha$ acting on it;

    \item The map $c$ is proper. The functor $c_*$ carries quasi-coherent sheaves to quasi-coherent sheaves, coherent sheaves to coherent sheaves, and is exact;

    \item $c_*\Ou_W \cong \Ou_{\uW}$.
  \end{enumerate}
\end{prop}

Now we introduce the notion of generating sheaves, which can be understood as ``relatively ample" sheaves of a stack over its coarse moduli space.

\begin{defn}
  Let $W$ be a Deligne--Mumford stack over $\C$, with coarse moduli space $c: W\rar \uW$. A vector bundle $\E$ on $W$ is called a \emph{generating sheaf} on $W$, if for every quasi-coherent sheaf $\F$, the morphism $\theta_\E(\F): c^* c_* \mathcal{H}om(\E,\F) \otimes_{\Ou_W} \E \rar \F$, defined as the left adjoint of the map $c_* (\F\otimes_{\Ou_W} \E^\vee) \xrar{\Id} c_* (\F\otimes_{\Ou_W} \E^\vee)$, is surjective.
\end{defn}

\begin{prop}
  \begin{enumerate}[1)]
    \setlength{\parskip}{1ex}

    \item $\E$ is a generating sheaf if and only if for every geometric point of $W$, the representation of the stabilizer group of that point on the fiber of $\E$ contains every irreducible representation;

    \item Let $f: \uW'\rar \uW$ be any morphism of algebraic spaces, then $\uW'$ is the coarse moduli space of $W':= W\times_{\uW} \uW'$, with map $f': W'\rar W$. Suppose $\E$ is a generating sheaf on $W$, then $f'^*\E$ is a generating sheaf on $W'$.
  \end{enumerate}
\end{prop}

\begin{eg}[(Proposition 5.2 of \cite{OS})]
  Let $G$ be a finite group. A sheaf $\F$ on $BG$ is equivalent to a complex $G$-representation. Then $\F$ is a generating sheaf if and only if it contains every irreducible representations of $G$ as a $G$-submodule. In particular, the left regular representation $R$ of $G$ is a generating sheaf.

  More generally, let $U$ be a scheme with a $G$-action. Let $f:[U/G]\rar BG$ be the canonical 1-morphism. Then $f^*R$ is a generating sheaf on $[U/G]$.
\end{eg}

\begin{cor} \label{cor_gen}
  Let $W$ be a Deligne--Mumford stack over $\C$, with generating sheaf $\E$. Then
  \begin{enumerate}[1)]
    \setlength{\parskip}{1ex}

    \item $\E^\vee$ is still a generating sheaf;

    \item If $f: W'\rar W$ is a representable 1-morphism between Deligne--Mumford stacks, then $f^*\E$ is a generating sheaf on $W'$.
  \end{enumerate}
\end{cor}

\begin{proof}
  1) follows directly from 1) of the previous proposition. 2) follows from Frobenius reciprocity and the property of representable maps that homomorphisms between isotropy groups are monomorphic.
\end{proof}

\begin{defn}
  A Deligne--Mumford stack $W$ over $\C$ is called \emph{projective}, if it has a coarse moduli space $c: W\rar \uW$, where $\uW$ is a projective scheme, and it possesses a generating sheaf.

  Let $\pi: W\rar S$ be a Deligne--Mumford stack over $S$, with coarse moduli space $c: W\rar \uW$. Then $\pi: W\rar S$ is called a \emph{family of projective stacks} if the underlying map of schemes $\underline{\pi}: \uW \rar S$ is projective, and $W$ possesses a generating sheaf.
\end{defn}

There are some equivalent definitions of a projective Deligne--Mumford stack.

\begin{prop}[(Proposition 5.1 and Corollary 5.4 of \cite{Kr})] \label{Alt_defs_proj}
  Let $W$ be a proper Deligne--Mumford stack over $\C$, then the following are equivalent:
  \begin{enumerate}[1)]
    \setlength{\parskip}{1ex}

    \item $W$ is a projective Deligne--Mumford stack;

    \item $W$ has a projective coarse moduli space and is isomorphic to the quotient of a projective scheme by a reductive algebraic group acting linearly;

    \item $W$ has a projective coarse moduli space and every coherent sheaf admits a surjective morphism from a vector bundle;

    \item $W$ admits a closed embedding into a smooth proper Deligne--Mumford stack with projective coarse moduli space.
  \end{enumerate}
  In particular, a smooth Deligne--Mumford stack over $\C$ with projective coarse moduli space is always projective.
\end{prop}

The following concept first appeared in \cite{OS}, with the name ``generalized Hilbert polynomial", as an analog to the usual Hilbert polynomials on schemes to refine the ordinary Quot-spaces. We will also use this later, with a little modification.

\begin{defn}
  The \emph{Hilbert homomorphism} of a coherent sheaf $\F$ on $W$ with respect to $\E$ is the group homomorphism $P^\E_\F: K^0(W)\rar \Z$, defined as
  $$[V]\mapsto \chi \left( W, V\otimes_{\Ou_W} \F\otimes_{\Ou_W} \E^\vee \right),$$
  where $V$ is a vector bundle on $W$ and extended additively to $K^0(W)$.
\end{defn}

The following concept is introduced in \cite{Ni} to define the stability condition. We will also use it later for the boundedness.

\begin{defn}
  Let $\underline{H}$ be an ample line bundle on the projective coarse moduli space $\uW$, and $H=c^*\underline{H}$. The \emph{modified Hilbert polynomial} of $\F$ with respect to the ``polarization" $(\E,H)$ is the polynomial $P^{\E,H}_\F$ defined as $P^{\E,H}_\F (v):= P^\E_\F(H^{\otimes v})$.
\end{defn}

One of the reasons that we are interested in generating sheaves is the following property, the consequence of a series of propositions in \cite{Ni}.

\begin{prop} \label{prop_F_qF}
  Let $W$ be a projective Deligne--Mumford stack over $\C$, and $c: W\rar \uW$ be its coarse moduli space, together with a generating sheaf $\E$. Let $\F$ be a coherent sheaf on $W$. Then
  \begin{enumerate}[1)]
    \setlength{\parskip}{1ex}

    \item $c_*(\F\otimes_{\Ou_W} \E^\vee) =0$ if and only if $\F=0$;

    \item The sheaf $c_*(\F\otimes_{\Ou_W}\E^\vee)$ on $\uW$ has the same dimension with $\F$. Moreover, if $\F$ is pure, $\Supp c_*(\F\otimes_{\Ou_W}\E^\vee) = c (\Supp \F)$;

    \item $\F=0$ if and only if $P^\E_\F=0$, also if and only if the polynomial $P^{\E,H}_\F=0$.
  \end{enumerate}
\end{prop}

To end this section we give a description about the coarse moduli space of the bubble.

\begin{prop} \label{coarse_Delta}
  Let $(Y,D)$ be a locally smooth pair and $\Delta=\Pj(\Ou_D\oplus N_{D/Y})$ be the bubble. Let $\uD$, $\uDelta$ be the corresponding coarse moduli spaces. Then there is a $\Pj^1$-bundle $B$ over $\uD$, fitting into the following diagram
$$\xymatrix{
& \uDelta \ar[dr] & \\
\Delta \ar[rr] \ar[d] \ar[ur] & & B \ar[d] \\
D \ar[rr] & & \uD,
}$$
where the map $\Delta \to B$ is surjective, proper and quasi-finite, and $\uDelta\to B$ is finite.

In particular, $\uDelta$ is a projective scheme and $\Delta$ is projective as a Deligne--Mumford stack.
\end{prop}

\begin{proof}
  By definition $\Delta= \text{Proj}_{\Ou_D} (\bigoplus_{d=0}^\infty N_{D/Y}^d)$. By Lemma 2.1.2 of \cite{AGV}, there is a smallest integer $e$ such that $N_{D/Y}^e\cong c^*M$ for some line bundle $M$ on the coarse moduli space. Take
  $$B:= \text{Proj}_{\Ou_{\uD}} \left(\oplus_{d=0}^\infty N_{D/Y}^{de} \right) \cong \Pj_{\uD}(\Ou_{\uD} \oplus M).$$
  There is a proper, quasi-finite and surjective map $\Delta\rar B$ induced by the embedding of the graded algebras, which by the universal property of coarse moduli spaces, factors through $c: \Delta\rar \uDelta$. The induced map $\uDelta\rar B$ is also quasi-finite. By Proposition 2.6 of \cite{Vi}, there is a finite surjective map $A\rar \Delta$ where $A$ is a scheme. Apply Exercise 4.4 in \cite{Har} to the composition $A\rar \Delta\rar \uDelta$ over $B$ and we see that $\uDelta$ is proper, and hence finite over $B$.
\end{proof}

\subsection{Moduli of stable quotients with fixed topological data} \label{moduli of fixed data}

To refine the Quot-stacks and get some finite-type spaces, we need to fix some topological data. From now on, let $(Y,D)$ be a \emph{smooth} pair, with $Y$ \emph{projective}; and let $\pi: X\rar \A^1$ be a \emph{simple degeneration}, which is a \emph{family of projective Deligne--Mumford stacks}. Let $\E_X$, $\E_Y$ be generating sheaves on $X$ and $Y$, which we will just call $\E$ if there is no confusion.


Consider the Grothendieck K-group of $Y$. Note that by smoothness and 3) of Proposition \ref{Alt_defs_proj} we have $K_0(Y)\cong K^0(Y)$, just denoted by $K(Y)$ for simplicity. Let $p:Y[k]\rar Y$ be the contraction map, and $\F$ be a coherent sheaf on $Y[k]$. We can define a homomorphism $P^\E_\F: K(Y)\rar \Z$, still called the \emph{Hilbert homomorphism} with respect to $\E$, as follows,
$$[V]\mapsto \chi \left( Y[k], p^*V\otimes_{\Ou_W} \F\otimes_{\Ou_W} p^*\E^\vee \right),$$
where $V$ is a vector bundle and the homomorphism is naturally extended additively. Again, if one chooses an ample line bundle $\underline{H}$ on the coarse moduli space $c:Y\rar \underline{Y}$, one can define the \emph{modified Hilbert polynomial} as $P^{\E,H}_\F (v):= P^\E_\F(p^*H^{\otimes v})$, where $H=c^*\underline{H}$.

Note that properties in Proposition \ref{prop_F_qF} do not necessarily hold in this case. Although $p^*\E$ would still be a generating sheaf on $Y[k]$, the pull-back of $H$ may not be ample. But we can still use them to refine the Quot-stacks.

Consider a family of expanded pairs $\pi:\cY_S\rar S$, with a coherent sheaf $\F$ on $\cY_S$, flat over $S$. For each point $s\in S$, we have a Hilbert homomorphism on the fiber $P^\E_{\F|_s}: K(Y)\rar \Z$. By the same argument as Lemma 4.3 of \cite{OS}, if $S$ is connected, there exists a group homomorphism $P: K(Y)\rar \Z$, such that $P^\E_{\F_s}=P$, $\forall s\in S$.

Now fix $P:K(Y)\rar \Z$, and let $\fQuot_{\fY/\fA}^{\cV,P}$ be the subfunctor of $\fQuot_{\fY/\fA}^\cV$ parameterizing stable quotients with Hilbert homomorphism $P$. This is an open and closed subfunctor and thus represented by a Deligne--Mumford stack, locally of finite type.

For the expanded degenerations, we make the similar definition. The only subtle difference is that we need to composite with the restriction $K(X)\rar K(X_c)$ for $c\neq 0$ and $K(X)\rar K(X_0)$ for the Hilbert homomorphisms on fibers. Also for a fixed $P: K(X)\rar \Z$, one can define $\fQuot_{\fX/\fC}^{\cV,P}$, which is an open and closed subfunctor of $\fQuot_{\fX/\fC}^\cV$. As for the modified Hilbert polynomial, one needs to fix a relatively ample line bundle $H$ for $\pi: X\rar \A^1$.

\section{Properness of the moduli of 1-dimensional stable quotients}

In this section we prove the properness of the Quot-stacks $\fQuot_{\fY/\fA}^{\cV,P}$ and $\fQuot_{\fX/\fC}^{\cV,P}$, for a fixed generalized Hilbert polynomial $P$. Since our main interest is the application in Donaldson--Thomas theory and Pandharipande--Thomas theory, we concentrate on the 1-dimensional case. We will use the valuative criteria for separatedness and properness of Deligne--Mumford stacks. Let $(Y,D)$ and $\pi: X\rar \A^1$, $\E$, $H$ be fixed as in Subsection \ref{moduli of fixed data}.

The following lemmas justify what we mean by the 1-dimensional case, i.e. the dimension of supports of the stable quotients is preserved in family.

\begin{lem}
  Let $\F$ be an admissible sheaf on $Y[k]$ (resp. $X_0[k]$), with contraction map $p:Y[k]\rar Y$ (resp. $p: X_0[k]\rar X_0$). Then $\dim \F= \dim p_*\F$.
\end{lem}

\begin{proof}
  It suffices to prove the result for $k=1$ and the general case simply follows by induction. Let $\dim \F=d$. Then $\dim \F|_\Delta =d$ or $\dim \F|_Y=d$. If the latter is true, the lemma is proved.

  Now we assume that $\dim \F|_\Delta=d$ and $\dim p_*\F\leq d-1$. For this to be true, $\F|_\Delta$ must support along the fibers of $\Delta$. Hence $\dim \F|_D = \dim\F|_\Delta -1 = d-1$. On the other hand, by admissibility $(\F|_Y)|_D$ must have the same dimension as $\F|_D$, and therefore $\dim \F|_Y= \dim\F|_D +1 =d$. Contradiction. Thus $\dim p_*\F=d$, which proves the lemma.
\end{proof}

\begin{lem}[(Lemma 3.18 in \cite{LW})] \label{chi_p_*}
  Let $\Delta$ be a bubble component with divisor $D=D_-$ and contraction $p: \Delta \rar D$. Suppose $\F$ on $\Delta$ is normal to $D$, then $R^1 p_*\F=0$. Moreover, if $P^{\E,H}_\F=P^{\E,H}_{\F|_D}$, then $\F\cong p^*\F|_D$.
\end{lem}

\begin{proof}
  Choose a surjection $p^*\cV \rar \F\rar 0$, where some $\cV$ is a vector bundle on $D$ and let $\mathcal{K}$ be the kernel. Look at the long exact sequence
  $$\cdots \rar R^1p_*p^*\cV \rar R^1p_*\F \rar R^2 p_*\K \rar \cdots.$$
  It's clear that $R^1 p_*p^* \cV =0$, and $R^2 p_*\K=0$ by dimensional reasons. Thus $R^1 p_*\F=0$.

  Now suppose $P^{\E,H}_\F=P^{\E,H}_{\F|_D}$. By admissibility, we have the short exact sequence $0 \to \F(-D) \to \F \to \F|_D \to 0$. By $R^1 p_*\F(-D)=0$ we have the surjection $p_*\F \rar \F|_D \rar 0$. Since $H$ is ample on $D$, $P^{\E,H}_\F=P^{\E,H}_{\F|_D}$ implies that this is actually an isomorphism $p_*\F \cong \F|_D$. Take the pull back of the inverse map we have $p^*\F|_D \cong p^*p_*\F \rar \F$, which we can assume to be surjective if one replaces $\F$ by its twist with a sufficiently relatively ample line bundle at the beginning. Again by comparison of the modified Hilbert polynomials, one concludes that this map is also injective and $p^*\F|_D\cong \F$.
\end{proof}

\begin{cor}
  Let $\F$ be an admissible sheaf on $Y[k]$ or $X_0[k]$. Then the Hilbert homomorphism and modified Hilbert polynomial can be computed using $p_*\F$ on $Y$ or $X_0$, i.e. $P^\E_\F=P^\E_{p_*\F}$, and $P^{\E,H}_\F=P^{\E,H}_{p_*\F}$.
\end{cor}

\begin{proof}
This follows from $R^1 p_* \F =0$.
\end{proof}

\subsection{Boundedness}

In this subsection for a fixed group homomorphism $P: K(Y)\rar \Z$ (resp. $P: K(X)\rar \Z$), we require that the associate modified Hilbert polynomial $v\mapsto P(H^{\otimes v})$ has degree one or less. We denote it by $f(v)=av+b$, with $a\geq 0$. Consider Quot-stacks $\fQuot_{\fY/\fA}^{\cV,P}$ and $\fQuot_{\fX/\fC}^{\cV,P}$. By the previous two lemmas, one can see that the objects parameterized by them have $\dim \F = \deg f$. By boundedness we mean the quasi-compactness of the Quot-stacks.

\begin{prop}
  The stack $\fQuot_{\fY/\fA}^{\cV,P}$ (resp. $\fQuot_{\fX/\fC}^{\cV,P}$) is of finite type over $\C$ (resp. over $\A^1$).
\end{prop}

Recall that $\coprod_{k\geq 0} [\Quot_{Y(k)/\A^k}^{p^*\cV,st,P}/R_{\sim,\Quot_{Y(k)/\A^k}^{p^*\cV}}]$ forms an \'etale covering of $\fQuot_{\fY/\fA}^{\cV,P}$. It suffices to prove the following.

\begin{prop} \label{prop_boundedness}
   Fix a polynomial $f(v)=av+b$ with $a,b\in \Z$, $a\geq 0$, and a polarization $(\E,H)$ on $Y$ (resp. $X$). There exists a constant $N=N(f,\E,H)$, such that for any $k\geq 0$, and any stable quotient $\phi: p^*\cV\rar \F$ on $Y[k]$ (resp. $X[k]$), with modified Hilbert polynomial $P^{\E,H}_\F = f$, one has $k\leq N$.
\end{prop}

The rest of this subsection is to prove this proposition. We mainly concentrate on the relative case; for the degeneration case the proof will be similar. Let's first state some results that would be useful.

\begin{lem}[(Proposition 5.9.3 of \cite{CG}, Cororllary VI.2.3 of \cite{Ko})] \label{devissage}
  Let $W$ be an $n$-dimensional scheme over $\C$, with the topological filtration $0=\Gamma_{-1} \subset \Gamma_0 \subset \cdots \subset \cdots \Gamma_n=K_0(W)$, where $\Gamma_i$ consists of those classes generated by coherent sheaves of dimension $\leq i$. Let $\F$ be a coherent sheaf on $W$ with $\dim \Supp \F=d$. Then we have $[\F]\in \Gamma_d$, and
  $$[\F]= \sum_{Z\subset \Supp\F} \operatorname{mult}(\F;Z) [\Ou_Z]  \mod \Gamma_{d-1}$$
  in $K_0(W)$, where the sum is over all $d$-dimensional irreducible components of $\Supp\F$, and $\operatorname{mult}(\F;Z)$ is the length of $\F$ as a module over the generic point of $Z$.
\end{lem}

We need a result of Grothendieck on boundedness of families of sheaves.

\begin{lem}[(Lemma 2.6 of \cite{Gr})] \label{bounded_Gr}
  Let $W$ be a projective scheme over a noetherian scheme $S$, and $\Ou_W(1)$ be ample on $W$ relative to $S$. Let $E$ be a family of isomorphism classes of coherent sheaves on the fibers of $W/S$, contained in the family of isomorphism classes of the quotients of a certain fixed coherent sheaf on $W$, such that $\forall \F\in E$, the Hilbert polynomial of $\F$ is
  $$P_\F(v) = a_\F \frac{v^r}{r!} + b_\F \frac{v^{r-1}}{(r-1)!} + \cdots$$
  and $a_\F$ is bounded. Then $b_\F$ is bounded from below.

  Moreover, if $b_\F$ is bounded, then the family consisting of all $\F_{(r)}$ is bounded, where $\F_{(r)}:=\F/T_r(\F)$ is $\F$ quotient by its torsion subsheaf.
\end{lem}

\begin{proof}[Proof of Proposition \ref{prop_boundedness}]
  By Corollary \ref{normality_split}, given a stable quotient $\phi:p^*\cV \rar \F$ on $Y[k]$, we have the short exact sequence
  $$0\rar \F \rar \F|_Y\oplus \bigoplus_{i=1}^k \F|_{\Delta_i} \rar \bigoplus_{i=1}^k \F|_{D_{i-1}} \rar 0.$$
  Thus
  \begin{eqnarray} \label{split_chi}
    av+b &=& \chi (Y[k], \F \otimes p^*(H^v \otimes \E^\vee)) \nonumber \\
    &=& \chi(Y, \F|_Y \otimes H^v \otimes \E^\vee) + \sum_{i=1}^k \chi(\Delta_i, \F|_{\Delta_i}\otimes p^*(H^v\otimes \E^\vee)) \nonumber \\
    && - \sum_{i=1}^k \chi(D_{i-1}, \F|_{D_{i-1}}\otimes p^*(H^v\otimes \E^\vee)) \nonumber \\
    &=& \chi(Y, \F|_Y \otimes H^v \otimes \E^\vee) + \sum_{i=1}^k \chi(\Delta_i, \F|_{\Delta_i}(-D_{i-1})\otimes p^*(H^v\otimes \E^\vee)) \\
    &=:& (a_0(k,\F)v+b_0(k,\F)) + \sum_{i=1}^k ( a_i(k,\F)v+b_i(k,\F) ), \nonumber
  \end{eqnarray}
  where in (\ref{split_chi}) we use the exact sequence $0\rar \F|_{\Delta_i}(-D_{i-1}) \rar \F|_{\Delta_i} \rar \F|_{D_{i-1}} \rar 0$, by admissibility. Here the coefficients $a$'s and $b$'s also depend on the polarization $\E$, $H$, but we only emphasize that on $k$ and $\F$. Let
  $$\Lambda:=\{i \mid 1\leq i\leq k, a_i(k,\F)>0 \},$$
  and write
  \begin{equation} \label{av+b=}
  av+b = (a_0(k,\F)v+b_0(k,\F)) + \sum_{i\in \Lambda} ( a_i(k,\F)v+b_i(k,\F) ) + \sum_{i\not\in \Lambda} b_i(k,\F).
  \end{equation}
  By $a_i\geq 1$, $i\in \Lambda$ and $a=\sum_{i=0}^k a_i$ we can bound the size of $\Lambda$ by $|\Lambda|\leq a$.

  For $i\not\in \Lambda$, we know $b_i(k,\F)\neq 0$, because otherwise by Lemma \ref{chi_p_*} $\F|_{\Delta_i}$ would be a pull-back from $D_{i-1}$ and thus unstable. Thus $b_i(k,\F)\geq 1$.

  Now by the following Lemma \ref{bound_b} for $i\in \Lambda\cup \{0\}$, we have the lower bound $b_i(k,\F)\geq -M$. By (\ref{av+b=}) we have $b\geq (|\Lambda|+1)(-M) + k-|\Lambda|$. Thus $k\leq b+M +(M+1)|\Lambda| \leq b+M+(M+1)a$, which proves the proposition.
\end{proof}

\begin{lem} \label{bound_b}
  In the proof of the previous proposition, there exists $M=M(f,\E,H)>0$, which does not depend on $k$ and $\F$, such that $b_i(k,\F)\geq -M$, $\forall i\in \Lambda\cup\{0\}$.
\end{lem}

\begin{proof}
  Consider $i\in \Lambda$. Let $\Delta\cong \Delta_i$, with divisor $D_-\cong D_{i-1}$. By Proposition \ref{coarse_Delta} there is a diagram
  $$\xymatrix{
& \uDelta \ar[dr] & \\
\Delta \ar[rr]^{c_B} \ar[d]_p \ar[ur]^c & & B \ar[d]^{p_B} \\
D \ar[rr]^c & & \uD,
}$$
where the underlined are corresponding coarse moduli spaces, $B$ is a $\Pj^1$-bundle over $\uD$, and $\uDelta$ is finite over $B$.

  Let's compute the terms in (\ref{split_chi}). For simplicity let $\G_i:= (c_B)_*(\F|_{\Delta_i}(-D_{i-1})\otimes p^*\E^\vee)$, $1\leq i\leq k$ and $\G_0:= c_*(\F|_Y\otimes \E^\vee)$. Then
  \begin{eqnarray*}
  a_i(k,\F)v+b_i(k,\F) &=& \chi(\Delta,\F|_{\Delta_i}(-D_{i-1})\otimes p^*\E^\vee \otimes p^*H^v) \\
  &=& \chi(\uDelta, c_*(\F|_{\Delta_i}(-D_{i-1})\otimes p^*\E^\vee) \otimes \underline{p}^*\uH^v) \\
  &=& \chi(B, \G_i\otimes p_B^* \uH^v),
  \end{eqnarray*}
  where we used $p^*c^*=c^*\underline{p}^*$ and the projection formula.

  By Corollary \ref{cor_gen} we see that $(p^*\E^\vee)^\vee$ is still a generating sheaf on $\Delta$, and by 2) of Proposition \ref{prop_F_qF} we have $\dim \Supp \G_i =\dim \F|_{\Delta_i}$. Thus by Lemma \ref{devissage},
  $$a_i(k,\F)v = \sum_{Z\subset \Supp \G_i} \operatorname{mult}(\G_i;Z) \cdot \chi(Z, \Ou_Z \otimes p_B^*\uH^v) \mod \Gamma_0,$$
  where $Z$ ranges over the 1-dimensional irreducible components of $\Supp \G_i$. We see (by choosing a divisor for $\uH$ and computing the leading coefficients by exact sequence) that $\chi(Z,\Ou_Z\otimes p_B^* \uH^v)$ has leading term $vc_1(p_B^* \uH)\cdot [Z]$, where one can view $c_1(p_B^* \uH)\in N^1(B)$ and $[Z]\in N_1(B)$. Thus
  \begin{eqnarray} \label{a_i}
    a_i(k,\F) &=& c_1(p_B^* \uH)\cdot \sum_Z \operatorname{mult}(\G_i;Z) \cdot [Z] \\
    &=& : c_1(p_B^* \uH)\cdot [\G_i], \nonumber
  \end{eqnarray}
  where $[\G_i]$ is the numerical 1-cycle determined by the support of $\G_i$.

  However, $p_B^* \uH$ is not ample on $B$, and the above is not a Hilbert polynomial, as required by Lemma \ref{bounded_Gr}. We need to introduce some ample line bundles on $B$. Replace $\uH$ by its sufficiently large power, such that $p_B^*\uH\otimes \Ou_{B}(\uD)$ is ample on $B$. Consider $h:=c_B^*(p_B^*\uH\otimes \Ou_{B}(\uD)) = H\otimes c_B^*\Ou_{B}(\uD)$ on $\Delta$. Note that this procedure does not depend on $k$.

  Now let's compute the modified Hilbert polynomial with respect to $h$. Let $a'_i(k,\F)v+b'_i(k,\F):= \chi(\Delta,\F|_{\Delta_i}(-D_{i-1}) \otimes p^*\E^\vee \otimes h^v)$, and $a'_0v+b'_0:=a_0v+b_0$. By the similar computation we arrive at the following
  \begin{equation*}
    a'_i(k,\F) = (c_1(p_B^* \uH)+[\uD])\cdot [\G_i].
  \end{equation*}

  Combining (\ref{a_i}) we get
  \begin{equation} \label{a_i+}
  a'_i(k,\F) = a_i(k,\F)+ [\uD]\cdot [\G_i] \leq a+[\uD]\cdot [\G_i].
  \end{equation}

  Now the lemma follows from the following claim and Lemma \ref{bounded_Gr}. A uniform upper bound for all $a'_i$'s would imply the existence of a uniform lower bound for all $b'_i$'s. Since $b'_i = b_i$ by Riemann--Roch computation, we get a uniform lower bound for all $b_i$'s.

  \textbf{Claim}: there is a uniform upper bound for $a_i'$, $\forall i$.

  We prove this by induction. For $i=0$, $a'_0(k,\F)=a_0(k,\F)\leq a$. Then analogous to (\ref{a_i}) we have
  \begin{equation}
    a_0(k,\F) = c_1(\uH)\cdot [\G_0].
  \end{equation}
  Since $\uH$ is ample on $\uY$, the boundedness of $a_0(k,\F)$ implies that $\{ [\G_0] \mid k, \F \}$ is bounded as a subset in $N_1(\uY)$. Thus the pairing $[\uD]\cdot [\G_0]$ is bounded, say by some constant $C_0>0$. On the other hand, we have
  $$[\uD]\cdot [\G_0] = [\uD]\cdot [\G_1],$$
  which follows from definition and restrictions of $\F$ to $D\cong D_0$ from $Y$ and $\Delta_1$ respectively. Thus by (\ref{a_i+}) we have
  $$a'_1(k,\F) -a_1(k,\F) \leq C_0.$$
  Now one can proceed by induction. Assume that $a_i'(k,\F)-a_i(k,\F)\leq C_i$ for some uniform bound $C_i$. Then one has $\{[\G_i] \mid k,\F\}$ is bounded as a subset in $N_1(\uDelta)$. By admissibility
  $$[\uD]\cdot [\G_{i+1}] = [\uD_+]\cdot [\G_i],$$
  is bounded by some uniform bound $C_{i+1}$.

  Now we look at $i+1$. If $i+1\in \Lambda$, then we have
  $$a'_{i+1}(k,\F)-a_{i+1}(k,\F)= [\uD]\cdot [\G_{i+1}] \leq C_{i+1}$$
  is bounded.

  If $i+1\not\in \Lambda$, then $\G_{i+1}$ is of dimension 0 in $\uDelta_{i+1}$, which by admissibility does not intersect with the divisors. Thus in this case one has
  $$a'_{i+1}(k,\F)-a_{i+1}(k,\F)=0.$$

  Keep proceeding by induction, and finally we have the set of numbers $\{a'_i(k,\F)-a_i(k,\F) \mid 0\leq i\leq k \}$ is bounded by $C:= \max \{C_i \mid 0\leq i\leq k\}$, where there are only $|\Lambda|$ many of nonzero $C_i$'s, and each nonzero one is determined by the previous nonzero one in some way independent of $k$ and $\F$ (The process only depends on the intersections and bounds on $N_1(B)$). Thus $a+C$ does not depend on $k$ and $\F$ (but depends on the choice of $\uH$, $\E$) and bounds all $a'_i(k,\F)$.
\end{proof}

\subsection{Separatedness}

In this section we prove the following.

\begin{prop} \label{prop_sep}
  The stack $\fQuot_{\fY/\fA}^{\cV,P}$ (resp. $\fQuot_{\fX/\fC}^{\cV,P}$) is separated over $\C$ (resp. over $\A^1$).
\end{prop}

We use the valuative criterion of separatedness for Deligne--Mumford stacks. Consider $S=\Sp R$, where $R$ is a valuation ring, with fractional field $K$. Let $\eta=\Sp K$ be the generic point and $\eta_0=\Sp k$ be the closed point. Since we have proved the Quot-stack is of finite type, by Proposition 7.8 in \cite{LMB}, it suffices to treat the case where $R$ is a complete discrete valuation ring with an algebraically closed residue field $k$. Let $u\in R$ be the uniformizer. Again we concentrate on the relative case; the proof for the degeneration case would be similar.

\begin{rem}
  By our description of the objects we need to consider representatives $\xi: S_\xi\rar \A^k$, where $S_\xi:=\coprod S_i \rar S$ is an \'etale covering of $S$. Since $R$ is a discrete valuation ring, for each $S_i\rar S$ to be \'etale, it must be an isomorphism. Therefore we can always assume $S_\xi=S$.
\end{rem}

Suppose that we are given two arrows $(\bar\xi,\bar\phi), (\bar\xi',\bar\phi'): S\rar \fQuot_{\fY/\fA}^{\cV,P}$, represented by $(\xi,\phi)$, $(\xi',\phi')$, whose restrictions to $\eta$ are isomorphic. $\xi: S\rar \A^k$, $\xi': S\rar \A^{k'}$ are of the following form
$$\C [t_1,\cdots,t_k] \rar  R, \quad t_i\mapsto c_i(u)u^{e_i},$$
$$\C [t_1,\cdots,t_{k'}] \rar  R, \quad t_i\mapsto c'_i(u)u^{e'_i},$$
where $c_i(u),c'_i(u)$ are either 0 or invertible in $R$, and $e_i,e'_i \geq 0$. Embedding them into a larger target and applying the $\sim$ equivalence relation, we can assume that $k=k'$ and the maps are of the following form
$$(t_1,\cdots,t_k) \mapsto (u^{e_1},\cdots,u^{e_l},0,\cdots, 0),$$
$$(t_1,\cdots,t_{k'}) \mapsto (u^{e'_1},\cdots,u^{e'_l},0,\cdots, 0),$$
where the two expressions have the same number of zero's, due to the isomorphism over $\eta$. The proof of the separatedness then reduces to the following lemma.

\begin{lem}
  Consider the map $\xi: S\rar \A^1$, given by $\C[t_1]\rar R$, $t_1\mapsto u^e$, $e\geq 1$. Let $\tilde\xi: S\rar \A^e$ be $\C[t_1,\cdots,t_e]\rar R$, $(t_1,\cdots,t_e)\mapsto (u,\cdots,u)$. Let $\cY_S$, $\tilde\cY_S$ be the corresponding associated families of expanded pairs. Then there is a map $h: \tilde\cY_S\rar \cY_S$, which is an isomorphism on $\eta$, and a contraction $Y[e]\rar Y[1]$ on $\eta_0$, where the only un-contracted component is $\Delta_e$.
\end{lem}

\begin{proof}
  This is due to the construction of $X(k)$. By the following diagram and the universal property of fiber product, we get the map $h$,
  $$\xymatrix{
  \tilde\cY_S \ar[r] \ar[d] & Y(e) \ar[r]^p \ar[d] & Y(1) \ar[d] \\
  S \ar[r]^{\tilde\xi} & \A^e \ar[r] & \A^1.
  }$$
  The map on the base is multiplication and the map on the central fiber is the contraction of the inserted components.
\end{proof}

\begin{proof}[Proof of Proposition \ref{prop_sep}]
  We need to prove that the isomorphism on $\eta$ extends to $S$. Let $E:=\sum_{i=1}^l e_i$ and $E':=\sum_{i=1}^l e_i'$. First we consider the special case $E=E'$.

  Consider the map $\tilde\xi: S\rar \A^{E+k-l}$, given by $(t_1,\cdots,t_E,\cdots,t_{E+k-l})\mapsto (u,\cdots, u,0,\cdots,0)$. Repeatedly applying the previous lemma, there is a map between associated families $h: \tilde\cY_S\rar \cY_S$, such that $h|_\eta$ is an isomorphism and $h|_{\eta_0}: Y[E+k-l]\rar Y[k]$ is a contraction of components, where the un-contracted components are $\Delta_{e_1}, \Delta_{e_1+e_2}, \cdots, \Delta_E$ and the last $(k-l)$ one's corresponding to the 0's. Applying the same with $\xi'$, we also have a map $h': \tilde\cY_S\rar \cY'_S$ contracting the corresponding components for $\xi'$.

  Consider quotient sheaves $h^*\phi$ and $h'^*\phi'$ on $\tilde\cY_S$. Now $(\tilde\xi,h^*\phi)$ and $(\tilde\xi,h'^*\phi')$ are isomorphic over $\eta$ (since $h|_\eta$, $h'|_\eta$ are isomorphisms), and flat over $S$ by admissibility and Lemma \ref{normality_simple}. By the separatedness of the ordinary Quot-space $\Quot_{Y(E+k-l)/\A^{E+k-l}}^{p^*\cV}$, they are also isomorphic over $S$. On the other hand, $h^*\phi|_{\eta_0}$ and $h'^*\phi'|_{\eta_0}$ are stable on un-contracted components of $Y[E+k-l]$, but are unstable on contracted one's. For them to be isomorphic, the contracted components for $h$ and $h'$ must coincide, which implies $e_i=e_i'$, $\forall 1\leq i\leq l$, in which case $h$, $h'$ must be isomorphisms and the conclusion follows.

  For the general case, assume $E\leq E'$. We take the embeddings
  $$\tilde\xi: S\rar \A^{E+k-l}, \quad (t_1,\cdots,t_E,\cdots, t_{E+k-l})\mapsto (u,\cdots, u,0,\cdots,0),$$
  $$\tilde\xi': S\rar \A^{E'+k-l}, \quad (t_1,\cdots,t_E,\cdots,t_{E'},\cdots,t_{E'+k-l})\mapsto (u,\cdots, u,u,\cdots,u,0,\cdots,0),$$
  with corresponding associated families $\tilde\cY_S$, $\tilde\cY'_S$. By the successive blow-up construction (Proposition \ref{succ_bl}), there is a map $c: \tilde\cY_S'\rar \tilde\cY_S$, which is a contraction of components if restricted to $\eta_0$. $c$ is given by the following diagram,
  $$\xymatrix{
  \tilde\cY'_S \ar[r] \ar[d] & Y(E'+k-l) \ar[r]^{p_{\Bl}} \ar[d] & Y(E+k-l) \ar[d] \\
  S \ar[r]^-{\tilde\xi'} & \A^{E'+k-l} \ar[r]^{\pr} & \A^{E+k-l},
  }$$
  where $p_{\Bl}$ is the contraction map for the successive blow-up construction, and the projection $\pr: \A^{E'+k-l}\rar \A^{E+k-l}$ is to forget the $t_{E+1}, \cdots, t_{E'}$ components. Again by the separatedness of the ordinary Quot-space, we have an isomorphism $c^*h^*\phi\cong h'^*\phi'$. Compare the stable components of the two quotients and one obtains the conclusion $e_i=e'_i$, $\forall 1\leq i\leq l$.
\end{proof}

\subsection{Numerical criterion for admissibility}

Before getting into the proof of properness, we need a numerical criterion for a sheaf to be admissible. Let $\F$ be a coherent sheaf on $Y[k]$ (resp. $X_0[k]$), with generating sheaf $\E$ on $Y$ and a line bundle $H$ which is a pull back of an ample line bundle from the coarse moduli space. Let $\I_i^-$ and $\I_i^+$ be the ideal sheaf of the divisors $D_{i-1},D_i \subset \Delta_i$, and $\I_0^+$ be the ideal sheaf of $D_0\subset Y$. Let $\J_i$ be the ideal sheaf of $D_i\subset Y[k]$. Let $\F^\tf$ be the quotient in the following exact sequence,
$$\xymatrix{
0\ar[r] & \bigoplus_{i=0}^{k-1} \F_{\J_i} \ar[r] & \F \ar[r] & \F^\tf \ar[r] & 0.
}$$

\begin{defn}
  Define the $i$-th \emph{error} of $\F$ to be the following polynomial (here we write the subscript in parenthesis, making them easy to read).

  In the relative case:
  \begin{eqnarray*}
  \Err^{\E,H}_i(\F)(v):
  &=& P^{\E,H}(\F_{\J_i})(v) + P^{\E,H}(\F^\tf|_{\Delta_i,\I_i^+})(v)+ P^{\E,H}(\F^\tf|_{\Delta_{i+1},\I_{i+1}^-})(v) \\
  && - \frac{1}{2} P^{\E,H}((\F^\tf|_{\Delta_i,\I_i^+})|_{D_i}) - \frac{1}{2} P^{\E,H}((\F^\tf|_{\Delta_{i+1},\I_{i+1}^-})|_{D_i}),
  \end{eqnarray*}
  for $0\leq i\leq k-1$, and
  $$\Err^{\E,H}_k(\F)(v):= P^{\E,H}(\F^\tf|_{\Delta_k,\I_k^+})(v),$$
  where we denote $\Delta_0:=Y$ and $\F^\tf|_{\Delta_i,\I_i^+}:= (\F^\tf|_{\Delta_i})_{\I_i^+}$.

  In the degeneration case:
  \begin{eqnarray*}
  \Err^{\E,H}_i(\F)(v):
  &=& P^{\E,H}(\F_{\J_i})(v) + P^{\E,H}(\F^\tf|_{\Delta_i,\I_i^+})(v)+ P^{\E,H}(\F^\tf|_{\Delta_{i+1},\I_{i+1}^-})(v) \\
  && - \frac{1}{2} P^{\E,H}((\F^\tf|_{\Delta_i,\I_i^+})|_{D_i}) - \frac{1}{2} P^{\E,H}((\F^\tf|_{\Delta_{i+1},\I_{i+1}^-})|_{D_i}),
  \end{eqnarray*}
  for $0\leq i\leq k$, where we denote $\Delta_0:=Y_-$ and $\Delta_{k+1}:=Y_+$.

  Moreover, we define the total error as
  $$\Err^{\E,H}(\F)(v):= \sum_{i=0}^k \Err_i^{\E,H}(\F)(v).$$
\end{defn}

For simplicity we will just omit the $\E,H$ and write $\Err_i(\F)$ and $\Err(\F)$. Here the polynomials are actual Hilbert polynomials and there is no problem with the ampleness of the pull back of $H$, since the definition only involves the restriction of $H$ to $D_i$. Hence $\F$ is admissible if and only if $\Err(\F)=0$. In other words, the error polynomial measures the failure for a sheaf to be admissible.

For $\Delta_i$, $0\leq i\leq k$, let $(\F^\tf|_{\Delta_i})^\tf$ be the quotient in the following exact sequences,
$$\xymatrix{
0 \ar[r] & \F^\tf|_{\Delta_i,\I_i^-} \oplus \F^\tf|_{\Delta_i,\I_i^+} \ar[r] & \F^\tf|_{\Delta_i} \ar[r] & (\F^\tf|_{\Delta_i})^\tf \ar[r] & 0, \quad 1\leq i\leq k;
}$$
$$\xymatrix{
0 \ar[r] & \F^\tf|_{Y,\I_0^+} \ar[r] & \F^\tf|_Y \ar[r] & (\F^\tf|_Y)^\tf \ar[r] & 0, \quad i=0.
}$$
It is clear that $(\F^\tf|_{\Delta_i})^\tf$ are admissible. Restrict to $D_i$ and we get the following sequences,
$$\xymatrix{
0 \ar[r] & ( \F^\tf|_{\Delta_i,\I_i^-} ) |_{D_{i-1}} \ar[r] & \F^\tf|_{D_{i-1}} \ar[r] & (\F^\tf|_{\Delta_i})^\tf |_{D_{i-1}} \ar[r] & 0,
}$$
$$\xymatrix{
0 \ar[r] & ( \F^\tf|_{\Delta_i,\I_i^+} ) |_{D_i} \ar[r] & \F^\tf|_{D_i} \ar[r] & (\F^\tf|_{\Delta_i})^\tf |_{D_i} \ar[r] & 0.
}$$
Let
$$\delta_i^-(\F):= P((\F^\tf|_{\Delta_i})^\tf) - P((\F^\tf|_{\Delta_i})^\tf|_{D_{i-1}}), \ 1\leq i \leq k,$$
$$\delta_i^+(\F):= P((\F^\tf|_{\Delta_i})^\tf) - P((\F^\tf|_{\Delta_i})^\tf|_{D_i}), \ 0\leq i\leq k-1,$$
and set
$$\delta_0^-(\F):=P((\F^\tf|_Y)^\tf), \quad \delta_k^+(\F):= P((\F^\tf|_{\Delta_k})^\tf).$$


We have the following lemma.

\begin{lem} \label{identity_delta}
  \begin{enumerate}[1)]
  \setlength{\parskip}{1ex}

  \item $P_\F = \Err(\F) + \frac{1}{2} \sum_{i=0}^k (\delta_i^-(\F) + \delta_i^+(\F));$

  \item For $1\leq i\leq k-1$, the leading coefficients of $\delta_i^\pm$ are non-negative. They vanish if and only if $(\phi|_{\Delta_i})^\tf$ is $\C^*$-equivariant, and also if and only if $(\phi|_{\Delta_i})^\tf$ is a pull-back from $D_i$ or $D_{i-1}$.
  \end{enumerate}
\end{lem}

\begin{proof}
  2) simply follows from Lemma \ref{chi_p_*}. For 1), let's compute the modified Hilbert polynomial. Recall that Lemma \ref{J_split} implies $\F^\tf = \ker(\bigoplus_{i=0}^k \F^\tf|_{\Delta_i} \rar \bigoplus_{i=0}^{k-1} \F^\tf|_{D_i})$. Thus
  \begin{eqnarray*}
    P(\F)
    &=& \sum_{i=0}^{k-1} P(\F_{\J_i}) + \sum_{i=0}^k P(\F^\tf|_{\Delta_i}) - \frac{1}{2} \sum_{i=1}^{k} P(\F^\tf|_{D_{i-1}}) - \frac{1}{2} \sum_{i=0}^{k-1} P(\F^\tf|_{D_i}) \\
    &=& \sum_{i=0}^{k-1} P(\F_{\J_i}) + \sum_{i=1}^k P(\F^\tf|_{\Delta_i,\I_i^-})+ \sum_{i=0}^k P(\F^\tf|_{\Delta_i,\I_i^+}) + \sum_{i=0}^k P((\F^\tf|_{\Delta_i})^\tf) \\
    && - \frac{1}{2} \sum_{i=1}^k P((\F^\tf|_{\Delta_i,\I_i^-})|_{D_{i-1}}) - \frac{1}{2} \sum_{i=0}^{k-1} P((\F^\tf|_{\Delta_i,\I_i^+})|_{D_i})
     - \frac{1}{2} \sum_{i=1}^k P((\F^\tf|_{\Delta_i})^\tf|_{D_{i-1}}) \\
    && - \frac{1}{2} \sum_{i=0}^{k-1} P((\F^\tf|_{\Delta_i})^\tf|_{D_i}).
  \end{eqnarray*}
  By definition we have
  \begin{eqnarray*}
    \Err(\F)
    &=& \sum_{i=0}^{k-1} P(\F_{\J_i}) + \sum_{i=0}^k P(\F^\tf|_{\Delta_i,\I_i^+})+ \sum_{i=1}^k P(\F^\tf|_{\Delta_i,\I_i^-}) - \frac{1}{2} \sum_{i=0}^{k-1} P((\F^\tf|_{\Delta_i,\I_i^+})|_{D_i}) \\
    && - \frac{1}{2} \sum_{i=1}^k P((\F^\tf|_{\Delta_i,\I_i^-})|_{D_{i-1}}).
  \end{eqnarray*}
  Then
  \begin{eqnarray*}
    P(\F)-\Err(\F)
    &=& \sum_{i=0}^k P((\F^\tf|_{\Delta_i})^\tf) - \frac{1}{2} \sum_{i=1}^k P((\F^\tf|_{\Delta_i})^\tf|_{D_{i-1}}) - \frac{1}{2} \sum_{i=0}^{k-1} P((\F^\tf|_{\Delta_i})^\tf|_{D_i}) \\
    &=& \frac{1}{2} \sum_{i=0}^k (\delta_i^-(\F)+ \delta_i^+(\F)).
  \end{eqnarray*}
\end{proof}

To compare the polynomials, we introduce the following order in the set of Hilbert polynomials. For two polynomials $f(v)= a_r \frac{v^r}{r!} + \cdots$ and $g(v)= b_s \frac{v^s}{s!} + \cdots$, where $a_r,b_s\in \Z_+$. We say $f\prec g$, if $r<s$ or $r=s$ and $a_r<b_s$; $f \preceq g$, if ``$<$" is replaced by $\leq$. Note modified Hilbert polynomials are actually of this kind, since they can be viewed as usual Hilbert polynomials on coarse moduli spaces. Another observation is that under this order the space of Hilbert polynomials satisfies the strict descending condition: any strictly descending chain $f_1\succ f_2\succ \cdots$ must attains 0 at some finite step.

\subsection{Properness}

In this subsection we prove the properness.

\begin{thm} \label{thm_properness}
  $\fQuot_{\fY/\fA}^{\cV,P}$ (resp. $\fQuot_{\fX/\fC}^{\cV,P}$) is a complete Deligne--Mumford stack over $\C$ (resp. over $\A^1$). In particular, if the associated modified Hilbert polynomial $P(H^{\otimes v})$ has degree 0 or 1, the stack is proper (resp. over $\A^1$).
\end{thm}

\begin{rem}
  Here by complete we mean that it satisfies the valuative criterion for properness, but is not necessarily quasi-compact. In other words, it is separated and universally closed. We expect the ``dimension$\leq$1" condition to be superfluous and the properness is true for moduli of stable quotients in any dimensions, which is treated in \cite{LW}, but for simplicity we just pose this assumption and restrict to the 1-dimensional case.
\end{rem}

Again let $S=\Sp R$, where $R$ is a complete discrete valuation ring with uniformizer $u\in R$, generic point $\eta=\Sp K$ and closed point $\eta_0=\Sp k$. As remarked in the separatedness part, we also work on $S$ directly instead of passing to an \'etale covering. Let $(\bar\xi_\eta, \bar\phi_\eta): \eta \rar \fQuot_{\fY/\fA}^{\cV,P}$ be an object of $\fQuot_{\fY/\fA}^{\cV,P} (\eta)$, represented by some map $\xi_\eta: \eta\rar \A^k$, and a stable quotient $\phi_\eta: p^*\cV\rar \F_\eta$ on $\cY_\eta$. By the valuative criterion, it suffices to extend this object over $\eta_0$ after some finite base change of $S$.

Applying the $\sim$ equivalence relation, one can always make $\xi_\eta$ in the following form
$$\xi_\eta: \eta\rar \A^k, \quad \C[t_1,\cdots,t_k] \rar K,$$
$$(t_1,\cdots,t_k) \mapsto (1,\cdots,1,0,\cdots,0),$$
where $t_i\mapsto 1$ for the first $l$ (possibly $0$) coordinates. Note that if $l> 0$, then $\xi_\eta$ actually factors through the standard embedding $\{(1,\cdots,1)\}\times \A^{k-l} \hookrightarrow \A^k$. In this case we can pull everything back to $\A^{k-l}$ and work on $\A^{k-l}$ instead of $\A^k$. Thus we can assume that $l=0$.

Now for this $\xi_\eta$ there is a naive extension $\xi: S\rar \A^k$, which maps $S$ constantly to $0\in \A^k$. Then $\cY_S\cong S\times Y[k]$, and by the completeness of the ordinary Quot-space $\Quot_{Y[k]}^{p^*\cV}$ (Theorem 1.1 of \cite{OS}), $\phi_\eta$ extends to some $\phi$ over $S$. If $\phi|_{\eta_0}$ is stable, then we are done; otherwise, we need to modify our extension $\xi$. We'll see that a good modification is always available.

\textbf{Step 1.} Normal to the distinguished divisor.

\begin{lem} \label{normal to dist divisor}
    Let $\xi: S\rar \A^k$ be a map and $\phi: p^*\cV\rar \F$ be a quotient on $\cY_S$, flat over $S$, such that $\phi|_\eta$ is admissible. Then there exists $(\xi',\phi')$, with $\xi': S\rar \A^{k'}$ and $S$-flat quotient $\phi'$ on $\cY'_S$, such that $(\xi',\phi')|_\eta \cong (\xi,\phi)|_\eta$, and $\phi'$ is normal to the distinguished divisor $\cD'_S\subset \cY'_S$.
  \end{lem}

  \begin{proof}
    Consider $D[k]\subset \cY_{\eta_0}\cong Y[k]$. If $\phi|_{\eta_0}$ is normal to $D[k]\times \eta_0$, then we are done. Suppose this is not the case. Let's look at what this means in a local chart. \'Etale locally near a point in $D[k]\times \eta_0$, the local model of $\cY_S$ can be taken as $U= \Sp A:= \Sp R[y,\vec{z}]$, where $(D[k]\times \eta_0)|_U$ is defined by the ideal $(u,y)$. Geometrically, $\vec z$ stands for coordinates in $D$, and $y=0$ is the local defining equation for $D$.

    The quotient $\phi$ is represented by a sequence $0\rar K\rar A^{\oplus r} \rar M\rar 0$. Assume that $K=(f_1(u,y),\cdots,f_m(u,y))$, with generators
    \begin{equation} \label{generator_K}
    f_i(u,y)=c_i + u^{\alpha_i} g_i(u) + y h_i(u,y) \in A^{\oplus r},
    \end{equation}
    where $c_i\in k[\vec{z}]^{\oplus r}$, $g_i\in k[\vec{z}]^{\oplus r}[u]$, $h_i\in k[\vec{z}]^{\oplus r}[u,y]$, $g_i(0)\neq 0$ and $\alpha_i \geq 1$. By generic normality, $c_i+u^{\alpha_i} g_i(u)\neq 0$. Assume that these generators are minimal in the sense that $m$ is as small as possible. One can easily observe that $M$ is flat over $R$ if and only if $\forall i$, $u\nmid f_i(u,y)$, i.e. $c_i + y h_i(0,y)\neq 0$.

    Restricting to $\eta_0$ and using flatness, $\phi|_{\eta_0}$ is given by the sequence $0\rar K_0\rar A_0^{\oplus r}\rar M_0\rar 0$, with $K_0=(f_1(0,y),\cdots,f_m(0,y))$. The subscript 0 here means restriction to $\eta_0$. By our assumption, $M_0$ is not normal to $D[k]$, which is equivalent to $y\mid f_i(0,y)$, i.e. $c_i= 0$, for some $i$.

    To modify this family, we apply successive blow-ups at the divisor $D[k]\times \eta_0\subset \cY_S$; in other words, we take the modified family as
    $$\cY'_S \cong \Bl_{D[k]\times \eta_0} \cY_S.$$
    This is still a family of expanded pairs, as it fits into the diagram
    $$\xymatrix{
    \cY'_S \ar[r] \ar[d] & Y(k+1) \ar[r]^-{p_{\Bl}} \ar[d] & Y(k) \ar[d] \\
    S \ar[r]^-{\xi'} & \A^{k+1} \ar[r]^-{\pr_{1,\cdots,k}} & \A^k,
    }$$
    where $\xi'$ is given by $(t_1,\cdots,t_k,t_{k+1})\mapsto (u^{e_1},\cdots, u^{e_k},u)$.

    The the local model of $\cY'_S$ is the blow up of $\Sp R[y,\vec{z}]$ at the ideal $(u,y)$. Thus \'etale locally around a point in the new distinguished divisor $D[k+1]\times \eta_0\subset \cY'_S$, the local model is $\Sp B:= \Sp R[w,\vec{z}]$, with $y=uw$.

    By completeness of the ordinary Quot-space, $\phi|_\eta$ extends to an $S$-flat quotient $\phi'$ on the new family. Locally $\phi'$ is given by a sequence $0\rar K'\rar B^{\oplus r}\rar M'\rar 0$. By construction of flat limit one has
    $$K'=K_u \cap B^{\oplus r},$$
    where we view $A\subset B\subset B_u = A_u$ as submodules. Then $$K'=(f_1(u,uw),\cdots,f_m(u,uw))A_u^{\oplus r} \cap B^{\oplus r},$$
    $$f_i(u,uw)= u^{\alpha_i} g_i(u) + uw h_i(u,uw),$$
    from which we see that $u\mid f_i(u,uw)$.

    Take $\gamma_i:= \min \{ \alpha_i, \text{ord} (h)+1 \} \geq 1$, where $\text{ord}(h)$ is the minimal degree of the monomials of $h$ in $u$ and $y$. We have
    $$f_i(u,uw)= u^{\gamma_i} \left( u^{\alpha_i-\gamma_i} g_i(u) + w \frac{h_i(u,uw)}{u^{\gamma_i-1}} \right)= u^{\gamma_i} \left( u^{\alpha_i-\gamma_i} g_i(u) + w h'_i(u,w) \right),$$
    where $h'_i$ is some polynomial with $u\nmid \left( u^{\alpha_i-\gamma_i} g_i(u) + w h'_i(u,w) \right)$.

    After localizing at $u$ and intersection with $B^{\oplus r}$, we may take
    $$f'_i(u,w):= u^{\alpha'_i} g_i(u) + w h'_i(u,w)$$
    as the generators of $K'$, where $\alpha'_i:=\alpha_i-\gamma_i <\alpha_i$.

    If $\alpha'_i>0$ for some $i$, then $c'_i$ as similarly defined in (\ref{generator_K}) for $f'_i$ is still 0. We can  repeat the procedure above to further decrease $\alpha'_i$. After finitely many steps we get all $\alpha'_i$ to be 0, i.e. all $c'_i$ are nonzero; in other words, $\phi'|_{\eta_0}$, after successive blow-ups, would be normal to $D[k]\times \eta_0$.
  \end{proof}

\textbf{Step 2.} Admissibility. The crucial lemma in this step is the following.

\begin{lem} \label{key_lem_properness}
  Let $\xi: S\rar \A^k$ be a map and $\phi: p^*\cV\rar \F$ be a quotient on $\cY_S$, flat over $S$, such that $\cY_S|_\eta$ is smooth over $\eta$, $\phi|_\eta$ is admissible. Then there exists a finite base change $S'\rar S$, and $(\xi', \phi')$, with $\xi': S'\rar \A^{k'}$, and $S'$-flat quotient $\phi'$ on $\cY_{S'}$, such that $(\xi',\phi') \cong (\xi,\phi) \times_\eta \eta'$, where $\eta'$ is the generic point of $S'$, $\phi'$ is admissible, and $\left[ \Aut(\phi'|_{\eta'_0}): \Aut(\phi|_{\eta_0}) \right]$ is finite.
\end{lem}

One can assume that $\xi: S\rar \A^k$ is of the form $\C[t_1,\cdots,t_k]\rar R$ given by
$$(t_1,\cdots,t_k) \mapsto (u^{e_1},\cdots,u^{e_k}),$$
where $e_i\geq 1$. Moreover, by Step 1, we assume that $\phi|_{\eta_0}$ is normal to $D[k]$.

Pick $1\leq l\leq k-1$, such that $\deg \Err_l(\phi|_{\eta_0}) =\deg \Err(\phi|_{\eta_0})$, and we would like to modify the family and reduce this $l$-th error. Our strategy is to embed $\eta$ into a lager $\A^{k+1}$, take the flat limit in the larger target and somehow resolve the error to the new introduced divisor. However, the generic point can approaches $0\in \A^{k+1}$ in many different ways. To parameterize those different directions, we apply the following procedure, analogous to the construction of $X(1)$.

  Consider the map $m: S\times S\rar S$ given by $R\rar R\otimes R$, $u\mapsto v\otimes w$, which fits in the diagram
  $$\xymatrix{
  S\ar[r]^-\alpha  & S\times S \ar[r]^{\tilde\xi} \ar[d]_m & \A^{k+1} \ar[d]^{(t_l,t_{l+1})\mapsto t_l t_{l+1}} \\
  & S \ar[r]^\xi & \A^k.
  }$$
  Then $\tilde\xi: S\times S\rar \A^{k+1}$ is given by
  $$\quad \C[t_1,\cdots,t_{k+1}] \rar R \otimes R,$$
  $$(t_1,\cdots,t_{k+1}) \mapsto (u^{e_1},\cdots,u^{e_{l-1}},v^{e_l}, w^{e_l}, u^{e_{l+1}},\cdots,u^{e_k}),$$
  where $u=vw$, and we view $(v,w)$ as uniformizers of $S\times S$. Let $\C^*$ act on $S\times S$ by $\lambda \cdot (v,w):= (\lambda v, \lambda^{-1} w)$.

  Any map $\alpha: S\rar S\times S$ will give an arrow $\tilde\xi \circ \alpha: S\rar \A^{k+1}$, which is isomorphic to $\xi_\eta$ over the generic point. Note that there are two obvious such $\alpha$'s from the standard embedding, $\alpha: S\hookrightarrow \{1\} \times S$ and $S\hookrightarrow S\times \{1\}$; and if $e_l=1$ they are the only one's. (Here by $1\in S$ we mean the $\C$-point $\Sp\C\rar \eta$ given by $K\rar \C$, $u\mapsto 1$.) In general $m\circ \alpha: S\rar S$ is a finite base change. $\tilde\xi: S\times S\rar \A^{k+1}$ contains the information of all possible further embeddings of the original family.

  By $\C^*$-action, the stable quotient $\phi_\eta$ on $\cY_\eta\cong \cY_{\{1\}\times \eta}$ extends to a $\C^*$-equivariant stable quotient on $\cY_{\eta\times \eta}$, which furthermore extends to an equivariant family of quotients on $\cY_{S\times S-\eta_0\times \eta_0}$, by the original extension given in the assumption of Lemma \ref{key_lem_properness}. This extension on $\cY_{S\times S-\eta_0\times \eta_0}$ is stable over the points $\eta\times \eta_0$ and $\eta_0\times \eta$.

  Now for the origin $\eta_0\times \eta_0$, the problem comes that different choices of $\alpha$ give different flat limits. In other words, the equivariant family of quotients over $\cY_{S\times S-\eta_0\times \eta_0}$ is equivalent to an equivariant map $f: S\times S- \eta_0\times \eta_0 \rar \Quot_{Y(k+1)/\A^{k+1}}^{p^*\cV,P}$, but it does not necessarily extends to the codimension-2 point $\eta_0\times \eta_0$.

  We need the resolution of indeterminacy. Here $S\times S- \eta_0\times \eta_0$ is a smooth surface, and $\Quot_{Y(k+1)/\A^{k+1}}^{p^*\cV,P}$ is a projective scheme over $\A^{k+1}$ by Theorem 1.5 of \cite{OS}. In this case, $f$ extends to some $\tilde f: V\rar \Quot_{Y(k+1)/\A^{k+1}}^{p^*\cV,P}$, where $V\rar S\times S$ is a composite of successive blow-ups at points, and the exceptional divisor $E$ is a chain of rational curves. Moreover, $\tilde f$ is $\C^*$-equivariant with respect to the canonical $\C^*$-actions on both sides.

  Let's describe the family $(\cY_V,\tilde\phi)$ induced by $\tilde f$. Let $E=:\Sigma_1 \cup \cdots \cup \Sigma_m$ be the exceptional divisor of $V\rar S\times S$; let $\Sigma_0$ and $\Sigma_{m+1}$ be the proper transforms of $S\times \eta_0$ and $\eta_0\times S$ respectively. The only intersections are $q_i:= \Sigma_i\cap \Sigma_{i+1}$, $0\leq i\leq m$. Then by properties of $X(k)$ (Proposition \ref{Li_Xk}) we have
  $$\cY_V|_{\Sigma_0}\cong \Bl_{D_l\times \eta_0} (\cup_l \Delta \times S) \cup_{D\times S} (Y[k-l]\times S),$$
  $$\cY_V|_{\Sigma_{m+1}} \cong (Y[l]\times S)\cup_{D\times S} \Bl_{D\times \eta_0} (\cup_{k-l} \Delta\times S),$$
  and $\cY_V|_{\Sigma_i} \cong \Pj^1\times Y[k+1]$, for $1\leq i\leq m$. Then $\tilde \phi: p^*\cV\rar \tilde\F$ is a $\C^*$-equivariant quotient $\cY_V$, flat over $V$, obtained as the pull-back via $\tilde f$ of the universal family on the Quot-scheme.

  We have $\tilde \phi|_a \cong \phi|_{\eta_0}$ for every $\C$-point $a\in \eta\times \eta_0 = \Sigma_0-q_0$ or $\eta_0\times \eta= \Sigma_{m+1}-q_m$; and the restriction of $\tilde\phi$ on $E$ parameterizes various flat limits of $\phi|_{S\times S-\eta_0\times \eta_0}$ from different directions. For convenience, we denote the singular divisors in $Y[k+1]$ by
  $$D_0, \cdots, D_{l-1}, D_l^-, D_l^+, D_{l+1}, \cdots, D_k,$$
  and bubble components by
  $$\Delta_1, \cdots, \Delta_l, \tilde\Delta, \Delta_{l+1},\cdots, \Delta_k.$$
  We see that $\cY_V|_{\Sigma_0}$ is a smoothing of $D_l^-$, and $\cY_V|_{\Sigma_{m+1}}$ is a smoothing of $D_l^+$. $\C^*$ acts by weights $\pm e_l$ on $\Sigma_0$, $\Sigma_{m+1}$, and by weights $\pm 2e_l$ on other $\Sigma_i$'s and the newly inserted components $\tilde\Delta$ along the fiber. It acts trivially on other components.

\begin{lem}
  For any $0\leq i\leq m-1$, pick $a\in \Sigma_{i+1}- \{q_i,q_{i+1}\}$. Then
  \begin{enumerate}[1)]
  \setlength{\parskip}{1ex}

  \item $\Err_{l^-} (\tilde\F|_{a})= \Err_{l^-} (\tilde\F|_{q_i})$,  $\Err_{l^+} (\tilde\F|_{a})= \Err_{l^+} (\tilde\F|_{q_{i+1}})$, where $\Err_{l^\pm}$ denote the error at $D_l^\pm$; moreover, $\Err_{l^+}(\tilde \F|_{q_0})= \Err_{l^-}(\tilde \F|_{q_m}) =\Err_l(\F|_{\eta_0})$.

  \item $\Err_j(\tilde\F|_{q_i})= \Err_j(\tilde\F|_a) =\Err_j(\tilde\F|_{q_{i+1}})$, for any $j\neq l,l+1$.

  \item $\Err_{l^-}(\tilde \F|_{q_0})= \Err_{l^+}(\tilde \F|_{q_m})=0$.
  \end{enumerate}
\end{lem}

\begin{proof}
  For 1), we observe that near $D_l^-\times \{q_i\}$, there is an \'etale local chart $\Sp A:=\Sp \C[\vec{z},x,y,t]/(xy)$ of $\cY_{\Sigma_{i+1}-q_{i+1}}$, where $t$ is the coordinate in $\Sigma_{i+1}$, $x$, $y$ are coordinates along $\tilde\Delta$ and $\Delta_l$ respectively. $\C^*$ acts on $x$ and $t$ by weights of opposite signs. We are in the situation of Lemma \ref{3.13}. Globally this implies that for $a\in \Sigma_{i+1}-\{q_i,q_{i+1}\}$,
  $$P((\tilde\F|_a)_{\J_{l^-}})= P(\tilde\F_{\tilde\J_{l^-}}|_a)= P(\tilde\F_{\tilde\J_{l^-}}|_{q_i}) = P((\tilde\F|_{q_i})_{\J_{l^-}}),$$
  where $\tilde\J_{l^-}$ denotes the ideal sheaf of $D_l^-\times \Sigma_{i+1}\subset Y[k+1] \times \Sigma_{i+1}$. The first equality follows from the isomorphism given by $\C^*$-action over $\Sigma_{i+1}-\{q_i,q_{i+1}\}$, the second by flatness and the third from Lemma \ref{3.13}. One has the similar equality for $\tilde\F^\tf|_{\Delta_l,\I_{l^+}}$, $\tilde\F^\tf|_{\tilde\Delta, \I_{D_l^-\subset \tilde\Delta}}$ and their restriction to $D_l^-$. Thus 1) holds.

  2) follows from the $\C^*$-action on $\Sigma_{i+1}$.

  For 3), one needs to prove that $\tilde\F|_{q_0}$ is normal to $D_l^-$. Let $\Theta:=\Bl_{D_l\times \eta_0}(\Delta_l\times S)$ be the irreducible component of $\cY_{\Sigma_0}$, and $\Theta^*:=\Theta- (\text{proper transform of } D_{l-1}\times S \cup D_l\times S)$. Consider the map
  $$g: \Theta^*\hookrightarrow \Theta=\Bl_{D_l\times \eta_0}(\Delta_l\times S) \xrar{p_{\Bl}} \Delta_l \times S \rar \Delta_l $$
  which is the contraction map of $\cY_{\Sigma_0}$ to the original family $\cY_{\eta_0}\times S$, restricted on the involved component $\Theta$.

  Consider $\tilde\phi|_{\Theta^*}$. From the $\C^*$-equivariance there is an isomorphism between $g^*((\phi|_{\eta_0})^\tf|_{\Delta_l})^\tf$ and $\tilde\phi|_{\Theta^*}$ over $\Theta^*-\cY_{q_0}$. Since they are both flat over $\Sigma_0$ (because any torsion over the closed point would contradict the flatness of $g^*((\phi|_{\eta_0})^\tf|_{\Delta_l})^\tf$ on $\Delta\times S$), they must have the same flat limit over $q_0$. Thus it suffices to prove that $g^*((\phi|_{\eta_0})^\tf|_{\Delta_l})^\tf$, restricted to $\cY_{q_0}$, is normal to $D_l^-$, which is obvious from construction.
\end{proof}

\begin{proof}[Proof of Lemma \ref{key_lem_properness}]
  If $\phi|_{\eta_0}$ is stable, then we have nothing to do. Suppose otherwise; i.e. $\phi|_{\eta_0}$ is normal to $D[k]$, but not admissible. Take $l$ as above, $1\leq l\leq k-1$. We apply the procedure as above.

  Since $\tilde\F$ is flat over $V$, $\tilde\F|_a$ has the same modified Hilbert polynomial, for every $\C$-point $a\in V$. By 1) of Lemma \ref{identity_delta} and 2) of the last lemma, the following is constant for each $a\in \Sigma_1\cup \cdots \Sigma_m$,
  $$\Err_{l^-}(\tilde\F|_a) + \Err_{l^+}(\tilde\F|_a) + \frac{1}{2} (\delta_\sim^-(\tilde\F|_a) + \delta_\sim^+(\tilde\F|_a)),$$
  where $\delta_\sim^\pm$ denote the $\delta$ polynomials defined for the component $\tilde\Delta$.

  Since $\tilde\phi$ is $\C^*$-equivariant and each $\tilde\cY_{q_i}$ is $\C^*$-invariant under the action, $\tilde\phi|_{q_i}$ is $\C^*$-equivariant. By 2) of Lemma \ref{identity_delta}, we have $\delta_\sim^\pm(\tilde \F|_{q_i})=0$, $0\leq i\leq m$ and $\delta_\sim^\pm(\tilde \F|_a)$ have non-negative leading coefficients. Hence,
  $$\Err_{l^-}(\tilde\F|_{q_j}) + \Err_{l^+}(\tilde\F|_{q_j}) = \Err_l(\tilde \F|_{\eta_0}), \quad \forall 0\leq j\leq m.$$

  On the other hand, we have $\Err_{l-}(\tilde\F|_{q_0})=0$ and $\Err_{l-}(\tilde\F|_{q_m})=\Err_l(\tilde\F|_{\eta_0})\neq 0$. Thus there exists some $i$ such that $\Err_{l-}(\tilde\F|_{q_i}) \prec \Err_{l-}(\tilde\F|_{q_{i+1}})$. Then for $a\in \Sigma_{i+1}-\{q_i,q_{i+1}\}$, by the last lemma, we have
  \begin{eqnarray*}
  \Err_{l^-}(\tilde\F|_a) + \Err_{l^+}(\tilde\F|_a)
  &=& \Err_{l^-}(\tilde\F|_{q_i}) + \Err_{l^+}(\tilde\F|_{q_{i+1}}) \\
  &\prec& \Err_{l^-}(\tilde\F|_{q_{i+1}}) + \Err_{l^+}(\tilde\F|_{q_{i+1}}) \\
  &=& \Err_l(\tilde\F|_{\eta_0}).
  \end{eqnarray*}
  Thus $\Err(\tilde\F|_a) \prec \Err(\tilde\F|_{\eta_0})$. Moreover, since the total Hilbert polynomial is constant, we must have $\delta_\sim^\pm(\tilde\F|_a) \succ 0$; in particular, $\Aut(\tilde\F|_a)$ is finite on the new bubble $\tilde\Delta$.

  Since $V$ is smooth, one can pick a curve $S'\subset V$ that contains $a$, such that the map $S'\rar V\rar S\times \A^1 \rar S$ is finite. Take the map $\xi': S'\rar V\rar \A^{k+1}$, and the quotient pulled back from the universal family. We obtain a quotient $\phi'$ on $\cY_{S'}$, with $\phi'|_{\eta'} \cong \phi|_\eta\times_\eta \eta'$, $\Err(\F'|_{\eta'_0})=\Err(\tilde\F|_a) \prec \Err(\F|_{\eta_0})$ and only finitely many new autoequivalences arise. Repeat this process, and finally we'll get $\Err(\F'|_{\eta'_0})=0$.
\end{proof}

\textbf{Step 3.} Finite autoequivalences.

\begin{lem}
  Consider the family $\Delta\times S\rar S$, viewed as a component of the family $\cY_S$ associated to the map $\xi: S\rar \A^1$, where $\xi$ maps $S$ constantly to $0\in \A^1$. Let $\phi$ be an $S$-flat quotient on $\Delta\times S$, such that $\phi|_\eta$ is stable, but $\phi|_{\eta_0}$ is $\C^*$-equivariant. Then there is another $S$-flat quotient $\psi$, such that $\psi|_\eta$ is related to $\phi|_\eta$ via $\C^*$-action on $Y(1)\rar \A^1$, and $\F|_{\eta_0}$ normal to $D_-$.
\end{lem}

\begin{proof}
  First note that $\delta_\Delta^\pm=0$ by Lemma \ref{identity_delta}. Thus $\Err(\F|_{\eta_0})\neq 0$, i.e. the sheaf must be not admissible. Suppose $\F|_{\eta_0}$ is not normal to $D_-\times \eta_0$.

  \'Etale locally around a point in $D_-\times \eta_0$, $\Delta\times S$ has an affine local chart $\Sp A:= \Sp R[y,\vec{z}]$, where $D_-\times\eta_0$ is defined by the ideal $(u,y)$. The quotient $\phi$ is represented by a sequence $0\rar K\rar A^{\oplus r}\rar M\rar 0$. Assume that $K=(f_1(u,y),\cdots,f_m(u,y))$, with generators
  $$f_i(u,y)=c_i + u^{\alpha_i} g_i(u) + y h_i(u,y) \in A^{\oplus r},$$
  where $c_i\in k[\vec{z}]^{\oplus r}$, $g_i\in k[\vec{z}]^{\oplus r}[u]$, $h_i\in k[\vec{z}]^{\oplus r}[u,y]$, $g_i(0)\neq 0$ and $\alpha_i \geq 1$. As in Step 1 one concludes that
  $$f_i(u,y)=u^{\alpha_i} g_i(u) + y h_i(u,y) \in A^{\oplus r}.$$

  Now we take some $N\geq \max\{\alpha_i\}$ and apply the $\C^*$-action induced by
  $$\eta\rar \C^*, \quad \C[t_1,t_1^{-1}]\rar K, \quad t_1\mapsto u^N,$$
  which acts as $y\mapsto u^N y$ on $\Delta$.

  The family $\phi|_\eta$ becomes a new family $\psi_\eta$ represented by $0\rar K'_u\rar A_u^{\oplus r}\rar M'_u\rar 0$, where the generators become
  $$f_i(u,u^N y)= u^{\alpha_i} g_i(u) + u^N yh_i(u,u^N y),$$
  and we may also take
  $$f'_i(u,y)= g_i(u) + u^{N-\alpha_i} yh_i(u,u^N y)$$
  as the generators of $K'_u$.

  Let's look at its flat limit $\psi$, or a new sequence $0\rar K'\rar A^{\oplus r}\rar M'\rar0$. By construction we have
  $$K'= K'_u\cap A^{\oplus r},$$
  with generators $f'_i$. Now for every $i$, $f'_i$ is not divisible by $y$, i.e. the flat limit $\psi|_{\eta_0}$ is normal to $D_-\times \eta_0$.
\end{proof}

\begin{lem} \label{step_3}
  Let $\Delta\times S$ be given as in the previous lemma. Let $\phi$ be an $S$-flat quotient on $\Delta\times S$, such that $\phi|_\eta$ is stable. Then there is another $S$-flat quotient $\psi$, such that $\psi|_\eta$ is related to $\phi|_\eta$ via $\C^*$-action on $Y(1)\rar \A^1$, with $\Aut(\phi|_{\eta_0})$ finite.
\end{lem}

\begin{proof}
  Suppose $\Aut(\phi|_{\eta_0})$ is not finite. Then $\phi|_{\eta_0}$ is $\C^*$-equivariant and it must be not admissible. By the previous lemma, up to a $\C^*$-action we can assume that it is normal to $D_+$. Again by the previous lemma, there is another quotient $\psi|_\eta$, related to $\phi|_\eta$ via a $\C^*$-action $\lambda: \eta\rar \C^*$, such that $\psi|_{\eta_0}$ is normal to $D_-$.

  We apply a similar argument to Step 2. Consider the 2-dimensional family $\Delta\times S\times S$ and two different embeddings of $\Delta\times S$ into it. The first is the standard one induced by $\{1\}\times S\rar S\times S$, i.e. $u\mapsto (1,u)$; and the second is given by $u\mapsto  (\lambda\cdot u, 1)$. Let $\C^*$ acts on $S\times S$ via multiplication by $\mu\cdot (u,v):= (\lambda\cdot \mu u, \mu^{-1} v)$. Then as in Step 2, there is a $\C^*$-equivariant quotient on $S\times S- \eta_0\times \eta_0$, whose restrictions to the two embeddings are $\phi|_\eta$ and $\psi|_\eta$ respectively.

  Again we pass to a successive blow-up $V\rar S\times S$, and obtain a $\C^*$-equivariant quotient $\tilde\phi$ on $\Delta\times V$. The exceptional fiber $E=\Sigma_1\cup \cdots \cup \Sigma_m$ is a chain of $\Pj^1$'s. Let $\Sigma_0:= \{1\}\times S$ and $\Sigma_{m+1}:= S\times \{1\}$. Again $\Sigma_{m+1}$ is viewed as twisted by $\lambda$. Consider the quotients on these $\Delta\times \Sigma_i$.

  For $a\in \Sigma_0$ or $\Sigma_{m+1}$, $\tilde\phi|_a$ is isomorphic to $\phi|_{\eta_0}$ and $\psi|_{\eta_0}$ respectively, whose $\Err$'s concentrate on $D_-$, $D_+$ respectively. By same arguments as in Step 2, one concludes that there is some $1\leq i\leq m$ and $a\in \Sigma_i^\circ$, such that $\Err(\tilde\F|_a) \prec \Err(\F|_{\eta_0})$. Equivalently, this means $\delta_\Delta^\pm(\tilde\F|_a)\neq 0$ and $\Aut(\tilde\phi|_a)$ is finite. Again one picks a curve $S'\subset V$ passing through $a$ which is finite over $S$ and projects it to $S$. The lemma follows.
\end{proof}

\begin{proof}[Proof of Theorem \ref{thm_properness}]
  The proof is a combination of the three steps. As mentioned at the beginning of the subsection, we have a naive extension $(\xi, \phi)$ with $\xi: S\rar \A^k$, mapping constantly to $0\in \A^k$, and $\phi: p^*\cV\rar \F$ on $\cY_S\cong S\times Y[k]$, such that $\phi|_\eta$ is stable, but $\phi|_{\eta_0}$ not necessarily stable. By Lemma \ref{step_3}, we can assume that $\Aut(\phi|_{\eta_0})$ is finite. Thus we just need admissibility.

  Let's modify $\xi$. We have
  $$\cY_S\cong S\times Y[k] = (S\times Y) \cup_{S\times D} \cdots \cup_{S\times D} (S\times \Delta_k).$$
  Apply Lemma \ref{normal to dist divisor} and \ref{key_lem_properness} to each smooth pair $(\Delta_j, D_{j-1}\cup D_j)$. After a finite base change $S'\rar S$, for each $0\leq j\leq k$, one can find $\xi': S'\rar \A^{k'}$, and extend $\phi|_\eta\times_\eta \eta'$ to an admissible quotient $\phi'_i$ on  $\cY_{j,S'}$, where $\cY_{j,S'}$ is the associated family of expanded pairs with respect to the pair $(\Delta_j, D_{j-1}\cup D_j)$. Note that $\phi'_i$ for adjacent $\Delta$'s must coincide on the intersecting divisors because of admissibility and the uniqueness of the flat limit. Hence these quotients glue together to an admissible quotient $(\xi',\phi')$ on $\cY_{S'}$, and $\phi'|_{\eta'} \cong \phi|_\eta \times_\eta \eta'$.

  During each step the autoequivalence group remains finite, and for the base changes we can take the fiber products and pass to a common finite base change. The resulting $\phi'|_{\eta'}$ is stable.
\end{proof}

\section{Orbifold Donaldson--Thomas theory} \label{DT}

In this section we consider the (absolute) Donaldson--Thomas theory on 3-dimensional smooth Deligne--Mumford stacks. Let $S$ be a scheme and $\pi: W\rar S$ be a \emph{smooth} family of projective Deligne--Mumford stacks over $\C$, of relative dimension 3, and let $c: W\rar \uW$ be the coarse moduli space. Let $\E$ be a generating sheaf on $W$, and $H=c^*\uH$ be a pull-back of an relatively ample line bundle $\uH$ on $\uW\rar S$.

Let $P: K(W)\rar \Z$ be a group homomorphism, and we require that the associated Hilbert polynomial $P(H^{\otimes v})$ has degree not greater than 1. One can form the Hilbert scheme $\cM:=\Hilb_{W/S}^P$ parameterizing certain closed substacks of $W$ over $S$. $\cM$ is a projective scheme over $S$ by Theorem 4.4 of \cite{OS}. The following lemma gives another description of $\cM$.

\begin{lem} \label{moduli_ideal}
  $\cM$ is the fine moduli space of torsion free coherent sheaves $\I$ on $W$ with relative Hilbert homomorphism $P_\I= P_{\Ou_W} - P$ and $\det \I$ trivial.
\end{lem}

\begin{proof}
  We sketch the idea of the proof. Let $\cM'$ denote the functor parameterizing such sheaves on $W$, i.e. given any $T\rar S$, $\cM'(T)$ consists of $T$-flat families $\I_T$ of coherent sheaves on $T\times_S W$, whose fibers are torsion free, have Hilbert homomorphism $P_{\Ou_W} - P$ and trivial determinant. There is a natural transformation $\cM\rar \cM'$, sending a flat family of closed substacks to its ideal sheaf. This is obviously fully faithful. One can check that this is a bijection over closed points by embedding a torsion free sheaf into its double dual and apply Lemma 1.1.15 of \cite{OSS}. For the proof in families, one can apply Lemma 6.13 of \cite{Kol2}.
\end{proof}

In the following we define a perfect relative obstruction theory on $\cM$ over $S$, in the sense of \cite{BF}. The construction follows from the idea of D. Huybrechts and R. P. Thomas \cite{HT} and is already adopted by A. Gholampour and H.-H. Tseng in \cite{GT}.

Let $p: \cM \times_S W \rar \cM$, $q: \cM \times_S W \rar W$ be the projections and $\cZ\subset \cM\times_S W$ be the universal family of the Hilbert scheme. Let $\bI\subset \Ou_{\cM\times_S W}$ be the universal ideal sheaf. We have maps (in the derived category) $\Id: \Ou_{\cM\times_S W} \rar R\cH om(\bI,\bI)$ and $\tr: R\cH om(\bI,\bI)\rar \Ou_{\cM\times_S W}$ which give a splitting of $R\cH om(\bI,\bI)$.

Let $L^\bullet_{\cM\times_S W/S}$ be Illusie's cotangent complex of $\cM\times_S W \rar S$. We have the Atiyah class
$$\text{At}_{\cM\times_S W/S}(\bI) \in \Ext^1_{\cM\times_S W} (\bI,\bI \otimes L^\bullet_{\cM\times_S W/S}),$$
which can be seen as a map $\bI\rar \bI \otimes L^\bullet_{\cM\times_S W/S} [1]$. By functoriality, if we compose it with the natural projection $L^\bullet_{\cM\times_S W/S} \rar p^* L_{\cM/S}^\bullet$, we get the Atiyah class of $\bI$ as a module over $\cM$,
$$\text{At}_{\cM\times_S W/W}(\bI) \in \Ext^1_{\cM\times_S W} (\bI, \bI \otimes p^* L^\bullet_{\cM/S}).$$
After restriction to the traceless part we get $R\cH om(\bI,\bI)_0 \rar p^* L^\bullet_{\cM/S} [1]$, and then tensor with the dualizing sheaf $R\cH om(\bI,\bI)_0 \otimes q^* \omega_{W/S} \rar p^* L^\bullet_{\cM/S} \otimes q^* \omega_{W/S} [1] \cong p^! L_{\cM/S}^\bullet [-2]$, where the last isomorphism is the Serre duality for Deligne--Mumford stacks \cite{Ni2}. We have obtained the map
$$\Phi: E^\bullet:= Rp_* (R\cH om(\bI,\bI)_0 \otimes q^* \omega_{W/S}) [2] \rar L_{\cM/S}^\bullet.$$

\begin{thm} \label{obs_smooth}
  The map $\Phi: E^\bullet\rar L_{\cM/S}^\bullet$ is a perfect relative obstruction theory on $\cM$ over $S$, in the sense of \cite{BF}.
\end{thm}

\begin{proof}
  First let's prove that $\Phi$ is an obstruction theory. Consider a square-zero extension $T\hookrightarrow \overline T$ of schemes with ideal sheaf $J$, and a map $g: T\rar \cM$. The canonical map
  $$g^* L_{\cM/S}^\bullet \rar L_{T/S}^\bullet \rar L_{T/\overline T}^\bullet \rar \tau^{\geq -1} L_{T/\overline T}^\bullet \cong J[1]$$
  gives an element $\omega(g)\in \Ext^1(g^* L_{\cM/S}^\bullet, J)$, which is the obstruction to extending the map $g$ to $\overline T$. Composing with $\Phi$, we get $\Phi^*\omega(g) \in \Ext^1(g^* E^\bullet, J)$.

  By the the relative version of Theorem 4.5 of \cite{BF}, it suffices to prove that $\Phi^*\omega(g)$ is also an obstruction to the extension of $g$ to $\overline T$, i.e. $\Phi^*\omega(g)=0$ if and only if an extension $\bar g$ of $g$ to $\overline T$ exists, and in that case all extensions form a torsor under $\Ext^0(g^* E^\bullet, J)$.

  Denote by $p_T$ and $q_T$ the corresponding projections from $T\times_S W$ to $T$ and $W$. By the construction above, $\Phi^*\omega(g)$ is the composite of
  $$g^*\Phi: g^*E^\bullet = R {p_T}_* (R\cH om(g^*\bI, g^*\bI)_0 \otimes q_T^* \omega_{W/S}) [2] \rar g^* L_{\cM/S}^\bullet,$$
  with $g^* L_{\cM/S}^\bullet \rar L_{T/S}^\bullet$, and the Kodaira--Spencer map
  $$\kappa(T/\overline T/S): L_{T/S}^\bullet \rar L_{T/\overline T}^\bullet \rar \tau^{\geq -1} L_{T/\overline T}^\bullet \cong J[1].$$

  By functoriality and Serre duality, the first composition
  $$g^*E^\bullet = R {p_T}_* (R\cH om(g^*\bI, g^*\bI)_0 \otimes q_T^* \omega_{W/S}) [2] \rar g^* L_{\cM/S}^\bullet \rar L_{T/S}^\bullet,$$
  is just the traceless part of the Atiyah class
  $$\text{At}_{T\times_S W/W} (g^*\bI) \in \Ext^1_{T\times_S W} (g^*\bI, g^*\bI \otimes p_T^*L_{T/S}^\bullet).$$

  Hence, viewed as an element in the $\Ext^2$ group below, the element
  $$\Phi^*\omega(g) \in \Ext^{-1}_T(R {p_T}_* (R\cH om(g^*\bI, g^*\bI)_0 \otimes q_T^* \omega_{W/S}), J) \cong \Ext^2_{T\times_S W} (g^*\bI, g^*\bI \otimes p_T^*J)_0,$$
  is just the product of the traceless Atiyah class and Kodaira--Spencer class
  $$p_T^*\kappa(T/\overline T/S) \circ \text{At}_{T\times_S W/W} (g^*\bI).$$

  We claim that this is an obstruction as desired. Note that (as in the proof of Theorem 4.5 of \cite{BF}) it suffices to work under the further assumption that $T$ is affine. Now we can apply Proposition 3.1.8 of \cite{Il}, which says that the obstruction class to extending $\bI$ from $T$ to $\overline T$, as a coherent sheaf is exactly the product of Atiyah class and Kodaira--Spencer class. Now similar arguments as in \cite{Th} would show that the traceless part is exactly the obstruction with $\det \bI$ fixed.

  A map $g: T\rar \cM$, by Lemma \ref{moduli_ideal}, corresponds exactly to an ideal sheaf, given by the pull-back of the universal ideal sheaf $g^*\bI$, and $g$ extends to $\overline T$ if and only if $g^*\bI$ extends in $\cM$. As an open condition, the torsion-free condition poses no restriction on possible deformations. Thus $\Phi^*\omega(g)$ gives precisely the obstruction to the extension of $g$, and $\Phi$ is an obstruction theory.

  It remains to prove that $\Phi$ is perfect. The argument is completely the same with \cite{HT} and \cite{PT}. It suffices to prove that $Rp_* R\cH om(\bI, \bI)_0$ is quasi-isomorphic to a perfect complex with amplitude in $[1,2]$. Pick a finite complex of locally free sheaves $A^\bullet$ resolving $R\cH om(\bI,\bI)_0$ such that $R^i p_* A^j=0$ for $\forall i\neq 0$ and $\forall j$. Then each $F^j:=p_* A^j$ is locally free and the complex $F^\bullet$ is a resolution of $Rp_* R\cH om(\bI, \bI)_0$.

  By the following Lemma \ref{Ext12} and the cohomology and base change theorem, $F^\bullet$ has nontrivial cohomology only in degree 1 and 2. Suppose $F^j$ with $j>2$ is the last nonzero term of $F^\bullet$. We can replace $F^{j-1}\rar F^j\rar 0$ by $\ker (F^{j-1}\rar F^j) \rar 0\rar 0$. Same with the nonzero terms before 1. Finally one obtains a two-term perfect complex.
\end{proof}

We still need the following lemma to finish the proof.

\begin{lem} \label{Ext12}
  Let $W$ be a 3-dimensional smooth projective Deligne--Mumford stack over $\C$ and $Z\subset W$ be a 0 or 1-dimensional closed substack, with ideal sheaf $\I$. Then $\Ext^i_W (\I,\I)_0=0$, for all $i\neq 1,2$.
\end{lem}

\begin{proof}
  It suffices to show that $\Ext^i_W (\I,\I)_0=0$ for $i=0,3$, or in other words, $\Ext^i_W (\I,\I)= H^i(W, \Ou_W)$ for $i=0,3$. For $0\rar \I\rar \Ou_W \rar \Ou_Z\rar 0$ we have the long exact sequence,
  $$\xymatrix@R-1pc{
  0 \ar[r] & \cH om_W (\Ou_Z, \I) \ar[r] & \cH om_W(\Ou_Z, \Ou_W) \ar[r] & \cH om_W (\Ou_Z, \Ou_Z) \\
  \ar[r] & \E xt^1_W (\Ou_Z, \I) \ar[r] & \E xt^1_W (\Ou_Z, \Ou_W),
  }$$
  where $\cH om_W (\Ou_Z,\Ou_W) = \E xt^1_W (\Ou_Z, \Ou_W) =0$, since $\mathrm{codim} Z\geq 2$ (see Proposition 1.1.6 of \cite{HL}). Hence $\cH om_W (\Ou_Z,\I) = 0$,  $\E xt^1_W (\Ou_Z, \I) =\Ou_Z$. Then we look at
  $$\xymatrix@R-1pc{
  & 0=\cH om_W (\Ou_Z,\I) \ar[r] & \cH om_W (\Ou_W, \I) \ar[r] & \cH om_W (\I,\I) \\
  \ar[r] & \E xt^1_W (\Ou_Z, \I) \ar[r] & \E xt^1_W (\Ou_W, \I)=0,
  }$$
  Hence $\cH om_W(\I, \I)$ differs with $\I$ by a sheaf $\Ou_Z$ of codimension 2, and we have $\Hom(\I, \I) = H^0(W, \I) = H^0(W, \Ou_W)$. The last equality is because $\I$ coincides with $\Ou_W$ up to codimension 2.

  Tensoring the sequence above with $\omega_W$ we have $\cH om_W (\I,\I\otimes \omega_W)$ coincides with $\I\otimes \omega_W$, and hence $\omega_W$, up to codimension 2. Thus by Serre duality $\Ext^3_W (\I,\I) = \Hom_W (\I,\I \otimes \omega_W)^\vee = H^0(W, \I \otimes \omega_W)^\vee = H^0 (W, \omega_W)^\vee = H^3 (W, \Ou_W)$.
\end{proof}

\begin{cor}
  There is a virtual fundamental class $[\cM]^{\vir} \in A_{\operatorname{vdim}+\dim S}(\cM)$, of virtual dimension $\operatorname{vdim} = \operatorname{rk} E^\bullet$.
\end{cor}

The following gives the deformation invariance of the perfect obstruction theory.

\begin{prop}
  Assume that the base $S$ is smooth of constant dimension. For each closed point $s\in S$, we have the pull-back $i_s^*E^\bullet$ is a perfect (absolute) obstruction theory on the fiber $\cM_s$, and $[\cM_s]^{\vir} = i_s^![\cM]^{\vir}$.
\end{prop}

\begin{proof}
  Appy Proposition 7.2 of \cite{BF}.
\end{proof}

Following \cite{Ed} and Appendix A of \cite{Ts}, we introduce the following notion of Chern character. Let $\cX$ be a smooth proper Deligne--Mumford stack and let $I\cX$ be the inertia stack. Connected components of $I\cX$ are gerbes over their coarse moduli spaces. Given a vector bundle $V$ on $I\cX$, it splits into a direct sum of eigenbundles $\bigoplus_\zeta V^{(\zeta)}$ of the gerbe actions, where $V^{(\zeta)}$ has eigenvalue $\zeta$.

\begin{defn}
  Define $\rho: K(I\cX)\rar K(I\cX)_\C$ as
  $$\rho(V):= \sum_\zeta \zeta V^{(\zeta)} \in K(I\cX)_\C.$$
  Define $\ch:K(\cX)_\Q\rar A^*(I\cX)_\C$ as
  $$\ch(V):= ch(\rho(\pi^*V)),$$
  where $\pi: I\cX\rar \cX$ is the usual projection and $ch$ is the usual Chern character.
\end{defn}

It is easy to see that $\ch$ is a ring homomorphism. If $\cX$ is furthermore projective and thus satisfies the resolution property, then by splitting principle one can check that
\begin{eqnarray*}
\ch(V^\vee) &=& \ch^\dagger (V) \\
&:=& \overline{ch^\vee (\rho(\pi^*V))},
\end{eqnarray*}
where $ch^\vee$ means the usual dual of Chern character, and the bar over it means the conjugate with respect to the natural real structure $K(I\cX)\otimes_\Z \mathbb{R} \subset K(I\cX)\otimes_\Z \C$.

Now for a 3-dimensional smooth projective Deligne-Mumford stack $W$ over $\C$, we can define the Donaldson-Thomas invariants. Again let $\cM= \Hilb_W^P$, with $P$ fixed as above. Consider $p: \cM\times W\rar \cM$, $q: \cM\times W \rar W$, the universal family $\cZ\subset \cM\times W$ and the universal ideal sheaf $\bI\subset \Ou_{\cM\times W}$.

$\cM\times W$ is projective, thus $\bI$ admits locally free resolutions of finite length, and the Chern character $ch(\bI)\in A^*(\cM\times W)$ in the operational cohomology. Similarly, on inertia stacks we have the modified Chern character $\ch(\bI)\in A^*(\cM\times IW).$

Consider the orbifold or Chen-Ruan cohomology $A^*_\orb$ as in \cite{AGV, Ts}, defined as
$$A^*_\orb(W):= \bigoplus_i A^{*-\age_i}(W_i),$$
where $\age_i$ is the degree shift number, and $W_i$ is a connected component of $IW$. Then we can define an orbifold version of the Chern character operator
$$\ch^\orb(\bI)\in A^*_\orb(\cM\times W)$$
just by shifting the degrees.

Given $\gamma\in A^l_\orb(W)$, define the operator
$$\ch_{k+2}^\orb(\gamma): A_*(\cM)\rar A_{*-k+1-l}(\cM)$$
as
$$\ch_{k+2}^\orb(\gamma)(\xi):= p_* \left( \ch_{k+2}^\orb(\bI) \cdot \iota^* q^* \gamma \cap p^* \xi \right),$$
where $p$ and $q$ here are the projections on $\cM\times IW$, and $\iota: IW\rar IW$ is the canonical involution map. Note that the orbifold degrees match well thanks to the identity
$$\age_i+ \age_{\iota(i)} = \text{codim}(W_i, W).$$

\begin{defn}
For $\gamma_i\in A^*_\orb(W)$, $1\leq i\leq r$, define the Donaldson--Thomas invariants with descendants as
$$\left\langle \prod_{i=1}^r \tau_{k_i}(\gamma_i) \right\rangle_W^P := \deg \left[ \prod_{i=1}^r \ch_{k_i+2}^\orb(\gamma_i) \cdot [\cM]^{\vir} \right]_0 \in \C.$$
If the dimensions don't match, we simply define the invariant to be 0.
\end{defn}

Note that when $r=0$, i.e. there are no insertions, the invariants are integers.

\section{Degeneration formula -- cycle version} \label{Deg_cycle}

In this section we consider the Donaldson--Thomas theory on Hilbert stacks of the moduli's of simple degenerations and relative pairs. This will lead to a degeneration formula.

\subsection{Modified versions of the stacks and decomposition of central fibers}

Let $\pi: X\rar C$ be a locally simple degeneration, where $X$ is a finite-type separated Deligne--Mumford stack over $\C$, with central fiber $X_0=Y_-\cup_D Y_+$. Recall that we have the stacks $\fC$ parameterizing expanded degenerations, $\fX$ the universal family, and $\fA$ and $\fY$ in the relative case. There is $\pi: \fX\rar \fC$, where the stacks are defined as the limits of $[X(k)/R_{\sim,X(k)}]$ and $[\A^{k+1}/R_{\sim,\A^{k+1}}]$.

We define a stack $\fX_0^\dagger$ as follows. Let $H_i\subset \A^{k+1}$ be the hyperplane defined by $t_i=0$. Recall that we have $p: X(k)\rar \A^{k+1}$ and by Proposition \ref{Hi_Xk},
\begin{equation} \label{split_explicit}
  X(k)|_{H_i}\cong (Y_-(i)\times \A^{k-i}) \cup_{\A^i\times D\times \A^{k-i}} (\A^i\times Y_+(k-i)^\circ).
\end{equation}
The equivalence relation restricts naturally to $H_i$ and $X(k)|_{H_i}$. Consider the stacks
$$\fC_0^\dagger:= \varinjlim \left[ \coprod_{i=0}^k H_i \middle/ R_{\sim,\coprod H_i} \right],$$
$$\fX_0^\dagger:= \varinjlim \left[ \coprod_{i=0}^k X(k)|_{H_i} \middle/ R_{\sim,\coprod X(k)|_{H_i}} \right].$$

\begin{rem}
  The stack $\fX_0^\dagger$ parameterizes all families of expanded degenerations with singular fibers, with one distinguished nodal divisor $D_i$ . More precisely, for any $\A^1$-map $S\rar H_i\subset \A^{k+1}$, the associated family $\cX_S$ has a decomposition $\cY_{S,-} \cup_{\cD_S} \cY_{S,+}$.
\end{rem}

To keep track of the Hilbert homomorphisms of the fibers, we need to introduce the weighted version of the classifying stacks.

\begin{defn}[(Definition 2.14 of \cite{LW})]
  Let $\Lambda:=\Hom(K(X),\Z)$. Consider $X_0[k]=\cup_{i=0}^{k+1} \Delta_i$, where $\Delta_0=Y_-$ and $\Delta_{k+1}=Y_+$.
  \begin{enumerate}[1)]
  \setlength{\parskip}{1ex}

  \item A \emph{weight assignment} on $X_0[k]$ is a function
    $$w: \{ \Delta_0,\cdots, \Delta_{k+1}, D_1, \cdots, D_k \} \rar \Lambda,$$
    such that $w(\Delta_i)\neq 0$, $\forall i\in \Lambda$.

  \item For any $0\leq a\leq b\leq k+1$, the \emph{total weight} of the segment $\cup_{i=a}^b \Delta_i$ is defined as
      $$w(\cup_{i=a}^b \Delta_i):= \sum_{i=a}^b w(\Delta_i) -\sum_{i=a+1}^b w(D_i).$$

  \item For a family of expanded degeneration $\pi: \cX_S\rar S$, a \emph{continuous weight assignment} on $X(k)$ is to assign a weight function on each fiber that is continuous over $S$.

      More precisely, if for some curve $C\subset S$ and $s_0,s\in C$, the general fiber $\cX_s\cong X_0[m]$ specializes to the special fiber $\cX_{s_0}\cong X_0[n]$, with $m\leq n$, in which $\Delta_i\subset \cX_s$ specializes to $\cup_{j=a}^b \Delta_j \subset \cX_{s_0}$, then one must have $w_s(\Delta_i)=w_{s_0}(\cup_{j=a}^b \Delta_j)$.
  \end{enumerate}
\end{defn}

\begin{rem}
  An alternative definition is as follows. Firstly, for a standard family $X(k)\rar \A^{k+1}$, one specifies the weight assignment on the central fiber $X_0[k]$. Then for any $s\in L_i$, where $L_i\subset \A^{k+1}$ is the coordinate line corresponding to $t_i$, define
  $$w_s(\Delta_j)= \left\{ \begin{aligned}
  & w_0(\Delta_i)+w_0(\Delta_{i+1})- w_0(D_i), \quad && j=i,\\
  & w_0(\Delta_j), \quad && j\neq i.
  \end{aligned} \right. $$
  For $s$ in lower strata, one can choose a slice in the base and define inductively. For a general family $\cX_S\rar S$, just take the map $S\rar \A^{k+1}$ for some $k$ and pull back the weight assignment on the standard family $X(k)$.
\end{rem}

Let $P\in \Lambda$ be fixed. We define $\fC^P$ to be the stack parameterizing weighted families of expanded degenerations, with total weight $P$, and $\fX^P$ to be the universal family. More precisely, let $(\A^{k+1})^P$ be the disjoint union of all $\A^{k+1}$ indexed by all possible continuous weight assignments, and $X(k)^P$ be the universal family. Let $R_{\sim,\A^{k+1}}^P$ be the equivalence relation generated by the original relations on each copy of $\A^{k+1}$ and identifications between different copies respecting weight assignments. Then we can define
$$\fC^P:= \varinjlim \left[ (\A^{k+1})^P \middle/ R_{\sim,\A^{k+1}}^P \right],$$
$$\fX^P:= \varinjlim \left[ X(k)^P \middle/ R_{\sim,X(k)}^P \right].$$
Similarly we can define $\fX_0^{\dagger,P}\rar \fC_0^{\dagger,P}$, the weighted version of $\fX_0^\dagger\rar \fC_0^\dagger$.

Let $\Lambda_P^{spl}$ be the set of splitting data
$$\Lambda_P^{spl}:= \{ \theta=(\theta_-,\theta_+,\theta_0) \mid \theta_\pm, \theta_0 \in \Lambda, \theta_- + \theta_+ - \theta_0 = P\}.$$
Given $\theta\in \Lambda_P^{spl}$, let $\fC_0^{\dagger,\theta}\subset \fC_0^{\dagger,P}$ be the open and closed substack parameterizing those families whose Hilbert homomorphism splits according to the datum $\theta$ on the fiber decompositions $\cY_{S,-} \cup_{\cD_S} \cY_{S,+}$.

Similarly in the relative case, given a smooth pair $(Y, D)$ and $P\in \Hom(K(Y),\Z)$, We also have the stacks $\fA^P$ and $\fY^P$. Given $(P,P')$ with $P' \prec P$, one can also define the substack $\fA^{P,P'}$ and $\fY^{P,P'}$, parameterizing those families whose Hilbert homomorphism on each fiber is $P$ and on the distinguished divisor of each fiber is $P'$. The following splitting result follows from (\ref{split_explicit}).

\begin{prop}[(Proposition 2.20 of \cite{LW})]
  Given $\theta\in \Lambda_P^{spl}$, we have the following isomorphisms,
  $$\xymatrix{
  (\fY_-^{\theta_-,\theta_0}\times \fA^{\theta_+,\theta_0}) \cup_{\fA^{\theta_-,\theta_0} \times D\times \fA^{\theta_+,\theta_0}} (\fA^{\theta_-,\theta_0} \times (\fY_+^\circ)^{\theta_+,\theta_0}) \ar[r]^-\cong \ar[d] & \fX_0^{\dagger,\theta} \ar[d] \\
  \fA^{\theta_-,\theta_0} \times \fA^{\theta_+,\theta_0} \ar[r]^-\cong & \fC_0^{\dagger,\theta}.
  }$$
\end{prop}

The next proposition describes the relationship between these stacks.

\begin{prop}[(Proposition 2.19 of \cite{LW})]
  Given $\theta\in \Lambda_P^{spl}$, there is a pair $(L_\theta, s_\theta)$, with $L_\theta$ a line bundles on $\fC^P$, and $s_\theta$ a section of $L_\theta$, such that
  \begin{enumerate}[1)]
    \setlength{\parskip}{1ex}

    \item $$\bigotimes_{\theta\in \Lambda_P^{spl}} L_\theta \cong \Ou_{\fC^P},\quad \prod_{\theta\in \Lambda_P^{spl}} s_\theta = \pi^* t,$$
        where $\pi: \fC^P \rar \A^1=\Sp \C[t]$ is the canonical map;

    \item $\fC_0^{\dagger,\theta}$ is the closed substack in $\fC^P$ defined by $(s_\theta=0)$.
  \end{enumerate}
\end{prop}

\begin{proof}
  $\fC^P$ has an \'etale covering consisting of $U_k:= [(\A^{k+1})^P / R_{\sim,\A^{k+1}}^P]$. It suffices to specify $L_\theta$ on each chart.

  Connected components of $(\A^{k+1})^P$ are indexed by various copies of $\A^{k+1}$ equipped with different continuous weight assignments $w$ whose total weight is $P$. We denote such a copy by $(\A^{k+1},w)$. Restricted to $H_i\subset \A^{k+1}$, the family $X(k)$ is of the form
  $$\cY_{-,i}\cup_{\cD_i} \cY_{+,i}:= (Y_-(i)\times \A^{k-i}) \cup_{\A^i\times D\times \A^{k-i}} (\A^i\times Y_+(k-i)^\circ).$$
  On each $\cY_{\pm,i}$, for any $s\in H_i$, the weight assignment $w$ gives a splitting datum
  $$\theta_{i,\pm}:=w_s(\cY_{\pm,i}),
  \quad \theta_{i,0}:=w_s(\cD_i),$$
  which is locally constant in $s$. They satisfy $\theta_{i,-}+ \theta_{i,+}- \theta_{i,0}= w_s(X(k))= P$, i.e. $(\theta_{i,\pm},\theta_{i,0})\in \Lambda_P^{spl}$. Moreover, one can easily check that $(\theta_{i,\pm},\theta_{i,0}) \neq (\theta_{j,\pm},\theta_{j,0})$ for $i\neq j$.

  We define the line bundles as follows.

  Given a fixed $\theta\in \Lambda_P^{spl}$, on the connected component $(\A^{k+1},w)$, if $(\theta_{i,\pm},\theta_{i,0})\neq \theta$ for any $i$, then let $L_\theta|_{(\A^{k+1},w)}:= \Ou_{\A^{k+1}}$ and $s_\theta|_{(\A^{k+1},w)}:= 1$.

  Otherwise if $(\theta_{i,\pm},\theta_{i,0})= \theta$ for some $i$, then let $L_\theta|_{(\A^{k+1},w)}:= \Ou_{\A^{k+1}}(H_i)$ and $s_\theta|_{(\A^{k+1},w)}$ be the image of $1$ under the canonical map $\Ou\rar \Ou(H_i)$.

  It's clear that this defines an $R_{\sim,\A^{k+1}}^P$-invariant line bundle on $(\A^{k+1})^P$ and thus a line bundle on $U_k$. One can check that these data actually define line bundles with sections $(L_\theta,s_\theta)$ on $\fC^P$ satisfying the properties stated in the proposition. Note that $\Ou(H_i)$ is trivial as a line bundle on $\A^{k+1}$. But written in this way, it is clear how one can glue those data together to form a global line bundle.
\end{proof}

Now we assume furthermore that $\pi: X\rar \A^1$ is a family of \emph{projective} Deligne--Mumford stacks, with $c: X\rar \uX$ the coarse moduli space, and $\E, H$ a fixed polarization. Let $P: K(X)\rar \Z$ be a group homomorphism such that the associated Hilbert polynomial $P(H^{\otimes v})$ has degree $\leq 1$. Let $\cV$ be a vector bundle of finite rank on $X$.

In previous sections we have defined the Quot-stack $\fQuot_{\fX/\fC}^{\cV,P}$ parameterizing stable quotients on the universal family $\fX\rar \fC$, with Hilbert homomorphism $P$. This is a proper Deligne--Mumford stack over $\A^1$ by results of Section 5. One has the map to the base $\fQuot_{\fX/\fC}^{\cV,P} \rar \fC^P$, defined by the usual map to the base and the weight assignments induced by the stable quotients.

Similarly, we have the proper Deligne--Mumford stack $\fQuot_{\fX_0^\dagger/\fC_0^\dagger}^{\cV,P}$. Given $\theta\in \Lambda_P^{spl}$, let
$$\fQuot_{\fX_0^\dagger/\fC_0^\dagger}^{\cV,\theta} \subset \fQuot_{\fX_0^\dagger/\fC_0^\dagger}^{\cV,P}$$
be the open and closed substack parameterizing those stable quotients whose Hilbert homomorphism splits in the type $\theta$ on the fiber decomposition.

In the relative case, given a smooth pair $(Y, D)$, we also have the proper Deligne--Mumford stack $\fQuot_{\fY/\fA}^{\cV,P}$. Given $(P,P')$ with $P' \preceq P$, one can also define the substack $\fQuot_{\fY/\fA}^{\cV, P,P'}$, parameterizing those stable quotients whose Hilbert homomorphism on each fiber is $P$ and on the distinguished divisor of each fiber is $P'$.

Now the splitting results naturally lead to a morphism
$$\xymatrix{
\fQuot_{\fY_-/\fA}^{\cV,\theta_-,\theta_0} \times_{\Quot_D^{\cV,\theta_0}} \fQuot_{\fY_+/\fA}^{\cV,\theta_+,\theta_0} \ar[r]^-{\Phi_\theta} \ar[d] &  \fQuot_{\fX_0^\dagger/\fC_0^\dagger}^{\cV,\theta} \ar[d] \\
\fA^{\theta_-,\theta_0} \times \fA^{\theta_+,\theta_0} \ar[r]^-\cong & \fC_0^{\dagger,\theta},
}$$
where we just glue two families of expanded pairs by (\ref{split_explicit}) to obtain an object in $\fX_0^\dagger$, and the stable quotients also glue together since the sheaves are admissible. One can easily see that,

\begin{prop} \label{Phi_theta}
  $\Phi_\theta$ is an isomorphism.
\end{prop}

We have the following results, whose proof is essentially the same as before.

\begin{prop}[(Theorem 5.27 of \cite{LW})] \label{L_theta}
  For $\theta\in \Lambda_P^{spl}$, let $(L_\theta,s_\theta)$ be the line bundle and section on $\fC$ defined earlier. Consider the map $F: \fQuot_{\fX/\fC}^{\cV,P} \rar \fC^P$. We have
  \begin{enumerate}[1)]
    \setlength{\parskip}{1ex}

    \item $$\bigotimes_{\theta\in \Lambda_P^{spl}} F^* L_\theta \cong \Ou_{\fQuot_{\fX/\fC}^{\cV,P}},\quad \prod_{\theta\in \Lambda_P^{spl}} F^* s_\theta = F^* \pi^* t,$$
        where $\pi: \fC^P \rar \A^1=\Sp \C[t]$ is the canonical map;

    \item $\fQuot_{\fX_0^\dagger/\fC_0^\dagger}^{\cV,\theta}$ is the closed substack in $\fQuot_{\fX/\fC}^{\cV,P}$ defined by $(F^* s_\theta=0)$.
  \end{enumerate}
\end{prop}

\subsection{Perfect obstruction theory on Hilbert stacks}

From now on we make the further assumption that $\pi:X\rar \A^1$ is of relative dimension 3, and take $\cV=\Ou_X$. Then we have the Hilbert stack
$$\cM^P:= \fQuot^{\Ou_X,P}_{\fX/\fC}.$$
Similarly we denote
$$\cM^\theta:= \fQuot^{\Ou_X,\theta}_{\fX_0^\dagger/ \fC_0^\dagger}, \quad \cN^{\theta_\pm,\theta_0}_\pm:= \fQuot^{\Ou_X,\theta_\pm,\theta_0}_{\fY_\pm/ \fA}.$$

We look for a virtual fundamental class on $\cM^P$. Let $p: \cM^P\times_{\fC^P} \fX^P\rar \cM^P$, $q: \cM^P\times_{\fC^P} \fX^P \rar \fX^P$ be the projections. Let $\cZ\subset \cM^P\times_{\fC^P} \fX^P$ be the universal family of the Hilbert stack, and $\bI\subset \Ou_{\fX^P\times_{\fC^P} \cM^P}$ be the universal ideal sheaf. In the derived category we have maps $\Id: \Ou_{\cM^P\times_{\fC^P} \fX^P} \rar R\cH om(\bI,\bI)$ and $\tr: R\cH om(\bI,\bI)\rar \Ou_{\cM^P\times_{\fC^P} \fX^P}$ which give a splitting of $R\cH om(\bI,\bI)$.

Let $L_{\cM^P\times_{\fC^P} \fX^P / \fC^P}^\bullet$ be the cotangent complex of $\cM^P\times_{\fC^P} \fX^P \rar \fC^P$, which is of Deligne--Mumford type. Consider the Atiyah class
$$\text{At}_{\cM^P\times_{\fC^P} \fX^P / \fC^P} (\bI): \bI \rar \bI\otimes L_{\cM^P\times_{\fC^P} \fX^P / \fC^P}^\bullet [1].$$
Composing with the projection and restricting to the traceless part we get
$$\text{At}_{\cM^P\times_{\fC^P} \fX^P / \fX^P} (\bI): R\cH om (\bI,\bI)_0 \rar p^* L_{\cM^P/\fC^P}^\bullet [1].$$
By Serre duality of simple normal crossing families, we have the map
$$\Phi: E^\bullet:= Rp_* (R\cH om(\bI,\bI)_0 \otimes q^* \omega_{\fX^P/\fC^P} )[2] \rar L^\bullet_{\cM^P/\fC^P}, $$
where $\omega_{\fX^P/\fC^P}$ is the relative dualizing line bundle of $\fX^P \rar \fC^P$.

\begin{thm}
  The map $\Phi: E^\bullet \rar L^\bullet_{\cM^P/\fC^P}$ is a perfect relative obstruction theory on $\cM^P$ over $\fC^P$, in the sense of \cite{BF}.
\end{thm}

\begin{proof}
  The proof is almost the same as Theorem \ref{obs_smooth}. Given a square-zero extension $T\hookrightarrow \overline T$ of schemes with ideal sheaf $J$, and a map $g: T\rar \cM^P$, one has an element $\Phi^*\omega(g)\in \Ext^1(g^* E^\bullet, J)$. To say that $\Phi$ is an obstruction theory is the same as that $\Phi^*\omega(g)$ is an obstruction to extending $g$ from $T$ to $\overline T$.

  $\Phi^*\omega(g)$ is the product of Atiyah class and Kodaira-Spencer map, which are
  $$\text{At}_{T\times_{\fC^P} \fX^P/ \fX^P} (g^*\bI): g^*E^\bullet= Rp_{T*} (R\cH om(g^*\bI,g^*\bI)_0 \otimes q^* \omega_{\fX^P/\fC^P}) [2] \rar g^* L^\bullet_{\cM^P/\fC^P} \rar L_{T/\fC^P}^\bullet,$$
  and
  $$\kappa(T/\overline T/\fC^P): L^\bullet_{T/\fC^P} \rar L^\bullet_{T/\overline T} \rar \tau^{\geq -1} L^\bullet_{T/\overline T} \cong J[1],$$
  where $p_T$ is the projection $T\times_{\fC^P} \fX^P \rar T$. In other words,
  $$\Phi^*\omega(g) = p^*_T \kappa(T/\overline T/\fC^P) \circ \text{At}_{T\times_{\fC^P} \fX^P/ \fX^P} (g^*\bI),$$
  in the group
  $$\Ext^{-1}_T(R {p_T}_* (R\cH om(g^*\bI, g^*\bI)_0 \otimes q^* \omega_{\fX^P/\fC^P}), J) \cong \Ext^2_{T\times_{\fC^P} \fX^P} (g^*\bI, g^*\bI \otimes p_T^*J)_0.$$
  Again we may assume that $T$ is affine, and Proposition 3.1.8 of \cite{Il} says that $\Phi^*\omega(g)$ is the obstruction to extending $g^*\bI$ from $T$ to $\overline T$ with $\det \bI$ fixed.

  By construction, a map $g: T\rar \cM^P$ can represented by a family of expanded degenerations $\cX_T\rar T$, with a closed substack $\cZ_T\subset \cX_T$, where $\cZ_T$ is the pull back of the universal family $\cZ\subset \cM^P\times_{\fC^P} \fX^P$. But since $\cZ$ and $\cZ_T$ are admissible, one can see that $g^*\bI$ is just the ideal sheaf of $\cZ_T\subset \cX_T$. Hence $\Phi^*\omega(g)$ is also the obstruction to extending $\cZ_T$, or equivalently $g$, from $T$ to $\overline T$, which proves that $\Phi$ is an obstruction theory.

  The proof that $\Phi$ is perfect is also the same as in Theorem \ref{obs_smooth}, with the following lemma in place of Lemma \ref{Ext12}.
\end{proof}

\begin{lem}
  Let $X\rar \A^1$ be a simple degeneration of relative dimension 3. Consider $X_0[k]$ for some $k\geq 0$ and let $Z\subset X_0[k]$ be a 1-dimensional admissible closed substack, with ideal sheaf $\I$. Then $\Ext^i_{X_0[k]}(\I,\I)_0=0$, for all $i\neq 1,2$.
\end{lem}

\begin{proof}
  We just prove for $k=0$; the proof for general $k$ is exactly the same. Let $X_0=Y_-\cup_D Y_+$. Since $Z$ is admissible, the ideal sheaf $\I$ is also admissible, which fits in the exact sequence
  $$\xymatrix{
  0\ar[r] & \I \ar[r] & \I|_{Y_-}\oplus \I|_{Y_+} \ar[r] & \I|_D \ar[r] & 0.
  }$$
  Applying $\Hom(\I,-)$, we have
  $$\xymatrix{
  0\ar[r] & \Hom(\I,\I) \ar[r] & \Hom(\I|_{Y_-},\I|_{Y_-}) \oplus \Hom(\I|_{Y_+},\I|_{Y_+}) \ar[r] & \Hom(\I|_D,\I|_D).
  }$$
  Then $\Hom(\I,\I)_0=0$ follows from $\Hom(\I|_{Y_\pm},\I|_{Y_\pm})_0=0$, as the sequence respects the trace map. Same arguments applied to the sequence tensored with the dualizing sheaf lead to the vanishing of $\Ext^3(\I,\I)_0$.
\end{proof}

\begin{cor}
  There is a virtual fundamental class $[\cM^P]^{\vir} \in A_*(\cM^P)$.
\end{cor}

In the same way one can prove that there are perfect obstruction theories on $\cN_\pm^{\theta_\pm,\theta_0}\rar \fA^{\theta_\pm,\theta_0}$, for a given $\theta\in \Lambda_P^{spl}$. Again let $p$, $q$ be the projections of $\cN_\pm^{\theta_\pm,\theta_0} \times_{\fA^{\theta_\pm,\theta_0}} \fY_\pm^{\theta_\pm,\theta_0}$ to its two factors. The obstruction theory is given by
$$\Phi_\pm: E_\pm^\bullet:= Rp_* \left(R\cH om(\bI_\pm,\bI_\pm)_0 \otimes q^* \omega_{\fY_\pm^{\theta_\pm,\theta_0} / \fA^{\theta_\pm,\theta_0}}\right) [2] \rar L_{\cN_\pm^{\theta_\pm,\theta_0} / \fA^{\theta_\pm,\theta_0}}^\bullet,$$
where $\bI_\pm$ is the corresponding universal ideal sheaf. We have the virtual fundamental classes $[\cN_\pm^{\theta_\pm,\theta_0}]^{\vir} \in A_*(\cN_\pm^{\theta_\pm,\theta_0})$.

Also on $\cM^\theta\rar \fC^{\dagger,\theta}_0$, the Cartesian diagram
$$\xymatrix{
\cM^\theta \ar[r] \ar[d] & \cM^P \ar[d] \\
\fC_0^{\dagger,\theta} \ar[r]^{\iota_\theta} & \fC^P
}$$
and Proposition 7.2 of \cite{BF} implies that the restriction of everything from $\cM^P\rar \fC^P$ to the $\theta$-piece is a perfect obstruction theory, and $[\cM^\theta]^{\vir}= \iota_\theta^! [\cM^P]^{\vir}$.

\subsection{Degeneration formula -- cycle version}

Now everything is ready for a cycle-version degeneration formula.

By Proposition \ref{Phi_theta} we have the following Cartesian diagram
\begin{equation} \label{Delta_diag}
\xymatrix{
\cM^\theta & \cN_-^{\theta_-,\theta_0} \times_{\Hilb_D^{\theta_0}} \cN_+^{\theta_+,\theta_0} \ar[l]^-{\Phi_\theta}_-\cong \ar[r] \ar[d]^g & \cN_-^{\theta_-,\theta_0} \times \cN_+^{\theta_+,\theta_0} \ar[d] \\
& \Hilb_D^{\theta_0} \times \fC_0^{\dagger,\theta} \ar[r] \ar[d] & \Hilb_D^{\theta_0} \times \Hilb_D^{\theta_0} \times \fC_0^{\dagger,\theta} \ar[d] \\
 & \Hilb_D^{\theta_0} \ar[r]^-\Delta & \Hilb_D^{\theta_0} \times \Hilb_D^{\theta_0},
 }
\end{equation}
with vertical arrows in the upper row given by the natural forgetful maps $\cN_\pm^{\theta_\pm,\theta_0}\rar \Hilb_D^\theta$ and $\cN_\pm^{\theta_\pm,\theta_0}\rar \fA^{\theta_\pm,\theta_0}$.

For a family $\pi: X\rar \A^1$ and a point $c\in \A^1$, let $i_c: \{c\} \hookrightarrow \A^1$ be the inclusion. For $c\neq 0$, the fiber $X_c$ is smooth. The restriction to $X_c$ gives a perfect obstruction theory, with virtual fundamental class $[\Hilb_{X_c}^P]^{\vir}$. One has the following two Cartesian diagrams,
$$\xymatrix{
\Hilb_{X_c}^P \ar[r] \ar[d] & \cM^P \ar[d] & \cM^\theta \ar@{^(->}[r]^-{\iota_\theta} & \fQuot^{\Ou_X,P}_{\fX_0/\fC_0} \ar@{^(->}[r] \ar[d] & \cM^P \ar[d] \\
\{c\} \ar@{^(->}[r] & \A^1, & &  \{0\} \ar@{^(->}[r] & \A^1,
}$$
where the map $\cM^P\rar \A^1$ is the composition $\cM^P\rar \fC^P\rar \A^1$.

\begin{thm}[(Degeneration formula -- cycle version)]
  \begin{equation} \label{cycle_c}
  i_c^! [\cM^P]^{\vir} = [\Hilb_{X_c}^P]^{\vir},
  \end{equation}
  \begin{equation} \label{cycle_0}
  i_0^! [\cM^P]^{\vir} = \sum_{\theta\in \Lambda_P^{spl}} \iota_{\theta*} \Delta^! \left( [\cN_-^{\theta_-,\theta_0}]^{\vir} \times [\cN_+^{\theta_+,\theta_0}]^{\vir} \right),
  \end{equation}
  where the classes in the second row are viewed in $0\times_{\A^1} \cM^P$.
\end{thm}

\begin{proof}
Let $\cZ^\theta \subset \cM^\theta\times_{\fC_0^{\dagger,\theta}} \fX_0^{\dagger,\theta}$, $\cZ_\pm\subset \cN_\pm^{\theta_\pm,\theta_0}\times_{\fA^{\theta_\pm,\theta_0}} \fY_\pm^{\theta_\pm,\theta_0}$, and $\cZ_D\subset \Hilb_D^{\theta_0}\times D$ be the universal families of the classifying stacks. Let $\bI^\theta$, $\bI_\pm$ and $\bI_D$ be the corresponding ideal sheaves. By admissibility we have the gluing
$$\cZ^\theta \cong (\cZ_-\times \fA^{\theta_+,\theta_0}) \cup_{\fA^{\theta_-,\theta_0}\times D\times \fA^{\theta_+,\theta_0}} (\fA^{\theta_-,\theta_0}\times \cZ_+),$$
and the exact sequence
$$\xymatrix{
0\ar[r] & \bI^\theta \ar[r] & \bI_-\boxplus \bI_+ \ar[r] & j_{D*}\ev^*_\theta \bI_D\ar[r] & 0.
}$$
in the total spaces
$$\cM^\theta\times_{\fC_0^{\dagger,\theta}} \fX_0^{\dagger,\theta} \cong \cM^{\theta}\times_{\fC_0^{\dagger,\theta}} \left( (\fY_-^{\theta_-,\theta_0}\times \fA^{\theta_+,\theta_0}) \cup_{\fA^{\theta_-,\theta_0} \times D\times \fA^{\theta_+,\theta_0}} (\fA^{\theta_-,\theta_0} \times (\fY_+^\circ)^{\theta_+,\theta_0}) \right).$$
Here $\bI_-\boxplus \bI_+$ means $\Delta^*(p_-^*\bI_-\oplus p_+^*\bI_+)$, where $p_\pm$ stands for (base change to the total space of) the projection $\cN_-^{\theta_-,\theta_0} \times \cN_+^{\theta_+,\theta_0} \rar \cN_\pm^{\theta_\pm,\theta_0}$. $\ev_\theta: \cM^\theta \rar \Hilb^{\theta_0}_D$ is the evaluation map and also denotes its base change to the total space. $j_D$ is the inclusion of the universal distinguished divisor $\cM^\theta \times D \subset \cM^\theta\times_{\fC_0^{\dagger,\theta}} \fX_0^{\dagger,\theta}$.

Now $\Delta^*p_-^*\bI_-$ is the universal ideal sheaf on $\cM^\theta \times_{\fC_0^{\dagger,\theta}} (\fY_-^{\theta_-,\theta_0}\times \fA^{\theta_+,\theta_0})$ and similar for $\Delta^*p_+^*\bI_+$; $\Delta^*(\ev_-\times \ev_+)^*\bI_D = \ev_\theta^*\bI_D$ is the ideal sheaf of $\cM^\theta\times D$.

Let $p$ denote the projections from the universal families to classifying stacks $\cM^\theta \times_{\fC_0^{\dagger,\theta}} \fX_0^{\dagger,\theta} \rar \cM^\theta$, $\cN_\pm^{\theta_\pm,\theta_0}\times_{\fA^{\theta_\pm,\theta_0}} \fY_\pm^{\theta_\pm,\theta_0} \rar \cN_\pm^{\theta_\pm,\theta_0}$ and $\Hilb_D^{\theta_0}\times D \rar \Hilb_D^{\theta_0}$. Then applying $Rp_* R\cH om (-,\bI^\theta)$, we get the following diagram of distinguished triangles
$$\xymatrix{
Rp_*R\cH om(\bI_D,\bI_D)_0^\vee \ar[r]\ar[d] & Rp_*R\cH om(\bI_-,\bI_-)_0^\vee \boxplus Rp_*R\cH om(\bI_+,\bI_+)_0^\vee \ar[r]\ar[d] & R\pi_*R\cH om(\bI^\theta,\bI^\theta)_0^\vee \ar[d] \\
L^\bullet_{\cM^\theta/\cN_-^{\theta_-,\theta_0} \times \cN_+^{\theta_+,\theta_0}} [-1] \ar[r] & L^\bullet_{\cN_-^{\theta_-,\theta_0}/\fA^{\theta_-,\theta_0}} \boxplus L^\bullet_{\cN_+^{\theta_+,\theta_0}/\fA^{\theta_+,\theta_0}} \ar[r] & L^\bullet_{\cM^\theta/\fC_0^{\dagger,\theta}},
}$$
where the upper row is obtained from the exact sequence of ideal sheaves, and the lower row is the distinguished triangle for the map $\cM^\theta \rar \cN_-^{\theta_-,\theta_0} \times \cN_+^{\theta_+,\theta_0}$, relative to $\fC_0^{\dagger,\theta}$.

It is easy to check that $\Ext^2(I,I)_0=0$ for an ideal sheaf $I$ on a 2-dimensional smooth Deligne--Mumford stack. Thus we have
$$Rp_*R\cH om(\bI_D,\bI_D)_0^\vee \cong \Omega_{\Hilb_D^{\theta_0}} \cong L^\bullet_{\Hilb_D^{\theta_0}/\Hilb_D^{\theta_0}\times \Hilb_D^{\theta_0}} [-1],$$
and the first column in the diagram is the same as the canonical map
$$g^* L^\bullet_{\Hilb_D^{\theta_0}/\Hilb_D^{\theta_0}\times \Hilb_D^{\theta_0}} \rar L^\bullet_{\cM^\theta/\cN_-^{\theta_-,\theta_0} \times \cN_+^{\theta_+,\theta_0}},$$
from Diagram (\ref{Delta_diag}).

In other words, the perfect obstruction theories on $\cM^\theta$ and $\cN_-^{\theta_-,\theta_0}\times \cN_+^{\theta_+,\theta_-}$ form a compatibility datum in the sense of \cite{BF}, relative to $\fC_0^{\dagger,\theta}$. As a result,
$$[\cM^\theta]^{\vir} = \Delta^! ([\cN_-^{\theta_-,\theta_0}]^{\vir} \times [\cN_+^{\theta_+,\theta_-}]^{\vir}).$$

Now the conclusion follows from the following splitting result by Proposition \ref{L_theta},
$$i_0^! [\cM^P]^{\vir} = \sum_{\theta\in \Lambda_P^{spl}} \iota_{\theta*} [\cM^\theta]^{\vir}.$$
\end{proof}

\section{Degeneration formula -- numerical version}

\subsection{Relative orbifold Donaldson--Thomas theory}

We defined the orbifold Donaldson--Thomas invariant for a 3-dimensional smooth projective Deligne--Mumford stack in Section \ref{DT}. Now we can define the relative Donaldson--Thomas invariant for a smooth pair $(Y,D)$, where $Y$ is a 3-dimensional smooth projective Deligne--Mumford stack, and $D\subset Y$ is a smooth divisor. Fix a polarization $(\E,H)$ of $Y$ and $P$ as before.

Consider $K(Y):=K(Y)_\Q$. The pairing $\chi: K(Y)\times K(Y)\rar \Z$ is nondegenerate because of the projectivity. We can identify $\Hom(K(Y),\Q)$ with $K(Y)$ and view $P\in K(Y)$. The dimension condition is that $P\in F_1K(Y)$, where $F_\bullet$ is the natural topological filtration. For those$P$ represented by addmissible sheaves, $P_0=i^! P$ lies in $F_0 K(D)$ by admissibility, where $i:D\hookrightarrow X$ is the inclusion.

Consider the Hilbert stack $\cN^{P,P_0}:=\fHilb^{P,P_0}_{\fY/\fA} \rar \fA^{P,P_0}$ parameterizing the stable quotients on the classifying stacks of expanded pairs. As in Section \ref{Deg_cycle}, we have a perfect obstruction theory and thus a virtual fundamental class $[\cN^{P,P_0}]^{\vir} \in A_*(\cN^{P,P_0})$.

Let $p,q$ be the projection of $\cN^{P,P_0}\times_{\fA^{P,P_0}} \fY^{P,P_0}$ to its two factors and $\bI$ be the universal ideal sheaf. Given $\gamma\in A^l_\orb(Y)$, we still have the operator $\ch^\orb(\gamma): A_*(\cN^{P,P_0})\rar A_{*-k+1-l}(\cN^{P,P_0})$ defined via the following diagram,
$$\xymatrix{
\cN\times_\fA (I_\fA \fY) \ar[r]\ar[d] & I_\fA \fY \ar[r]\ar[d] & IY \ar[d] \\
\cN\times_\fA \fY \ar[r]^-q\ar[d]_p & \fY \ar[r]\ar[d] & Y \\
\cN \ar[r] & \fA, &
}$$
where for simplicity we have omitted the superscripts $P$, $P_0$, and $I_\fA \fY\rar \fY$ is the inertia stack of $\fY$ over $\fA$. Let $\ev: \cN^{P,P_0}\rar \Hilb^{P_0}(D)$ be the evaluation map.

\begin{defn}
  Given $\gamma_i\in A^*_\orb(Y)$, $1\leq i\leq r$, and $C\in A^*(\Hilb^{P_0}(D))$, define the relative Donaldson--Thomas invariant as
  $$\left\langle \prod_{i=1}^r \tau_{k_i}(\gamma_i) \ \middle| \ C \right\rangle_{X,D}^P := \deg \left[ \ev^*(C) \cdot \prod_{i=1}^r \ch_{k_i+2}^\orb(\gamma_i) \cdot [\cN^{P,P_0}]^{\vir} \right] _0$$
\end{defn}

\subsection{Degeneration formula}

Now let's consider the case of a simple degeneration. Let $\pi: X\rar \A^1$ be a family of smooth projective Deligne--Mumford stacks, which is a simple degeneration of relative dimension 3, with central fiber $X_0=Y_-\cup_D Y_+$. Take $0\neq c\in \A^1$ and $X_c=\pi^{-1}(c)$. Fix $P\in F_1 K(X_c)$. Let $P_0= i^!P \in F_0 K(D)$.

Let $\gamma\in A^*_\orb(X_c)$ be in the image of the restriction from $A^*_\orb(X)$, and $\gamma_\pm$, $\gamma_0$ be its restrictions to $Y_\pm$, $D$ respectively. We abuse these notations to also denote their pushforwards to $X$; therefore $\gamma = \gamma_- + \gamma_+ - \gamma_0$. Let $\{C_k\}$ be a basis of $A^*(\Hilb^{P_0}(D))$, with cup product
$$\int_{\Hilb^{P_0}(D)} C_k\cup C_l = g_{kl}.$$
Let $(g^{kl})$ be the inverse matrix.

We have the numerical version of the degeneration formula in the following.
\begin{thm}[(Degeneration formula -- numerical version)] \label{thm_deg_P}
  Given $P\in F_1 K(X_c)$, assume that $\gamma_{i,\pm}$ are disjoint with $D$. We have
  $$\left\langle \prod_{i=1}^r \tau_{k_i}(\gamma_i) \right\rangle_{X_c}^P =
  \sum_{\substack{\theta_-+\theta_+ - P_0 = P, \\
                S\subset \{1,\cdots,r\}, k,l}}
  \left\langle \prod_{i\in S} \tau_{k_i}(\gamma_{i, -}) \middle| C_k \right\rangle_{Y_-,D}^{\theta_-} g^{kl} \left\langle \prod_{i\not\in S} \tau_{k_i}(\gamma_{i, +}) \middle| C_l \right\rangle_{Y_+,D}^{\theta_+}, $$
  where $\theta_\pm \in F_1 K(Y_\pm)$ range over all configurations that satisfy $\theta_-+\theta_+ - P_0 = P$.
\end{thm}

\begin{proof}[Proof of Theorem \ref{thm_deg_P}]
  Use the cycle-version degeneration formula. Apply $\ch_{k_i+2}^\orb(\gamma_i)$ to (\ref{cycle_c}) and take the degree 0 part, one gets the LHS of the formula. For the RHS, we apply $\ch_{k_i+2}^\orb(\gamma_i)$ to (\ref{cycle_0}). Let $p,q$ be the projections from $\cM^\theta\times_{\fC^{\dagger,\theta}_0} \fX^{\dagger,\theta}_0$ to the two factors. Let $\cZ^\theta$ be the universal family and $\bI^\theta$ be the universal ideal sheaf. Consider the embedding $\cM^\theta \hookrightarrow \cN^{\theta_-,\theta_0}_- \times \cN^{\theta_+,\theta_0}_+$.

  Recall the sequence on $\cM^\theta \times_{\fC^{\dagger,\theta}_0} \fX^{\dagger,\theta}_0$,
  $$\xymatrix{
  0\ar[r] & \bI^\theta \ar[r] & \bI_-\boxplus \bI_+ \ar[r] & j_{D*}\ev^*_\theta \bI_D\ar[r] & 0,
  }$$
  where $j_D$ is the embedding $\cM^\theta\times D\subset \cM^\theta \times_{\fC^{\dagger,\theta}_0} \fX^{\dagger,\theta}_0$, and $\ev_\theta: \cM^\theta \rar \Hilb^{\theta_0}(D)$. Hence
  \begin{eqnarray*}
  \ch_{k+2}^\orb(\gamma)\cdot i_0^![\cM^P]^{\vir}
  &=& \sum_{\theta\in \Lambda^{spl}_P} p_* \left(q^*\gamma \cap \ch_{k+2}^\orb(\bI^\theta)\cdot p^*\iota_{\theta*} \Delta^! \left( [\cN_-^{\theta_-,\theta_0}]^{\vir} \times [\cN_+^{\theta_+,\theta_0}]^{\vir} \right) \right)\\
  &=& \sum_{\theta\in \Lambda^{spl}_P} p_* \iota_{\theta*} \Delta^! \left( \left( q^*\gamma_- \cap \ch_{k+2}^\orb(\bI_-) \cdot p^* [\cN_-^{\theta_-,\theta_0}]^{\vir} \right) \times p^* [\cN_+^{\theta_+,\theta_0}]^{\vir} \right)  \\
  && + \sum_{\theta\in \Lambda^{spl}_P} p_* \iota_{\theta*} \Delta^! \left( p^* [\cN_-^{\theta_-,\theta_0}]^{\vir} \times \left( q^* \gamma_+ \cap \ch_{k+2}^\orb(\bI_+) \cdot p^* [\cN_+^{\theta_+,\theta_0}]^{\vir} \right) \right) \\
  && - \sum_{\theta\in \Lambda^{spl}_P} p_* \left( q^*\gamma \cap \ch_{k+2}^\orb (j_{D*} \ev^* \bI_D) \cdot p^* [\cM^\theta]^{\vir} \right).
  \end{eqnarray*}
 Note that since $\gamma$ is disjoint from $D$, the last term actually vanishes.

  Therefore we have proved the identity
  \begin{eqnarray*}
  \ch_{k+2}^\orb(\gamma)\cdot i_0^![\cM^P]^{\vir}
  &=& \sum_{\theta\in \Lambda^{spl}_P} p_* \left( q^*\gamma \cap p^*\iota_{\theta*} \Delta^! \left( \left( \ch_{k+2}^\orb(\bI_-) \cdot [\cN_-^{\theta_-,\theta_0}]^{\vir} \right) \times [\cN_+^{\theta_+,\theta_0}]^{\vir} \right) \right) \\
  && + \sum_{\theta\in \Lambda^{spl}_P} p_* \left( q^*\gamma \cap p^*\iota_{\theta*} \Delta^! \left( [\cN_-^{\theta_-,\theta_0}]^{\vir} \times \left( \ch_{k+2}^\orb(\bI_+) \cdot [\cN_+^{\theta_+,\theta_0}]^{\vir} \right) \right) \right).
  \end{eqnarray*}
  Take $\theta_0=P_0$ and the data $\theta\in \Lambda_P^{spl}$ can be identified with $(\theta_-, \theta_+, P_0)$ satisfying the condition in the assumption.

  For a basis $\{C_k\}$ a $A^*(\Hilb^m(D))$, we have Kunneth decomposition of the diagonal
  $$[\Delta]= \sum_{k,l} g^{kl} C_k \otimes C_l.$$
  Apply this to the equality and the degeneration formula follows.
\end{proof}

We are particularly interested in a special type of curve classes. Let $(Y,D)$ be a smooth pair.

\begin{defn}
A class $P\in K(Y)$ is called \emph{multi-regular}, if it can be represented by some coherent sheaf, such that the associated representation of the stabilizer group at the generic point is a multiple of the regular representation.
\end{defn}

Denote by $F_1^{\mr}K(Y)\subset F_1K(Y)$ the subgroup generated by multi-regular classes. Let $(\beta,\varepsilon)\in F_1^{\mr}K(Y)/F_0 K(Y)\oplus F_0 K(Y)$ be the image of $P$ in the associated graded K-group. One can check that $P$ is multi-regular if and only if $\beta$ is a pull-back from a curve class in the coarse moduli space.

Let $F_0^{\mr}K(D)$ be the subgroup generated by 0-dimensional substacks whose associated representations are multi-regular. Then $F_0^{\mr} K(D)\cong F_0 K(\uD)\cong \Z$, where $\uD$ is the coarse moduli space. Let $P\in F_1^{\mr} K(Y)$ be represented by some admissible curve. Then $P_0=i^*P\in F_0^{\mr} K(D)$ only depends on $\beta$. Let $m$ be the number such that $\beta\cdot D=m [\Ou_x]$, where $x\in D$ is the preimage of a point in $\uD$.

Now for simple degeneration $\pi: X\rar \A^1$. Given classes $(\beta_1,\varepsilon_1)\in F_1^{\mr} K(Y_-)/F_0 K(Y_-) \oplus F_0 K(Y_-)$, $(\beta_2,\varepsilon_2)\in F_1^{\mr} K(Y_+)/F_0 K(Y_+) \oplus F_0 K(Y_+)$, they come from a splitting data if $\beta_1+\beta_2=\beta$, $\varepsilon_1+\varepsilon_2-m=\varepsilon$.

\begin{thm}[(Degeneration formula -- numerical version for multi-regular case)] \label{thm_deg}
  Given $\beta\in F_1^{\mr}K(X_c)/F_0 K(X_c)$, assume that $\gamma_{i,\pm}$ are disjoint with $D$. We have
  $$\left\langle \prod_{i=1}^r \tau_{k_i}(\gamma_i) \right\rangle_{X_c}^{\beta,\varepsilon} =
  \sum_{\substack{\beta_-+\beta_+=\beta, \\
                \varepsilon_-+\varepsilon_+= \varepsilon+m, \\
                S\subset \{1,\cdots,r\}, k,l}}
  \left\langle \prod_{i\in S} \tau_{k_i}(\gamma_{i,-}) \middle| C_k \right\rangle_{Y_-,D}^{\beta_-,\varepsilon_-} g^{kl} \left\langle \prod_{i\not\in S} \tau_{k_i}(\gamma_{i,+}) \middle| C_l \right\rangle_{Y_+,D}^{\beta_+,\varepsilon_+}, $$
  where $\beta_-\in F_1^{\mr}K(Y_-)/F_0 K(Y_-)$, $\beta_+\in F_1^{\mr}K(Y_+)/F_0 K(Y_+)$ range over all curve classes that coincide on $D$ and satisfy $\beta_-+\beta_+=\beta$.
\end{thm}

Define the descendent Donaldson-Thomas partition function of $X_c$ as
$$Z_\beta \left( X_c;q \ \middle| \ \prod_{i=1}^r \tau_{k_i}(\gamma_i) \right) := \sum_{\varepsilon\in F_0 K(X_c)} \left< \prod_{i=1}^r \tau_{k_i}(\gamma_i) \right>_{X_c}^{\beta,\varepsilon} q^\varepsilon.$$
Similarly for a pair $(Y,D)$, and $C\in A^*(\Hilb^{(\beta\cdot D)[\Ou_x]}(D))$, define the relative Donaldson-Thomas partition function as
$$Z_{\beta,C} \left( Y,D;q \ \middle| \ \prod_{i=1}^r \tau_{k_i}(\gamma_i) \right) := \sum_{\varepsilon\in F_0 K(Y)} \left< \prod_{i=1}^r \tau_{k_i}(\gamma_i) \ \middle| \ C \right>_{Y,D}^{\beta,\varepsilon} q^\varepsilon.$$

\begin{cor}
Given $\beta\in F_1^{\mr}K(X_c)/F_0 K(X_c)$, assume that $\gamma_{i,\pm}$ are disjoint with $D$. Then,
  \begin{eqnarray*}
  Z_\beta \left( X_c;q \ \middle| \ \prod_{i=1}^r \tau_{k_i}(\gamma_i) \right) &=& \sum_{\substack{\beta_- + \beta_+ =\beta \\
                            S\subset \{1,\cdots, r\},k,l }}
  \frac{g^{kl}}{q^m} Z_{\beta_-,C_k} \left( Y_-,D;q \ \middle| \ \prod_{i\in S} \tau_{k_i}(\gamma_{i,-}) \right) \\
  && \cdot Z_{\beta_+,C_l} \left( Y_+,D;q \ \middle| \ \prod_{i\not\in S} \tau_{k_i}(\gamma_{i,+}) \right).
  \end{eqnarray*}
\end{cor}

\begin{rem}
  In practice, one has to find a good basis $\{C_k\}$ for the cohomology of the Hilbert scheme of an orbifold surface. In the orbifold case, depending on the specific problem, there are usually natural choices of such choices. For example, when a torus action is involved, to work in the equivariant setting and take the fixed point basis is one such choice.

  Another important case is that the orbifold surface $D$ is of ADE type, i.e. $[\C^2/\Gamma]$ where $\Gamma$ is a finite subgroup of $SL(2,\C)$. Then $\Hilb(D)$ has a structure of Nakajima's quiver variety which provides natural basis of its cohomology.
\end{rem}

\bibliographystyle{amsalpha}
\bibliography{reference}

\end{document}